\documentclass[a4paper,12pt]{amsart}
\usepackage{amsmath,amsfonts,amssymb,amsthm,graphicx}

\usepackage[totalwidth=15.75cm,totalheight=22.275cm]{geometry}

\usepackage{color}

\usepackage[shortlabels]{enumitem}

\makeatletter
\@namedef{subjclassname@2020}{%
  \textup{2020} Mathematics Subject Classification}
\makeatother

\usepackage{subfig}
\numberwithin{equation}{section}

\theoremstyle{plain}
\newtheorem{thm}{Theorem}[section]%
\newtheorem{lem}[thm]{Lemma}%
\newtheorem{cor}[thm]{Corollary}%
\newtheorem{prop}[thm]{Proposition}%
\newtheorem{deflem}[thm]{Definition and Lemma}%

\newtheorem{cordef}[thm]{Corollary and Definition}

\newtheorem*{douadyhubbardlandingtheorem}{Douady-Hubbard landing theorem}

\theoremstyle{definition}
\newtheorem{defn}[thm]{Definition}%
\newtheorem{obs}[thm]{Observation}%
\newtheorem{rmk}[thm]{Remark}%
 \newtheoremstyle{claimstyle}%
   {}
   {}
   {\normalfont}
   {}
   {\itshape}
   {.}
   { }
   {\thmnote{#3}}

\theoremstyle{claimstyle}
\newtheorem*{varclaim}{}

\newenvironment{claim}[1][Claim]{\begin{varclaim}[#1]}{\end{varclaim}}
\newenvironment{remark}[1][Remark]{\begin{varclaim}[#1]}{\end{varclaim}}

\newenvironment{subproof}{\begin{proof}}{%
               \end{proof}}

\usepackage{accents}
\newcommand{\unbdd}[1]{\accentset{\infty}{#1}}

\newcommand*{\defeq}{\mathrel{\vcenter{\baselineskip0.5ex \lineskiplimit0pt
                     \hbox{\scriptsize.}\hbox{\scriptsize.}}}%
                     =}
\newcommand{\eqdef}{=\mathrel{\vcenter{\baselineskip0.5ex \lineskiplimit0pt
                     \hbox{\scriptsize.}\hbox{\scriptsize.}}}}

\newcommand{\cyl}{\operatorname{cyl}}

\newcommand{\WW}{W_0}

\newcommand{\tail}{\tau}  
\newcommand{\tailtilde}{\vartheta} 

\newcommand{\closure}{\operatorname{cl}}

\newcommand{\Deriv}{{\rm D}}
\newcommand{\deriv}{{\rm d}}

\newcommand{\distW}{\dist_{\cyl}^{\WW}}
\newcommand{\diamW}{\diam_{\cyl}^{\WW}}

\newcommand{\PC}{\mathcal{P}}
\newcommand{\PP}{\PC}
\renewcommand{\P}{\PC}

\newcommand{\B}{\mathcal{B}}

\newcommand{\FF}{\mathcal{F}}
\renewcommand{\AA}{\mathcal{A}}
\newcommand{\BB}{\mathcal{B}}

\newcommand{\s}{{\underline{s}}}

\newcommand{\ov}{\overline}
\newcommand{\ra}{\rightarrow}

\newcommand{\N}{\mathbb{N}}

\newcommand{\Z}{\mathbb{Z}}

\newcommand{\R}{\mathbb{R}}
\newcommand{\C}{\ensuremath{\mathbb{C}}}
\newcommand{\Ch}{\hat{\mathbb{C}}}

\newcommand{\D}{\mathbb{D}}

\newcommand{\eps}{\ensuremath{\varepsilon}}
\renewcommand{\theta}{\vartheta}
\renewcommand{\phi}{\varphi}

\newcommand{\re}{\operatorname{Re}}

\newcommand{\dist}{\operatorname{dist}}
\newcommand{\diam}{\operatorname{diam}}

\usepackage{hyperref}

\title[A landing theorem for entire functions]{A landing theorem for entire functions \\ with bounded post-singular sets}
\author{Anna Miriam Benini}
\address{Dipartimento di Scienze Matematiche Fisiche e Informatiche, Universit\`a di Parma, Parma,  IT}
\email{ambenini@gmail.com}
\author{Lasse Rempe}
\address{Dept.\ of Mathematical Sciences, University of Liverpool, Liverpool L69 7ZL, UK}
\email{l.rempe@liverpool.ac.uk}
\date{\today}
\thanks{The first author  was supported by the European Union’s Horizon 2020 research and innovation programme under the Marie Sk{\l}odowska-Curie Grant Agreement No.~703269 COTRADY and by the SIR grant NEWHOLITE no.~RBSI14CFME. The second author was partially supported by a Philip Leverhulme Prize.}

\begin{document}

\begin{abstract}
  The \emph{Douady-Hubbard landing theorem} for periodic external 
    rays is one of the cornerstones of the study
   of polynomial dynamics. It states that, for a complex polynomial $f$ with bounded postcritical set, every 
   periodic external ray lands at a repelling or parabolic periodic point, and conversely every repelling or parabolic
    point is the landing point of at least one periodic external ray. 

We prove an analogue of this theorem for an entire function $f$ with bounded postsingular set. If $f$ has finite order of growth, then it is known that
  the escaping set $I(f)$ contains certain curves called \emph{periodic hairs}; we show that 
  every periodic hair lands at a repelling or parabolic periodic point, and conversely every repelling or parabolic periodic point is the landing point of at least one periodic hair.
  For a postsingularly bounded entire function $f$ of infinite order, such hairs may not exist. Therefore we introduce 
  certain dynamically natural connected subsets of $I(f)$, called \emph{filaments}. We show that 
 every periodic filament lands at a repelling or parabolic periodic point, and conversely every repelling or parabolic periodic point is the landing point of at least one periodic filament. 

   More generally, we prove that every point of a hyperbolic set is the landing point of a filament.
\end{abstract}

\maketitle


\section*{A note on terminology by the second author}
 In previous versions of this article, including that published in 2020 in \emph{Geometric and Functional Analysis}, 
  the objects we now call ``filaments'' were called ``dreadlocks''~-- a term that I was using informally to refer to generalisations of 
   ``hairs'' for some time before collaboration on the present paper began. I now regret this 
    terminology. Dreadlocks are a hairstyle particularly associated
  with afro-textured hair. I was ignorant of the shameful history of discrimination 
   against those whose hair does not conform to Eurocentric beauty standards; see
   ``How the fight for natural black hair became a civil rights issue''
     (Keli Godd for The Guardian, April 2021).
   While I find the mathematical objects in question beautiful and natural, some may consider them 
   exotic or pathological. This makes the term ``dreadlocks''
   particularly inappropriate; it has been replaced
    with the agreement of my co-author. I encourage all mathematicians to be 
   mindful of terminology that may be offensive or unwelcoming to colleagues from underrepresented groups.

\section{Introduction}
  Let $p\colon\C\to\C$ be a polynomial. The \emph{filled-in Julia set} $K(p)$ consists of those
    points $z\in\C$ whose orbits remain  bounded under repeated application of $p$. 
    In their study of the dynamics of complex polynomials and the Mandelbrot set \cite{DH84},
    Douady and Hubbard introduced the notion of \emph{external rays}, which can be characterised as the
     gradient lines of the Green's function on the basin  of attraction  of infinity, $\C\setminus K(p)$.  
     \emph{Periodic}  (and pre-periodic) rays are of particular importance, due to the following result.
 
  \begin{douadyhubbardlandingtheorem}
    Let $p$ be a polynomial whose \emph{post-critical set}
      \begin{equation}\label{eqn:postcritical} \PC(p) \defeq \overline{\bigcup_{c\colon  p'(c)=0} \{p^n(c)\colon  n\geq 1\}} \end{equation}
     is bounded. (Equivalently, assume that $K(p)$ is connected.) 

   Then every periodic ray of $p$ lands at a repelling  or parabolic periodic point, and conversely every
     repelling or parabolic periodic point of $p$ is the landing point of at least one and at most finitely many
     periodic external rays.  
  \end{douadyhubbardlandingtheorem}

   The first half of this theorem, concerning the landing of periodic rays, can be found in
     \cite[Expos\'e~VIII.II,  Proposition 2]{DH84}. The second half,
     which is more difficult, is due to  Douady; the first published proofs are in \cite{ELv,Hu}. 
    Ever since, the Douady-Hubbard theorem has been  a cornerstone of  the   study of polynomial dynamics. In particular, 
    it forms the basis of the ``puzzle techniques'' that were pioneered by Yoccoz, Branner and Hubbard,  and continue to    
   lead to fundamental new results; see \cite{Hu,roeschyin,avilalyubichshen}.

 In the study of  rational functions and 
    transcendental entire functions, there is no  immediate analogue of the basin of infinity, 
    and this is  one of  the reasons that the study  of these classes has presented greater challenges
    than that of polynomials. Nonetheless, in  both settings 
    analogues of the above-mentioned
    puzzle techniques have been employed to certain classes of functions with considerable 
    success. We refer to \cite{Ro08,Be15} for two examples. 

  Our goal is to extend the Douady-Hubbard landing theorem to 
    the case of a transcendental entire function $f$. 
    In this setting, the role that critical values play in polynomial dynamics is taken by the larger set $S(f)$ 
    of \emph{singular values} of $f$. These are those points  
     not having a neighbourhood in which all branches of $f^{-1}$ are defined and holomorphic. 
    Analogously to~\eqref{eqn:postcritical}, the \emph{postsingular set} of $f$ is defined as
    \begin{equation}\label{eqn:postsingularset}
      \P(f) \defeq \overline{\bigcup_{s\in S(f)} \{f^n(s)\colon n\geq 0\}}. \end{equation}

 For transcendental maps,  
    $\infty$ is an essential  singularity,  rather than  a super-attracting
    fixed point. Hence the definition of external rays for polynomials as gradient lines of a Green's function
    has no natural analogue. 
   Nonetheless,  it has long been known that the \emph{escaping set}
    \[ I(f) \defeq \{z\in\C\colon f^n(z)\to\infty\} \]
   often contains curves to infinity; indeed, in some cases this was already noticed by Fatou~\cite{fatou26}.
   It was the work of Devaney and his collaborators (see e.g.\ \cite{DK,DT86}) that really began the 
   study of these \emph{hairs} or \emph{dynamic rays} in the 1980s, particularly for functions in the 
   \emph{exponential family},
             \begin{equation}\label{eqn:exponentialfamily} f_a\colon z\mapsto e^z+a. \end{equation}
   Devaney, Goldberg and Hubbard were probably the first to suggest that such hairs 
   could serve as analogues of external rays of polynomials; compare \cite{dgh,BD}. Subsequently,
   Schleicher and Zimmer \cite{expescaping} and Schleicher and Rottenfu{\ss}er \cite{cosescaping} proved
    that, for the families of exponential maps~\eqref{eqn:exponentialfamily} and of
     \emph{cosine maps} $z\mapsto ae^z+ be^{-z}$, 
respectively, the \emph{entire} escaping set $I(f)$ consists of hairs. 

   On the other hand, in \cite{R3S} it is shown that there is a transcendental entire function $f$ 
     for which $I(f)$ 
     contains no arcs. Hence there are no
     curves in $I(f)$, of any kind, landing at any of the repelling
     periodic points of $f$. (Recall that 
     repelling periodic points  are dense in the Julia set of any transcendental entire function.) 
    Furthermore, the postsingular set $\P(f)$ of this function is bounded. Indeed, $S(f)$ is a compact subset of 
    the immediate basin of attraction of a single attracting fixed point. 

\subsection*{Filaments}
  In view of the preceding example, we develop a novel approach to the landing problem that 
    removes the focus on hairs altogether, by connecting repelling periodic points to infinity using 
    more general sets of escaping points. 

 More precisely, we introduce a notion of \emph{filaments} for postsingularly bounded entire functions. These are 
    certain unbounded connected
    sets of escaping points
    generalising the concept of hairs. (See Section~\ref{sec:filaments} for formal definitions.) The set of
    filaments has a natural combinatorial structure, and in tame cases, all filaments are in fact
    hairs. In general, however, filaments can be topologically much more complicated.
    Indeed, it follows from \cite{arclike} that 
    the closure of a filament may be a hereditrarily indecomposable continuum.\footnote{%
    More precisely, \cite{arclike} shows that a \emph{Julia continuum} of a disjoint-type entire function
       may have this property, Every such Julia continuum is the closure of a filament in our sense;
      compare Lemma~\ref{lem:control} and Remark~\ref{rmk:disjointtype}.}

    With this terminology, we are able to prove the following generalisation of
    the Douady-Hubbard landing theorem for postsingularly bounded entire functions:
    Every periodic filament lands, and every repelling or parabolic periodic point is the landing point of at least one and at    most finitely many periodic filaments (Theorem~\ref{thm:DHfilaments}).
   In particular, {without requiring the definitions of Section~\ref{sec:filaments}, we can state
     the following result.}

\begin{thm}[Landing at periodic points]\label{thm:landingsets}
   Let $f$  be a transcendental entire function such that $\P(f)$ is bounded, and let $\zeta$ be a repelling
    or parabolic periodic point. Then there is a connected and unbounded set $A\subset I(f)$ and a period   $p$ with the following
     properties.
    \begin{enumerate}[(a)]
      \item $\overline{A} = A\cup \{\zeta\}$ and $\overline{A}$ does not separate the plane;\label{item:Aclosure} 
      \item $f^p(A)=A$ and $f^j(A)\cap A = \emptyset$ for $1\leq j < p$;\label{item:Aperiod}
      \item for every $\eps>0$, $f^n$ tends to $\infty$ uniformly on $\{z\in A\colon \lvert z-\zeta\rvert \geq \eps\}$.\label{item:Auniform}
    \end{enumerate} 
   If  $\tilde{\zeta}\neq\zeta$ is a different repelling or parabolic periodic point and $\tilde{A}$ is a set as above for $\tilde{\zeta}$, then $A\cap\tilde{A}=\emptyset$. 
\end{thm}

We emphasise that using filaments, rather than restricting to cases
 where hairs exist (see Theorem~\ref{thm:mainrays} below), is crucial
 if one wishes to obtain results for general classes
 of functions. Indeed, Pfrang \cite{pfrangthesis} uses
 our results to construct \emph{(homotopy) Hubbard trees} for all postsingularly
 finite entire functions. This is a natural result 
  whose 
  hypothesis and conclusion make no mention of hairs. Its proof in this form 
  is made possible by the use of filaments; compare the discussion at the end of the final section of \cite{pfrangrothgangschleicher}. Similarly,
  work of Fagella and the first author \cite{BF15,BF17,BF20}, 
  was formulated only for
  functions with hairs, but contains a number of results whose conclusion makes sense without this assumption. 
   For example, the conclusion of the 
  main theorem of \cite{BF17} states that every non-repelling cycle has a singular orbit that is associated to it in a certain
   explicit manner. These results should now extend to all postsingularly bounded
  entire functions, by replacing the role of hairs in the proofs by our ``filaments''. 
  In addition, the key technique of \emph{fundamental tails} that we use to control
  filaments (see Section~\ref{sec:tails}) has already found further applications,
  for instance 
  in the study of inner functions arising in transcendental dynamics \cite{FagellaTails}, and in a new version of the Fatou-Shishikura inequality~\cite{BF20}.

 Moreover, our results offer the possibility of developing puzzle-type arguments for 
   all post\-singularly bounded entire functions, and of using the powerful techniques of symbolic dynamics to study the 
   behaviour
   of non-escaping points. As mentioned above, this is the 
   reason why the structure of polynomial Julia sets is so well understood.
   Theorem~\ref{thm:landingsets} opens up large classes of entire transcendental  
   functions to the same type of analysis. 
   
\subsection*{Existence and landing of periodic hairs}
  In many interesting cases, periodic filaments are in fact 
   periodic hairs. That is, the connected set $A$ in Theorem~\ref{thm:landingsets} is
   an arc connecting $\zeta$ to $\infty$. In particular, this holds for functions
    satisyfing the following property, which states that
    the escaping set consists entirely of hairs. 

 \begin{defn}[Criniferous functions]
    We say that an entire function $f$ is \emph{criniferous}\footnote{``Criniferous'' means ``having hair'' or ``hairy'', from Latin \emph{crinis} (hair) + \emph{ferre} (to bear).}     if   the following holds for every $z\in I(f)$:
      For all sufficiently large $n$ there is an arc $\gamma_n$ 
      connecting $f^n(z)$ to $\infty$, in 
      such a way that $f$ maps $\gamma_n$ injectively onto $\gamma_{n+1}$, 
      and such that $\min_{z\in \gamma_n} \lvert z\rvert\to\infty$  
     as $n\to \infty$. 
 \end{defn}

    The counterexample from \cite{R3S} mentioned above shows that entire functions, even those with 
      bounded postsingular sets, need not be {criniferous}. However, the same article also establishes
     {criniferousness} for a large and natural class of functions, as follows. 
    The \emph{Eremenko-Lyubich class}
     $\BB$ consists of those transcendental entire functions for which $S(f)$ is bounded, and hence compact. 
     (If $\P(f)$ is bounded, then  $f\in \BB$ by definition.)
    It is proved in \cite{R3S} that $f$ is {criniferous} whenever  $f\in\BB$ and $f$ has \emph{finite order of growth}, 
     i.e., 
           \[ \log \log \lvert f(z)\rvert= O(\log \lvert z \rvert).\]
    Furthermore, any finite composition of functions with these properties is also {criniferous}. 

        To discuss periodic hairs,  let 
         us use the following definition from~\cite{Re08}.

 \begin{defn}[Periodic hairs]\label{defn:periodichair}
   An \emph{invariant hair} of  a transcendental entire function $f$ is a continuous and injective curve        
      $\gamma\colon \R\to I(f)$ such that $f(\gamma(t)) = \gamma(t+1)$ for all $t$ and 
      $\lim_{t\to+\infty} \lvert \gamma(t)\rvert = \infty$. A \emph{periodic hair} is a curve that is an invariant hair for some
      iterate $f^n$ of  $f$. 

     Such a hair \emph{lands} if the limit $z_0 = \lim_{t\to -\infty} \gamma(t)$ exists; this limit is called the
      \emph{landing point} (sometimes also \emph{endpoint}) of the hair $\gamma$. 
 \end{defn}
 
  With this terminology, Theorem~\ref{thm:landingsets} takes the following form for {criniferous} functions. 
 \begin{thm}[Landing theorem for periodic hairs]\label{thm:mainrays}
   Let $f$ be a transcendental entire function such that 
     the postsingular set $\P(f)$
     is bounded.  Then
     every periodic hair of  $f$ lands at a repelling or parabolic periodic point. If, in addition, $f$ is {criniferous}, then conversely every
     repelling or parabolic periodic point  of $p$ is the  landing point of at least one and at most finitely 
     many periodic hairs. 
 \end{thm}
  The first part of the theorem, concerning landing behaviour of periodic rays, is not new. It was 
    proved for exponential maps in~\cite{SZ03}, and later in full generality 
    by the second author~\cite[Corollary~B.4]{Re08}; see 
    also~\cite{De}. 
    The proof uses similar ideas as in the polynomial case, namely expansion properties 
    for the hyperbolic metric, although there are also some additional ingredients.   

  On the other hand, the usual proofs for accessibility
    of repelling and parabolic periodic points in the polynomial case \cite{ELv,Hu,Pr94} strongly rely on the 
    presence of the open basin of attraction of infinity, and thus break down completely in the transcendental setting. 
    Nonetheless, there has been some previous work in this direction. Under the additional dynamical assumption
    that $f$ is \emph{geometrically finite}, the theorem was proved by Mihaljevi\'c-Brandt \cite{Mi10}. 
    Furthermore, 
    the first author and Lyubich \cite{BL14} proved Theorem~\ref{thm:mainrays} when
    $f$ belongs to the exponential family~\eqref{eqn:exponentialfamily}. 

   For exponential maps, boundedness of the postsingular set is a strong dynamical condition (though weaker   than geometrical finiteness), as it implies \emph{non-recurrence} of the singular value  $a$.  However, the 
    non-recurrence property is not  used in any essential way in \cite{BL14}, and the ideas used there form one of the ingredients in our
    proofs of Theorems~\ref{thm:landingsets} and~\ref{thm:mainrays}. 

  During the preparation of this manuscript, Dierk Schleicher informed us that he has an alternative approach to Theorem~\ref{thm:mainrays}, using ideas 
   from~\cite{SZ03}.

\subsection*{Hyperbolic sets} As in \cite{BL14}, our techniques
   apply not only to repelling (and parabolic) periodic
   points, but also to \emph{hyperbolic sets}; see \cite{Pr94} for the corresponding
   result for polynomials.
   Recall that a compact, forward-invariant set $K \subset \C$ is called \emph{hyperbolic} 
   if for some $k\in\N$ and $\eta>1$  we have $\lvert (f^k)'(z)\rvert > \eta$ for all $z\in K$. 

   If $\P(f)$ is bounded and $K$ is such a hyperbolic set,  then we prove that every point of $K$ is
    ``accessible'' from the escaping set, via a filament  (see Theorem~\ref{thm:mainhyperbolic}). 
     Again, we can state the following result without requiring the terminology of filaments.

\begin{thm}[Landing at hyperbolic sets]\label{thm:hyperbolicsetsintroduction}
   Let $f$  be a transcendental entire function such that $\P(f)$ is bounded, and let 
    $K$ be a hyperbolic set of $f$. Then there is a collection $\mathcal{A}$ of pairwise disjoint, connected 
     and unbounded sets
      $A\subset I(f)$       with the following
     properties.
    \begin{enumerate}[(a)]
      \item For every $A\in\mathcal{A}$, there is $z_0(A)\in K$ such that
                    $\overline{A} = A\cup \{z_0(A)\}$, and $\overline{A}$ does not separate the plane;\label{item:hyperbolicclosure} 
      \item the function
                    $\mathcal{A}\to K; A\mapsto z_0(A)$ is surjective;
      \item $f(A)\in \mathcal{A}$ for all $A\in\mathcal{A}$; \label{item:hyperbolicimage}
      \item for every $\eps>0$, $f^n$ tends to $\infty$ uniformly on 
               $\{z\in \bigcup \mathcal{A}\colon \dist(z, K) \geq \eps\}$;\label{item:hyperbolicuniform}
      \item if $z_0(A)$ is periodic of period $p$, then $f^{kp}(A)=A$ for some $k\geq 1$.\label{item:hyperbolicperiod}
    \end{enumerate}
  If $f$ is {criniferous}, then every $A\in\mathcal{A}$ is an arc connecting $z_0(A)\in K$ to $\infty$. 
\end{thm}

   This generalisation is of particular relevance in the case where $\P(f)$ itself is a hyperbolic set,
     which is often the case for non-recurrent entire functions (see \cite{RvS}). Hence, in this case, each singular 
     value can itself be connected to infinity by a filament, which in turn allows one to study the Julia set
     via symbolic dynamics rather closely. For example, in \cite{Be15}, 
     the existence of a ray landing at the omitted value
     is exploited to prove strong \emph{rigidity} properties of non-recurrent parameters in 
    the exponential family,
     extending previous work \cite{Be11} in the postsingularly finite case.

To conclude the introduction, we remark on the case where the postsingular set $\P(f)$ is unbounded.
   If $f$ is a polynomial, then every unbounded orbit escapes to infinity. For polynomials with 
   escaping singular orbits, the Douady-Hubbard landing theorem no longer holds. Indeed,
   it is possible that a repelling periodic point is the landing point of uncountably many
   external rays, none of which are periodic. Compare \cite{LP}. 

  For transcendental entire functions, it is possible for singular orbits to be unbounded
    without converging to infinity.
    It is conceivable that, for $f\in\B$ with all singular orbits nonescaping, a version of the
    landing theorem holds. However, even for exponential maps
    this is not known (see \cite{Re06} for a partial result), and it appears that significant further new
    ideas would be required to approach it. See Section~\ref{sec:unboundedpostsingularset} for further discussion.

 \subsection*{Structure of the paper}
 {Section~\ref{sec:unboundedsets} gives an overview of expansivity properties for functions in class $\BB$ 
 without the assumption of bounded postsingular set. It also defines the concept of external addresses, and 
 gives sufficient conditions on such addresses to be realised by certain unbounded connected sets of points. 
 Several of the ideas used in this section are already implicitly or explicitly
   contained in the literature, e.g.\ in \cite{EL92,eremenkoproperty,Re08,Re09}, but 
   are combined here in a novel, systematic and  unified manner.

From Section~\ref{sec:tails} onward, we restrict to functions with bounded postsingular sets, beginning by
   discussing hyperbolic expansion estimates for such maps, and introducing the important combinatorial notion of
   \emph{fundamental tails}. With these preparations, Section~\ref{sec:filaments} 
   introduces filaments for  a function with bounded postsingular set, 
    studies their main 
   topological and combinatorial properties, and also shows that the escaping 
   set consists of filaments. The ideas in this section have their 
   roots in~\cite{eremenkoproperty}. In particular, we recover
   the main result of that paper; see Corollary~\ref{cor:lms}.
   In Section~\ref{sec:hairs}, we discuss the relation between filaments and hairs. 

 Section~\ref{sec:accumulationlanding} introduces accumulation sets of filaments at bounded addresses, 
   and gives different characterisations of when a filament \emph{lands}. This section contains a crucial   innovation, which is central to the proofs of our main theorems: Rather  than having    to contend with the potentially complicated topological structure of filaments,
   we can instead study their landing properties by considering a certain chain of 
   open simply connected sets. In Section~\ref{sec:separation}, we establish that such a landing filament cannot separate the plane.
  We are then ready to state our main theorems concerning filaments
   in  Section~\ref{sec:maintheorems}, and to derive
   Theorems~\ref{thm:landingsets} and~\ref{thm:mainrays} from these.

 {The} three following sections are dedicated to proving the main theorems of this paper. 
  Section~\ref{Periodic filaments land} establishes the landing of periodic filaments, 
   in Section~\ref{sec:Landing at hyperbolic sets} we show accessibility of hyperbolic sets and repelling periodic orbits,
  and finally Section~\ref{sec:parabolic} is dedicated to the proof of accessibility of parabolic points.

  We remark that one can take an alternative, less natural but more direct, approach to establishing our theorems,
    bypassing most of the material in Sections~\ref{sec:unboundedsets} and \ref{sec:filaments}--\ref{sec:accumulationlanding}. 
    Readers interested in such a short-cut are referred to  Remark~\ref{rmk:shortcut}.

  To round off the paper, Section~\ref{sec:classS} discusses bounds on the number of rays landing together at a given point
    in a hyperbolic set.
 We also include two appendices. The first, Section~\ref{sec:cyclicorder}, gives some details concerning the cyclic order at infinity of 
   unbounded connected sets, which are used in some of our arguments. The second, Section~\ref{sec:unboundedpostsingularset}, discusses
   open questions about landing theorems for entire functions with unbounded postsingular sets.

   \subsection*{Notation and preliminaries} We write $\C$ for the complex plane and $\hat{\C}$ 
   for the Riemann sphere. 
   We denote the closure in $\C$ of a set $A\subset\C$ by $\overline{A}$, and occasionally 
     $\closure(A)$. The closure of $A$ in $\Ch$ is denoted by
     $\hat{A}$.

  The Euclidean disk of radius $R$ around a point $z$ is denoted by $D_R(z)$; 
   the unit disk is $\D\defeq D_1(0)$.  If  $D$ is any Euclidean disc, we also write $r(D)$ for the radius of $D$. 

   We denote Euclidean distance and diameter by $\dist$ and $\diam$, respectively. If $U\subset\C$ is an open
     set omitting more than two points, then we denote hyperbolic distance on $U$ by $\dist_U$, and similarly
     $\diam_U$ for hyperbolic diameter. We also denote the density of the hyperbolic metric of $U$ at a point
      $z$ by $\rho_U(z)$. That is, the   length element of the hyperbolic metric is given by  $\rho_U(z)\lvert \deriv z\rvert$. 

 \subsection*{Acknowledgements} We are extremely grateful to Dave Sixsmith and to David Pfrang for their many and extraordinarily helpful suggestions that 
    considerably improved the presentation of the paper. We also thank Daniel Meyer for interesting comments, particularly a suggestion on the presentation
    of cyclic order in Section~\ref{sec:cyclicorder}.

\section{Unbounded sets of escaping points}\label{sec:unboundedsets}
  {In this section, we briefly review basic properties of the dynamics of 
    a function $f\in\B$, and  review the definition of external addresses for such maps.  
    Then we state a theorem (Theorem~\ref{thm:unboundedsets}) about the existence of 
    unbounded connected sets for these addresses, and devote the rest of the section to the proof thereof. 
    These sets will provide the basis of the ``filaments'' that
    are introduced (for postsingularly bounded entire functions) in Section \ref{sec:filaments}.}

{Throughout this section, fix a function  $f\in\B$. Recall that this implies that $S(f)$ is bounded. 
   For now, we do \emph{not} assume that $\PP(f)$ is also bounded.
  Let us} begin by reviewing the method of partitioning the locus where $f$ is large into (topological) half-strips known as
    \emph{fundamental domains}. 
    (Compare e.g.\ \cite[Section~2]{Re08} or \cite[Section~2]{guenterthesis}.)
    For this construction, we 
   fix a Euclidean disk $D$ around the origin containing $S(f)$. 
   The connected components of $f^{-1}(\C\setminus \overline{D})$ are called the \emph{tracts} of~$f$. 
    If $T$ is a tract, then $f\colon T \to \C\setminus \overline{D}$ is a universal covering map;
     in particular, $T$  is unbounded and simply connected. In fact (applying the same argument to a slightly smaller
     disc than  $D$), $T$ is a Jordan domain in $\Ch$ whose boundary passes through infinity, 
     and $f$ is a universal covering $f\colon \overline{T} \to \C\setminus D$ on the closure of $T$ (in $\C$). 

 We may assume in the following that $D\cap f(D)\neq \emptyset$, e.g.\ 
    by ensuring that $f(0)\in D$. Then it is easy to see that there is an arc 
    $\delta$ connecting  a point of $\partial D$ to infinity in the complement of the 
     closure 
      of the tracts. We define
     \begin{equation}
         \WW \defeq \C\setminus( \overline{D}\cup\delta).
     \end{equation}
   The connected components of $f^{-1}(\WW)$ are called the 
    \emph{fundamental domains} of $f$; see Figure~\ref{fig:tractsfundamentaldomains}. 

\begin{figure}
\def\svgwidth{.65\textwidth}
\begingroup%
  \makeatletter%
  \providecommand\color[2][]{%
    \renewcommand\color[2][]{}%
  }%
  \providecommand\transparent[1]{%
    \renewcommand\transparent[1]{}%
  }%
  \providecommand\rotatebox[2]{#2}%
  \ifx\svgwidth\undefined%
    \setlength{\unitlength}{705.78776855bp}%
    \ifx\svgscale\undefined%
      \relax%
    \else%
      \setlength{\unitlength}{\unitlength * \real{\svgscale}}%
    \fi%
  \else%
    \setlength{\unitlength}{\svgwidth}%
  \fi%
  \global\let\svgwidth\undefined%
  \global\let\svgscale\undefined%
  \makeatother%
  \begin{picture}(1,0.73807518)%
    \put(0,0){\includegraphics[width=\unitlength]{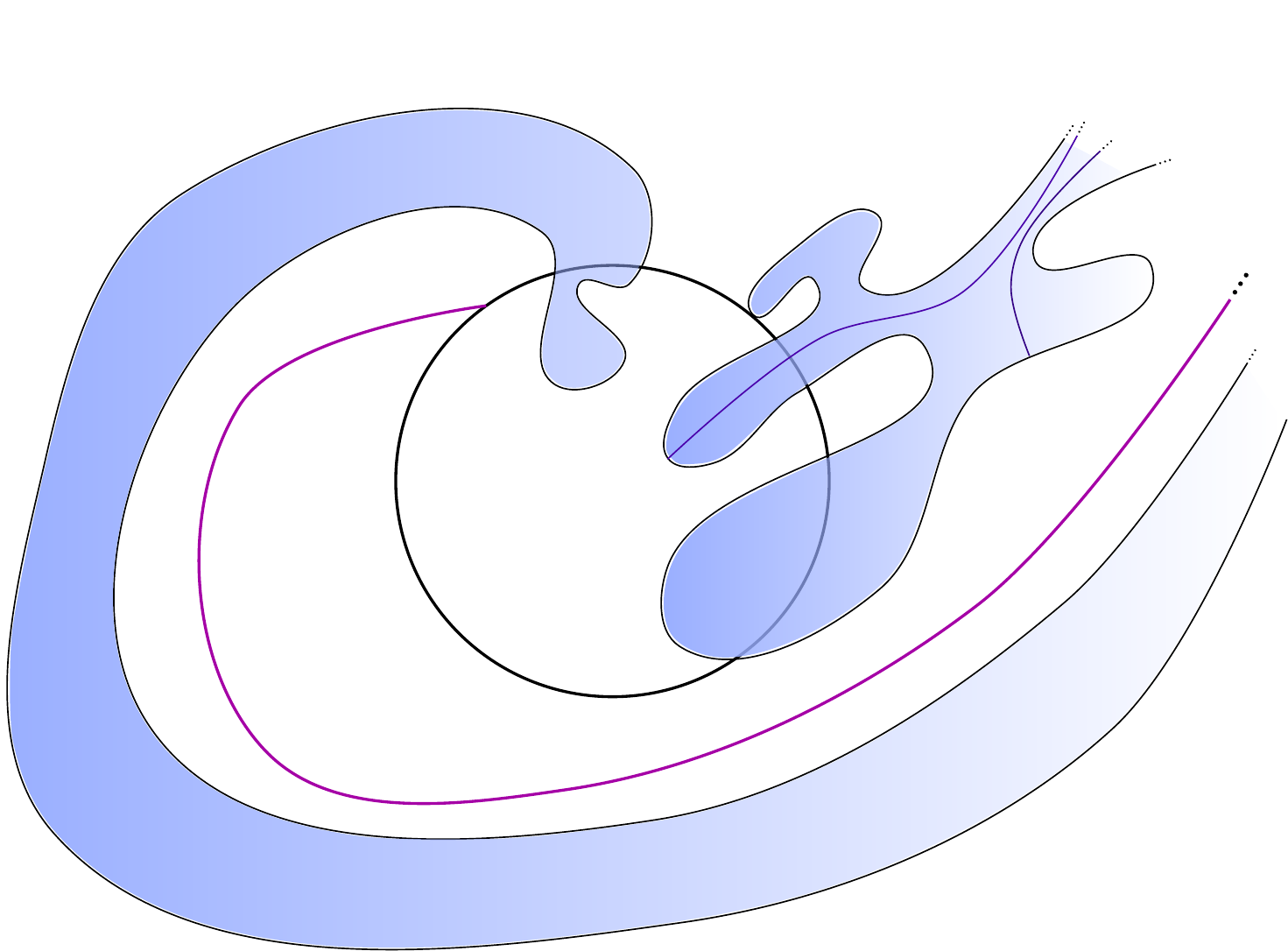}}%
    \put(0.35,0.35){\color[rgb]{0,0,0}\makebox(0,0)[lb]{\small{$D\supset P(f)$}}}%
    \put(-0.15,0.6){\color[rgb]{0,0,0}\makebox(0,0)[lb]{\small{$W_0=\C\setminus(\ov{D}\cup\delta)$}}}%
    \put(0.1213254,0.32438186){\color[rgb]{0,0,0}\makebox(0,0)[lb]{\small{$\delta$}}}%
    \put(0.13562029,0.0512865){\color[rgb]{0,0,0}\makebox(0,0)[lb]{\small{$T_1$}}}%
    \put(0.73980271,0.58009572){\color[rgb]{0,0,0}\makebox(0,0)[lb]{\small{$T_2$}}}%
    \put(0.73860633,0.47002682){\color[rgb]{0,0,0}\makebox(0,0)[lb]{\small{$F_i$}}}%
  \end{picture}%
\endgroup%
  \caption{\label{fig:tractsfundamentaldomains}
 The dynamical plane for  a function with two tracts $T_1$ and $T_2$. One of the fundamental domains  $F_i$ obtained by taking preimages of $\delta$ is shown inside $T_2$.}
\end{figure}

     We remark that there are only finitely many fundamental domains that intersect a given compact set, due to the following simple fact. 

     \begin{lem}[Preimage components intersecting a compact set]\label{lem:preimageK}
       Let $f\colon X \to Y$ be a holomorphic map
        between Riemann surfaces $X$ and $Y$. 
       Furthermore, let $U\subset Y$ be a domain 
       whose boundary (in $Y$) is locally connected; i.e.\
       every point of the boundary of $U$ in $Y$ has arbitrarily small connected relative neighbourhoods in $Y$.

      Then for any compact set $K\subset X$, 
         only finitely many connected components of  $f^{-1}(U)$ intersect $K$.
     \end{lem}
  \begin{remark}
    The condition that $\partial U$ is locally connected is necessary: Let $X=Y=\C$, $f=\exp$, and let $U$ be a simply connected domain
     in the punctured unit disc that spirals in towards the unit circle. (I.e., any branch of the argument on $U$ tends to infinity as $\lvert z\rvert\to 1$ in $U$.) 
    Then infinitely many components of $f^{-1}(U)$ intersect the closed unit disc.
  \end{remark}
     \begin{proof}
        We begin by reformulating the hypothesis that $\partial U$ is locally connected, as follows.
        \begin{claim}
         Let $\zeta\in \overline{U}$, and let $\Delta$ be a neighbourhood of $\zeta$ in $Y$. 
         Then there is a finite collection of connected open sets $W_1,\dots,W_n\subset V\cap U$ such that  
          $\{\zeta\}\cup W_1\cup \dots \cup W_n$ is a neighbourhood of  $\zeta$ in $\{\zeta\}\cup U$. 
         \end{claim}
         \begin{subproof}
          Shrinking $\Delta$ if necessary, we may assume that $\overline{\Delta}$ is a closed topological disc, and that 
             $U\not\subset \Delta$. 
                 Consider the compact set 
                            $Q \defeq \partial \Delta \cup (\Delta\setminus U)$.
                 Since $U$ is connected but not contained in $\Delta$, the boundary of each connected component of                  $U\cap\Delta = \Delta\setminus Q$ 
                 intersects $\partial\Delta$. Recall that $\partial U$ is locally connected; it follows readily that 
                 $Q$ is also. Hence $U\cap\Delta$ has 
                 at most finitely many connected components of diameter greater than, say,  
                    $\delta \defeq \dist( \zeta ,\partial \Delta)/2$; see 
                 \cite[Theorem~4.4 in Chapter VI]{whyb}. 
                 (Here $\dist$ refers to distance with respect to some metric on the topological disc $\overline{\Delta}$.) 
                 Therefore only finitely many connected components 
                 of
                 $W_1,\dots,W_n$ of $U\cap\Delta$ 
                 intersect the disc of radius $\delta$ around $\zeta$, as claimed.
         \end{subproof}
          
        Let $z\in f^{-1}(\overline{U})$, and $\zeta\defeq f(z)$. Then $z$ has a neighbourhood $V_1$ that is topologically mapped 
           as by
          $z\mapsto z^d$, where
          $d$ is the local degree of $f$ at $z$. Take $W_1,\dots,W_n$ as in the claim, for $\Delta = f(V_1)$. 
        If $V_z\subset V_1$ is a sufficiently small disc around  $z$, then 
         any point in $f^{-1}(U)\cap V_z$ maps into some $W_j$. As $f^{-1}(W_j)$ has $d$ connected components in  
         $V_1$, it follows  that 
         $V_z$ intersects at most $dn$ connected components of $f^{-1}(U)$.

        So the compact set $K\cap f^{-1}(\overline{U})$ has an open cover by sets $V_z$, each of which intersects
         only finitely many connected components of $f^{-1}(U)$. 
          The claim follows by taking a finite subcover.
     \end{proof} 

     It follows that 
     there are only finitely many fundamental domains $F$ whose closure intersects the disc $\overline{D}$. 
      When this does not occur for any $F$, the function $f$ is dynamically particularly simple; more precisely,
      it is of \emph{disjoint type} (hyperbolic with connected  Fatou set). For a detailed study of the 
      topological dynamics of such functions, see~\cite{arclike}. (Compare also the discussion of 
      disjoint-type addresses in Remark~\ref{rmk:disjointtype}.) In the following, given a  fundamental domain  $F$ we 
      denote by $\unbdd{F}$ the unbounded connected component of $F\setminus \overline{D}$.

 \subsection*{Expansion properties and relative cylindrical distance}
 It is known that functions in $\B$ are strongly expanding near infinity. More precisely, if $f\in\B$, then the 
   \emph{cylindrical derivative} of 
   $f$ is large whenever $f(z)$ is large \cite[Lemma~1]{EL92}. That is,
    \begin{equation}\label{eqn:ELexpansion}
     \lVert \Deriv f(z)\rVert_{\cyl} \defeq 
      \left\lvert f'(z) \cdot \frac{z}{f(z)}\right\rvert \to \infty \qquad \text{as}\qquad \lvert f(z)\rvert \to \infty. 
    \end{equation}
  In view of~\eqref{eqn:ELexpansion}, we may assume that the radius $r(D)$ is chosen 
    sufficiently large to ensure that
      \begin{equation}\label{eqn:normalised}
         \lVert \Deriv f(z)\rVert_{\cyl} \geq 2
      \end{equation} 
    whenever $f(z)\notin D$. In particular, $f(0)\in D$. These assumptions will remain 
    in place for the remainder of the paper. 

   A number of results in the literature are phrased not for  the function  $f$ directly, but in terms of a
    \emph{logarithmic transform} $\mathcal{L}$ of $f$. (See e.g.\ \cite{EL92} or \cite{eremenkoproperty}.) 
     Such a transform can be
     obtained using the change of variable $z = \rho \cdot \exp(\zeta)$, where $\rho=r(D)$ is the radius of
      $D$. I.e., there is a $2\pi i$-periodic
      function $\mathcal{L}$ defined by
       \[ \rho\cdot \exp(\mathcal{L}(\zeta)) = f(\rho\cdot \exp(\zeta)), \] 
     defined whenever the right-hand side (i.e., $f(z)$) belongs to $\C\setminus\overline{D}$. 
     Dynamical properties of $\mathcal{L}$
     easily translate to properties of  $f$ on the set of points whose orbits remain outside $\overline{D}$ forever. 
     Our assumptions on $D$ imply that
     the function $\mathcal{L}$ is \emph{normalised} in the sense of \cite{eremenkoproperty,R3S}. 

    We occasionally
     cite results from other articles that are phrased in this language, but
     never use the logarithmic transform $\mathcal{L}$ directly in this article. Instead, we use the following terminology, 
     which is inspired by this change of coordinates. See Figure~\ref{fig:W0}~\subref{subfig:W0-illustration}. 

   \begin{defn}[Relative cylindrical distance]\label{defn:relativedistance}
      For $z,w\in\WW$, we define the \emph{relative cylindrical distance} 
          $\distW(z,w)$ to be the shortest cylindrical length of a curve $\gamma$ from    
         $z$ to  $w$ that is {homotopic, in $\C\setminus\overline{D}$, to a curve in $\WW$}. 
    
        Equivalently, if $U$ is a connected component of $\exp^{-1}(\WW)$ and $\zeta$, $\omega$ are the  
        logarithms of $z$ and  $w$ that belong to  $U$, then 
      \[ \dist_{\cyl}^{\WW}(z,w) = \lvert \zeta - \omega \rvert. \]
        
          We similarly define the distance between two subsets of $\WW$, and the 
       diameter $\diamW$ with respect to this metric. 
   \end{defn}
 
\begin{figure}
  \subfloat[Definition~\ref{defn:relativedistance}]{
      \fbox{
\def\svgwidth{.47\textwidth}
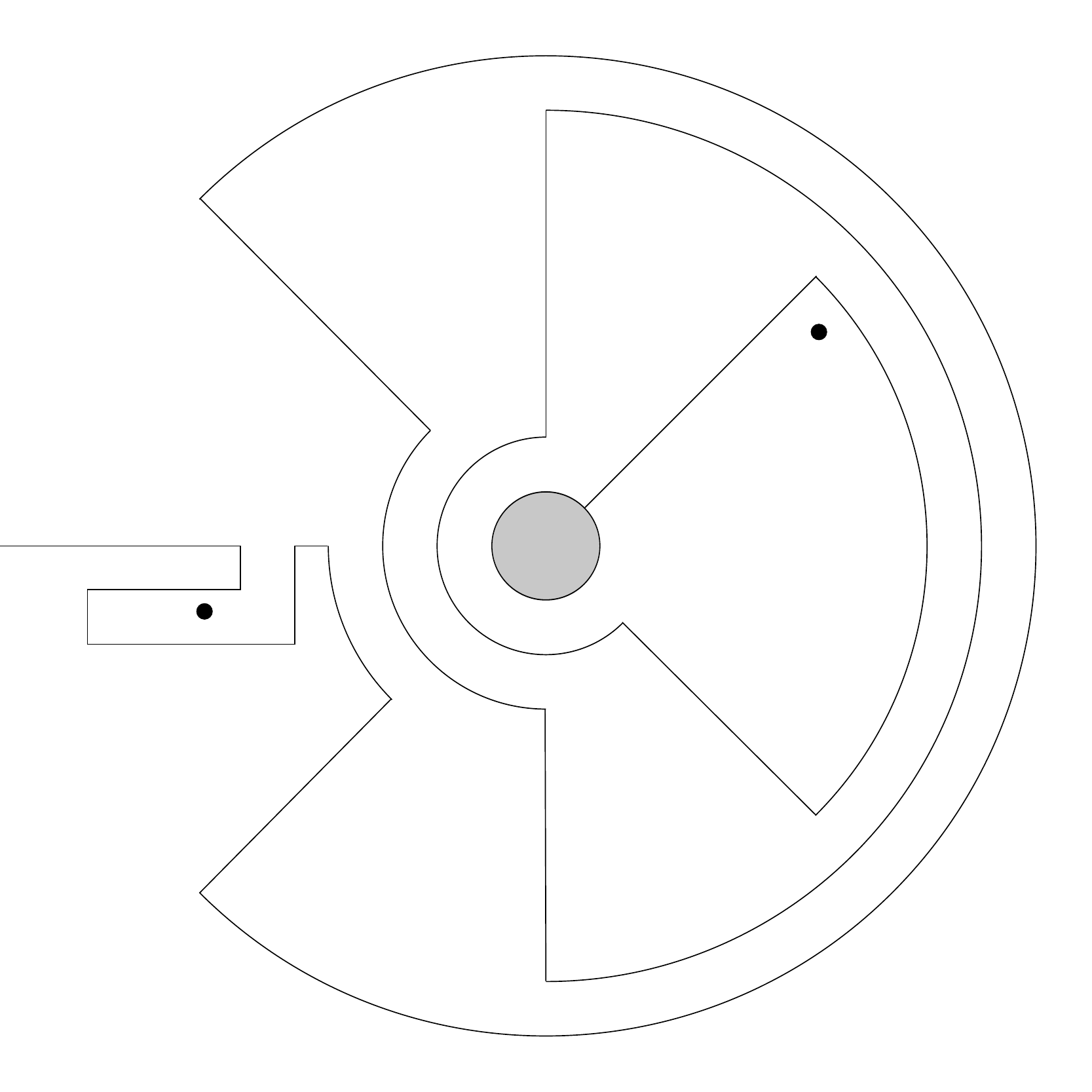\label{subfig:W0-illustration}
  }}\hfill
  \subfloat[Lemma~\ref{lem:crosscuts}]{
      \fbox{
\def\svgwidth{.47\textwidth}
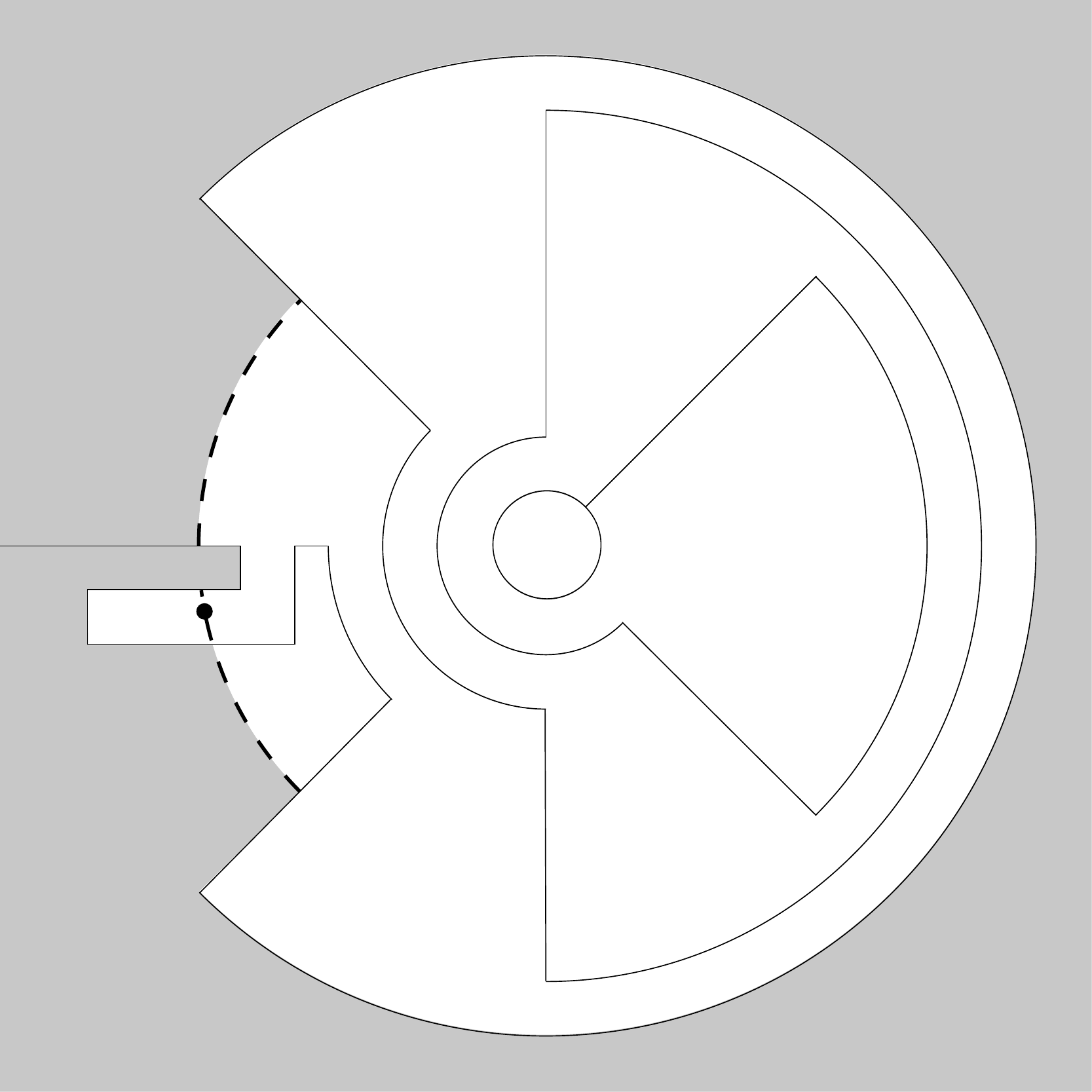\label{subfig:W0-crosscuts}
  }}
  \caption{\label{fig:W0}
     In~\protect\subref{subfig:W0-illustration}, the cylindrical distance between $z$ and $w$ is less than $\pi$. However,
       any curve connecting $z$ and $w$ in $\WW$ takes more than two full turns around  $D$, and hence 
       the relative distance in the sense of Definition~\ref{defn:relativedistance} is greater than  $4\pi$. 
       The second image, in~\protect\subref{subfig:W0-crosscuts} illustrates Lemma~\ref{lem:crosscuts} in this case.
       Here $\Gamma$ is a union of three cross-cuts of $\WW$ satisfying the conclusion of the lemma; 
       the domain $\WW^{\Gamma}$ is shown
       shaded in grey. (Nb.\ the set $\Gamma$ constructed in the proof of Lemma~\ref{lem:crosscuts}
       contains two additional cross-cuts; omitting these does not change the domain
       $\WW^{\Gamma}$, and therefore the conclusions of the lemma.)}
\end{figure}

      Since $f$ is expanding with respect to the cylindrical metric on  $\WW$ by~\eqref{eqn:normalised}, we have
        \begin{equation}\label{eqn:expansionW} \distW(f(z),f(w)) \geq 2\cdot \distW(z,w) \end{equation}
       whenever $z$ and $w$ both belong to $\unbdd{F}$ for some fundamental domain $F$. 

    Let us also note the following fact for future reference.
     \begin{obs}\label{obs:absolutevaluecylindrical}
      Let $\zeta_0\in \WW$. If
        $(z_n)_{n=0}^{\infty}$ is a sequence of points in  $\WW$, then 
        $\lvert z_n\rvert \to\infty$ if and only if $\distW(z_n, \zeta_0)\to\infty$. 
     \end{obs}

   \subsection*{External addresses and symbolic dynamics}
    The reason for introducing fundamental domains is that they can be used to assign symbolic dynamics to points
     whose orbit stays sufficiently large, and in particular to escaping points. 

    \begin{defn}[External addresses]
      Let $f\in\B$, and let fundamental domains be defined as above. 
      An \emph{(infinite) external address} is a sequence $\s= F_0  F_1 F_2 \dots$ of fundamental domains of 
        $f$.  The address $\s$ is \emph{bounded} if the set of fundamental domains occurring in  $\s$ is finite; 
       it is \emph{periodic} if there is  $k$ such that $F_{n+k}=F_n$ for all $n\geq 0$. 

      Let  $\s=F_0 F_1 F_2\ldots$ be an external address, and recall that $\unbdd{F}_n$  denotes the unbounded connected component 
        of $F_n\setminus\overline{D}$.  Then we define 
        
        \[J^0_{\s}(f):=\{z\in\C\colon  f^n(z)\in\unbdd{F_n} \text{   for all $n\geq 0$}\}. \]
    \end{defn}
 \begin{remark}
      For the purpose of this paper, we shall often use ``address'' synonymously with ``external address''. 
   \end{remark}

   The main goal of this section is to prove the following.
     \begin{thm}[Realisation of addresses]\label{thm:unboundedsets}
 Let $\s$ be an external address.  
     \begin{enumerate}[(a)]
       \item Suppose that $J^0_{\s}$ contains some point $z_0$. 
          Then  $J^0_{\s}$ also contains a {closed} unbounded 
          connected set $X$ on which the iterates of $f$ tend to infinity uniformly. Moreover,           
             $\distW(z_0, X)\leq 4\pi$. \label{item:unboundedsets}  
       \item If  $X_1$ and  $X_2$ are unbounded, closed, connected subsets of $J^0_{\s}$ with
               $X_1\not\subset X_2$, then
                $X_2\subset X_1$ and  $f^n|_{X_2}\to\infty$ uniformly. \label{item:terminal}
      \item  If $\s$ is bounded, then  $J^0_{\s}\neq\emptyset$. {Furthermore, 
                 there exists $R>r(D)$, depending on the finite collection of fundamental domains
                 occurring in $\s$ but not otherwise on $\s$, such that the set $X$ in~\ref{item:unboundedsets}
                 can be chosen to contain a point of modulus $R$.}\label{item:boundedexistence} 
      \item\label{item:boundedescape}%
     Conversely, if $\FF$ is a finite collection of fundamental  domains, then there is $R>0$
                such that the iterates of $f$ tend to infinity uniformly on the closed set
      \begin{equation}\label{eqn:boundedescape} \bigcup_{\s\in\FF^{\N_0}} J_{\s}^0\setminus D_R(0) = 
            \left\{z\in \C\colon \lvert z \rvert \geq R \text{ and }
                      f^n(z)\in {\bigcup_{F\in\FF} \overline{F}} \text{ for all $n\geq 0$} \right\}. \end{equation}
     \end{enumerate}
       \end{thm}
 
 This collection of results is not entirely new. Claims~\ref{item:unboundedsets} 
    and~\ref{item:terminal} are
   variants of  Proposition~3.2 and Corollary~3.4 of \cite{eremenkoproperty}.
   Part~\ref{item:boundedescape} follows from \cite[Lemma~2.1]{Re08}. 
   Claim~\ref{item:boundedexistence} is
   proved in  \cite[Theorem~2.4]{Re08} for fixed addresses, and the proof extends directly to the case of
   arbitrary bounded addresses; compare also \cite[Proposition 2.11]{BF15}.  Alternatively, when
     $f$ is of disjoint type, $J_{\s}\neq\emptyset$ for bounded
   addresses $\s$ (and even for certain unbounded $\s$, see Proposition~\ref{prop:exponentiallybounded} below)
    by \cite[Corollary~B' on p.~405]{BK07};
   compare also~\cite[Proposition~3.10]{arclike}. Using the results of~\cite{Re09}, it can then be deduced for 
   general
   functions in  the class $\B$ 
   that 
     $J^0_{\s}\neq \emptyset$ for bounded $\s$.
 
   Since the papers in question all use slightly different notation, we shall give a
    new proof of Theorem~\ref{thm:unboundedsets} that is self-contained
    and unified.
    We begin with a simple property of the set $J_{\s}^0$, which is 
    similar to \cite[Theorem~1]{EL92}.
   
   \begin{lem}[One-dimensionality of $J_{\s}^0$]\label{lem:J}  
     Let $\s=F_0 F_1\ldots$ be an external address. Then~$\overline{J_{\s}^0}$ is a subset of $J(f)$, 
     has empty interior and does not separate the plane.
   \end{lem}
   \begin{proof}
    For each $n\geq 0$, $U_n\defeq \C\setminus \closure(\unbdd{F}_n)$ is connected. Suppose, by contradiction, that
      $V_0$ was a connected component of $\C\setminus \partial  J_{\s}^0$ that does not contain 
      $U_0$.  Then $V_0\subset \unbdd{F}_0$, and $f\colon V_0 \to f(V_0)$ is a conformal isomorphism.

     It follows inductively that  {$f^n(V_0)\subset \unbdd{F}_n$} for all $n\geq 0$. On the other hand, fix $z_0\in V_0$, and   
     set $z_n\defeq f^n(z_0)$.  
   {It follows from the above that $f^n\colon V_0\to f^n(V_0)$ is univalent for all $n$. By~\eqref{eqn:expansionW}
    and the definition of the cylindrical derivative, we have $\lvert (f^n)'(z_0) \rvert / \lvert f^n(z_0)\rvert \to\infty$
    as $n\to\infty$. By 
    Koebe's $1/4$-theorem it follows that $0\in f^n(V_0)$ for sufficiently 
    large $n$. This is a contradiction and proves that $\overline{J_{\s}^0}$ has empty interior and does not separate the plane.}  

      Let
     $z\in \overline{J_{\s}^0}$. Suppose first that $\dist(f^n(z),D)\to\infty$ 
     as $n\geq 0$. If $z$ belonged to the Fatou set of $f$, then it would follow from equicontinuity that
     there are $n_0$ and a 
     neighbourhood $V$ of $z$ such that $f^n(V)\cap\overline{D}=\emptyset$ for
    all $n\geq n_0$. But then $f^{n_0}(V)\subset J_{\sigma^{n_0}(\s)}^0$, and this contradicts the result we have just proved. 
    So $z\in J(f)$.

   Otherwise, 
      there is a sequence $n_k$ such that $f^{n_k}(z)\not\to\infty$. Then the spherical derivative 
      of $f^{n_k}$ at $z$ is comparable to the corresponding cylindrical derivative. 
      By~\eqref{eqn:expansionW}, the latter tends to infinity as $k\to\infty$. Thus
      the family of iterates of $f$ is not normal at $z$ by  Marty's theorem, and again $z\in J(f)$.
   \end{proof} 

 \subsection*{A separation lemma} We now formulate a key technical lemma~--
       closely related to Lemmas~3.1 and~3.3 of \cite{eremenkoproperty}~-- that
      will be 
        crucial for our unified proof of Theorem~\ref{thm:unboundedsets}. 
        Before making the formal statement, which is slightly technical,
         let us explain the idea. Let $z_0\in \WW$, and suppose that   
            $z_0$ can be connected to infinity within the set of points of modulus greater
         than $R$. Then our lemma states (in particular) 
         that no point $\zeta$ of modulus at most $R$ can be 
         connected to infinity without passing near $z_0$, in the 
         sense of relative cylindrical distance. 

         It may appear at first as though this is obvious, since the round circle
          centred at $0$ and of modulus $\lvert z_0\rvert$ has cylindrical
         length $2\pi$, and must intersect
         any curve connecting $\zeta$ to infinity. 
         However, the diameter of the intersection of this circle with 
         $\WW$ 
         may be arbitrarily large when measured with respect to $\distW$;
         see~Figure~\ref{fig:W0}. 
          
   \begin{lem}[Cross-cuts of $\WW$]\label{lem:crosscuts} Let $z_0\in\WW$.       Then there exists a union $\Gamma\ni z_0$ of cross-cuts of $\WW$ 
       such that $\distW(z_0,\zeta)\leq 2\pi$ for all $\zeta\in\Gamma$ and such that 
       the unbounded connected component ${\WW}^{\Gamma}$ of $\WW\setminus \Gamma$ 
      has the following property. 
      If $R>0$ is such that $z_0$ belongs to the unbounded connected component of $\WW\setminus \overline{D_R(0)}$,
      then  $\WW^{\Gamma}$ is also disjoint from $\overline{D_R(0)}$. 
      Here $\Gamma$ can be chosen to consist only of arcs of the circle of radius $\lvert z_0\rvert$ centred at the origin. 

      Moreover, suppose that $A\subset \WW$ is any unbounded connected set with  $z_0\in A$. 
       Then, for all $z\in {\WW^{\Gamma}}$, $\distW(z,A) \leq 2\pi$. 
   \end{lem}
   \begin{remark}
     The curve $\delta$ in the definition of fundamental domains
      can be chosen to be piecewise analytic, in which case the number of cross-cuts 
      in  $\Gamma$ is 
      necessarily finite. However, we do not require this.
   \end{remark}
    \begin{proof}
      Let $U$ be a connected component of  $\exp^{-1}(\WW)$; then $\exp\colon U\to \WW$ is a conformal
        isomorphism. For $\zeta\in U$, let $I_{\zeta}$ denote the
        vertical segment $\zeta + i\cdot [-2\pi , 2\pi]$. The fact that $U$ is disjoint from its $2\pi i\Z$-translates
        implies the following separation property: if $\zeta_0,\zeta_1\in U$ with $\zeta_0\notin I_{\zeta_1}$, then
        either $I_{\zeta_1}$ separates $\zeta_0$ from infinity in $U$, or vice versa. 
        (Compare~\cite[Lemma~3.3]{eremenkoproperty}.) 
        
        Indeed, suppose otherwise. Then for $j=0,1$, there
         is a curve $\gamma_j\subset U$ connecting $\zeta_j$ to infinity, and not intersecting $I_{\zeta_{1-j}}$. 
         Set $A\defeq \gamma_0\cup\gamma_1\cup\{\infty\}$, and let $a\in A$ be a point of minimal real part; say 
         $a\in \gamma_j$. By~\cite[Corollary~5.4]{arclike}, $I_{\zeta}$ separates $a$ from $\infty$ for every
         $\zeta\in A\setminus I_{\zeta}$. This is a contradiction to the fact that $\zeta_{1-j}\in A$, but $\gamma_j$ does not
         intersect $I_{\zeta_{1-j}}$.

         Now let $\zeta_0$ be the unique point of
        $\exp^{-1}(z_0)\cap U$ and set $I\defeq I_{\zeta_0}$. Observe that the endpoints of $I$ are elements of
         $\exp^{-1}(z_0)$ different from  $\zeta_0$, and hence do not belong to  $U$.  
         We set $X \defeq I\cap U$; then $X$ is a collection of cross-cuts of $U$. 
         Set $\Gamma \defeq \exp(X)$; we will prove the claims of the lemma by considering the
         unbounded connected component $V_1$ of $U\setminus I$. In other words, $V_1$ consists
         of all points of $U$ that are not separated from $\infty$ by $I_{\zeta}$.

       By the separation property, for all $\zeta\in V_1$, $I_{\zeta}$ separates $\zeta_0$ from infinity in $U$. 
        So if $A\subset U$ is an unbounded connected set with $\zeta_0\in A$, then $I_{\zeta}\cap A\neq \emptyset$, and in particular
        $\dist(\zeta,A) \leq 2\pi$. Moreover, by hypothesis 
        $\zeta_0$ can be connected to infinity by a curve in $U$ that stays at real parts greater than $\log R$. So 
        $\re \zeta > \log R$ for all $\zeta\in V_1$. 

        Since $\exp(V_1)= \WW^{\Gamma}$, this completes the proof of the lemma.
    \end{proof} 

\subsection*{Existence of unbounded sets of escaping points} 
   With these preparations, we are now able to prove the existence of unbounded connected subsets of
     $J_{\s}^0$ under very general hypotheses. 

\begin{thm}[Unbounded subsets of $J_{\s}^0$]\label{thm:unboundedgeneral}
   Let $\s = F_0 F_1 F_2 \dots$ be an external address. Suppose that $(z_n)_{n=0}^{\infty}$ is a
     sequence of points such that each $z_n\in \unbdd{F}_n$ for all  $n \geq 0$. 

    Suppose furthermore that there is $C>0$ such that
     $\distW(z_n, f|_{F_n}^{-1}(z_{n+1}))\leq C$ for all $n\geq 0$ for which  
     $(f|_{F_n})^{-1}(z_{n+1})\in\unbdd{F}_n$. 

   Then there is a closed, unbounded and connected set $X\subset J_{\s}^0$ such that
       {$\distW(z_0, X)\leq  2\max(2\pi, C)$}.
\end{thm}
\begin{proof}
  For  $n\geq 0$, let $\Gamma_n$ be the union of cross-cuts from Lemma~\ref{lem:crosscuts}, applied to
       $z_n$. 
   Let $A_n^k$ be a sequence of unbounded closed connected sets, defined for $n, k \geq 0$ as follows.
     Let $A_n^0$ be the closure of the unbounded connected component of
      $F_n\setminus \Gamma_n$. For $k\geq 0$, let
     $A_{n}^{k+1}$ be the closure of the unbounded connected component of 
     $(f|_{F_n})^{-1}(A_{n+1}^k)\setminus \Gamma_n$.  Observe that $A_n^{k+1}\subset A_n^k$ for all $n$ and $k$.

     \begin{claim}
       {$\distW(A_n^k , z_n) \leq 2\max(2\pi, C)$ for all $n,k\geq 0$.}
     \end{claim}
     \begin{subproof}
       Since $A_n^0$ intersects $\Gamma_n$, the claim is true for $k=0$. 
        Suppose that $k\geq 0$ is such that the claim holds for all $n\geq 0$.
           Let $n\geq 0$, and set 
         $B\defeq (f|_{F_n})^{-1}(A_{n+1}^k)$. If $B\cap \Gamma_n\neq\emptyset$, then 
            $A_n^{k+1}$ intersects $\Gamma_n$ and the claim is immediate from the properties of $\Gamma_n$.
          Otherwise, $A_n^{k+1}=B$. {Suppose first that $(f|_{F_n})^{-1}(z_{n+1}))\notin\overline{D}$. Then 
          \[ \distW( A_n^{k+1} , z_n) \leq 
                  \distW( A_n^{k+1} , (f|_{F_n})^{-1}(z_{n+1})) + C \leq 
                  \max(2\pi , C ) + C \leq 2\max( 2\pi , C) \]
       by~\eqref{eqn:expansionW}. Now suppose that $(f|_{F_n})^{-1}(z_{n+1}))\notin\unbdd{F}_n$.
       By the inductive hypothesis, and since $\Gamma_n$ separates $\partial D$ from $B$, we can connect
       $B$ to $\Gamma_n$ by a curve $\gamma$ with $\diamW(\gamma)\leq \max(2\pi,C)$. Since
       $\diamW(\Gamma_n)\leq 2\pi$, the claim follows.}
     \end{subproof}

   Now set $X_1 \defeq \bigcap_{k\geq 0} A_0^k$. Then
      $X_1 \cup \{\infty\}$ is compact and connected as a countable intersection of 
       compact connected sets. Moreover, $X_1$ contains a point $\zeta$ with $\distW( \zeta , z_0)\leq 2\max(2\pi , C)$. 
       If $X$ is the connected component of $X_1$ containing $\zeta$, then $X$ is unbounded by the  
       boundary bumping theorem \cite[Theorem~5.6]{continuumtheory}, and the proof is complete. 
\end{proof}

 In particular, we obtain the following partial results towards Theorem~\ref{thm:unboundedsets}.

  \begin{cor}[Realised addresses have unbounded sets]\label{cor:unboundedsets}
    Suppose that $\s$ is an external address, and that there is a point $z_0\in J_{\s}^0$. 
          Then  $J^0_{\s}$ contains an unbounded 
          closed connected set $X$, and 
          $\distW(z_0, X)\leq 4\pi$.
  \end{cor}
  \begin{proof}
    Set $z_n \defeq f^n(z_0)$, and apply Theorem~\ref{thm:unboundedgeneral}.
   \end{proof} 

 \begin{cor}[Bounded addresses are realised]\label{cor:boundedexistence} 
   Suppose that $\FF$ is a finite collection of fundamental domains. Then there is 
    $R>0$ with the following property. 

    If 
     $\s$ is an external address whose entries are all in $\FF$, then 
     $J^0_{\s}$ contains an unbounded connected set $X$ which contains a point
      of modulus at most $R$.
 \end{cor} 
 \begin{proof}
   Pick a base-point $\zeta_0\in\WW$. For $F\in\FF$, let $\zeta_F$ be the preimage of $\zeta_0$ in $F$. Then
     there is a constant $C$ such that $\distW(\zeta_0, \zeta_F)\leq C$ for all $F\in\FF$ (when defined). 

    If $\s$ is an external address in $\FF^{\N_0}$, we can set $z_n\defeq \zeta_0$ for all $n\geq 0$, and apply
     Theorem~\ref{thm:unboundedgeneral}. We obtain an unbounded connected set 
     $X$ with $\distW(\zeta_0, X)\leq 2\max(2\pi,C)$. If $R$ is sufficiently large (depending only on $C$), then
      $X$ contains a point of modulus at most $R$. 
 \end{proof} 

\subsection*{Uniform escape to infinity} 
  To complete the proof  of Theorem~\ref{thm:unboundedsets}, we consider the question of
    uniform escape to infinity on an unbounded connected subset of $J^0_{\s}$.
  Recall that by definition of $J^0_{\s}$ a point $z$ belongs to $J^0_{\s}$ if and only 
   if  $z_n\in \unbdd{F}_n$ for all $n$, where $\s=F_0 F_1\ldots$.
\begin{lem}[Uniform escape on unbounded connected sets]\label{lem:uniformescape}
   Let $\s = F_0 F_1 F_2 \dots$ be an external address, and suppose that
     $X\subset J^0_{\s}$ is unbounded and connected.  Furthermore, assume that there is a sequence $(z_n)_{n=0}^{\infty}$ (not necessarily an orbit of $f$) 
     such that $z_n\in \unbdd{F}_n$, and
      such that $\distW(z_n, f^n(X))\to\infty$. 
   Then $f^n|_{\overline{X}}\to\infty$ uniformly. 
\end{lem}
\begin{proof} 
   Let $\zeta_0\in \WW$ be any base point. We may assume that
     $\distW(z_n, \zeta_0)\to\infty$. Indeed, set $\eta_n\defeq \distW(z_n, f^n(X))$ and let $(\xi_n)_{n=0}^{\infty}$ 
     be any sequence in $\WW$ with 
      \[ \distW(\xi_m , \zeta_0)\to \infty \qquad\text{and}\qquad  
                \distW( \xi_m , \zeta_0) \leq \eta_n/3. \] 
      Now define 
        \[ \tilde{z}_n \defeq \begin{cases} z_n &\text{if }
                  \distW(z_n, \zeta_0) > \eta_n/3 \\
                   \xi_n &\text{otherwise}.\end{cases}\]
      Clearly $\dist(\tilde{z}_n, \zeta_0)\to\infty$ and $\distW(\tilde{z}_n, f^n(X)) \geq \eta_n/3\to\infty$, as desired. 

   Let $\Gamma_n$ be the union of cross-cuts associated to  
     $z_n$ by Lemma~\ref{lem:crosscuts}. Then, for sufficiently large
     $n$, $f^n(\overline{X})$ is disjoint from $\Gamma_n$, and hence belongs to $\WW^{\Gamma_n}$. 
     Since $\lvert z_n\rvert \to\infty$ by Observation~\ref{obs:absolutevaluecylindrical}, it follows that 
     $f^n(\overline{X})\to\infty$ uniformly, as required. 
\end{proof}

\begin{proof}[Proof of Theorem~\ref{thm:unboundedsets}]
   We first prove~\ref{item:terminal}, so let  $X_1, X_2\subset J_{\s}^0$ be closed, unbounded and connected
     with $X_1\not\subset X_2$.  

   Let $z_0\in X_1\setminus X_2$ and set $z_n\defeq f^n(z_0)$. Then $\dist(z_n , f^n(X_2))\to\infty$ 
     by~\eqref{eqn:expansionW}. In particular, $f^n|_{X_2}\to\infty$ uniformly by Lemma~\ref{lem:uniformescape}. 

   Let $n_0$ be sufficiently large that $\distW(z_n, f^n(X_2)) > 2\pi$ for $n\geq n_0$.
      Let $\Gamma_n$ be the union of crosscuts associated to $z_n$ by Lemma~\ref{lem:crosscuts}.   
      Then $f^n(X_2)\subset \WW^{\Gamma_n}$ for $n\geq n_0$. Since $f^n(X_1)$ connects
       $z_n$ to $\infty$, we have 
       $\dist(f^n(X_1), f^n(\zeta))\leq 2\pi$ for all  $\zeta\in X_2$ and all  $n\geq 0$. 
      Again by~\eqref{eqn:expansionW}, we have
       $\dist(X_1, \zeta)=0$, and hence $\zeta\in X_1$, as required. 

   Now let us prove~\ref{item:unboundedsets}, so 
     suppose that $z_0\in J_{\s}^0$. By Corollary~\ref{cor:unboundedsets}, there is an unbounded closed connected set
     $X\subset J_{\s}^0$ with $\distW(X,z_0)\leq 4\pi$. We may assume that $z_0\notin X$. Indeed, otherwise
     we let $\eps$ be sufficiently small and
     replace $X$ by an unbounded connected component of $X\setminus D_{\eps}(z_0)$
     that intersects $\partial D_{\eps}(z_0)$.

    Now set $z_n\defeq f^n(z_0)$. By~\eqref{eqn:expansionW}, we have $\distW(f^n(X), z_n)\to\infty$, 
      and hence it follows from Lemma~\ref{lem:uniformescape} that $f^n|_X\to\infty$ uniformly.
      This completes the proof of~\ref{item:unboundedsets} of Theorem~\ref{thm:unboundedsets}. 
  
     Part~\ref{item:boundedexistence} 
      follows directly from  Corollary~\ref{cor:boundedexistence}. 
      To prove~\ref{item:boundedescape}, observe first that equality of the
       sets  in~\eqref{eqn:boundedescape} holds as soon as $R$ is sufficiently large. Indeed, suppose that 
        $F,F'\in\mathcal{F}$ and that $z\in\overline{F}$ maps to some point in $\overline{F'}$ of modulus
        at least $R$. Since $f(z)\in \WW$, we must in fact have $z\in F$, and additionally $z\in\unbdd{F}$ if $R$ 
        is sufficiently large.

       Now let $\zeta_0$ and  $(\zeta_F)_{F\in\FF}$ be defined as in the proof
       of Corollary~\ref{cor:boundedexistence}.   Let $R$ be sufficiently large such that any point $z\in\WW$ of modulus
       at least $R$ has $\distW(\zeta_0, z)\geq 3\max(C,2\pi)$. 

Let $z$ be a point whose orbit is contained 
        in $\bigcup_{F\in\FF} \overline{F}$, and furthermore $\lvert z\rvert \geq R$. Then it follows that 
         \[ \distW(\zeta_0, f(z)) \geq 
               2\cdot ( \distW(\zeta_0 , z) - \max(C,2\pi)) \geq \frac{4}{3} \distW(\zeta_0, z). \] 
      It follows by induction that $\distW(\zeta_0, f^n(z))\to\infty$ uniformly in $n$, and the
       claim follows. 
\end{proof} 

 \subsection*{Exponentially bounded addresses}
  As noted above, the results on bounded addresses can be generalised to certain \emph{unbounded} addresses.
    While we do not require this fact for this paper, we shall record it for future reference. 

  \begin{defn}[Exponentially bounded addresses]\label{defn:expbounded} 
    Let $\zeta_0\in \WW$ be an arbitrary base point. For
     any fundamental domain $F$, let $\zeta_F$ be the unique preimage of $\zeta_0$ in $F$. 

    We say that an infinite external address $\s$ is \emph{exponentially bounded} 
      if there exists a {positive real} number $T$ with the following property. For all $n\geq 0$, 
      if $\zeta_{F_n}\in \unbdd{F}_n$, then 
          \begin{equation} \distW(\zeta_0 , \zeta_{F_n}) \leq \exp^n(T). \end{equation}
  \end{defn} 
 \begin{remark}
    If $\xi$ is another base-point, then it follows from~\ref{eqn:expansionW} that
      $\distW( \zeta_F , \xi_F)$ is uniformly bounded (where defined) for all fundamental domains $F$.
      Thus it follows that the definition of exponentially bounded addresses is independent of the choice of base point
      $\zeta_0$. 

    Exponentially bounded addresses were  defined previously for exponential maps \cite{expescaping} and
      cosine maps \cite{cosescaping}; it is easy to see that in these cases the definition agrees with ours.
      For these families, the class of exponentially bounded addresses $\s$ agrees precisely with those for which 
      $J_{\s}^0\neq\emptyset$. This is no longer true for general $f$, even when $f$ has finite order of growth; 
      see~\cite{simonannalasse}. 

    In \cite[Corollary~B']{BK07}, it is shown that $J_{\s}^0\neq\emptyset$ for a certain class of 
     addresses, namely those whose orbits remain within finitely many tracts, and such that the 
     ``index'' of the corresponding fundamental domains within each tract does not grow faster than an iterated
     exponential. It is easy to see that such addresses are exponentially bounded in our sense, but 
     the converse is not the case. (For example, our definition  allows for addresses taking values in
     fundamental domains that lie in infinitely many different tracts.) 
 \end{remark}

  We now show that~\ref{item:boundedexistence} and~\ref{item:boundedescape} of Theorem~\ref{thm:unboundedsets}
    can be extended to exponentially bounded addresses as follows. 

  \begin{prop}[Exponentially bounded addresses are realised]\label{prop:exponentiallybounded} Let  $\s$ be an  exponentially bounded address. Then $J_{\s}^0\neq\emptyset$.

     More precisely, there is a number $R>0$, depending only on the base-point $\zeta_0\in \WW$ and 
       $T>0$, with the following property. If $\s$ is exponentially bounded for this choice of $\zeta_0$ and $T$,
       then  $J_{\s}^0$ contains an unbounded connected set $X$ on  which the iterates tend to infinity uniformly,
       and which contains a point of modulus $R$. 

    Conversely, the iterates of $f$ tend to infinity uniformly on 
       \begin{equation}\label{eqn:expboundedunion} \bigcup_{\s} J_{\s}^0\setminus D_R(0), \end{equation}
     where the union is taken over all external addresses $\s$ as above. 
  \end{prop}     
   \begin{proof}  
     We shall use an expansion estimate \cite[Lemma~3.1]{R3S}, which is stronger
      than~\eqref{eqn:expansionW} at large distances. (Compare also~\cite[Lemma~3.3]{BK07}.) We will use
      this estimate in the following form, which follows easily from the version stated in  \cite{R3S}: 
      There are constants $C_1,C_2>0$ with the following property.
       If $F$ is a fundamental  domain and $\zeta_1,\zeta_2\in \unbdd{F}$ with $\distW(\zeta_1,\zeta_2)\geq C_1$, then
              \begin{equation}\label{eqn:expgrowth}
                 \distW(f(\zeta_1), f(\zeta_2)) \geq \exp(C_2 \cdot \distW(\zeta_1, \zeta_2)).
              \end{equation}

         Now fix $\zeta_0$ and  $T>0$, and denote by $\mathcal{S}$ the set of all
          addresses that satisfy Definition~\ref{defn:expbounded} for these choices. We may 
         assume without loss of generality that $\zeta_0$ is chosen sufficiently large to ensure that
          $\zeta_F\in \unbdd{F}$ for every fundamental domain $F$. (To this end, we may need to increase $T$, but 
          only by a finite amount according to the remark following 
          Definition~\ref{defn:expbounded}.) 
          Define $E(t)\defeq \exp(C_2 \cdot t)$. 
          By basic properties of exponential growth, there is $\tilde{R}>0$ such that the following hold for all 
               $x\geq \tilde{R}/3$. 
          \begin{align}
                \label{eqn:exp1}E(x) &> 3E(x/2) > E(x/2) > 2x, \qquad\text{and}  \\
                \label{eqn:exp2}E^n(x) &\geq  E^n(\tilde{R}/3) > \exp^n(T) + 2\pi. 
          \end{align}
        We choose $R$ sufficiently large to ensure that $\distW(z,\zeta_0) > \tilde{R}$ whenever
            $\lvert z \rvert \geq  R$. 
     
       Suppose that  $\s\in\mathcal{S}$, and let $z\in J_{\s}^0$ such that  $t\defeq \distW(z,\zeta_0) \geq \tilde{R}$. 
         We claim that 
            \begin{equation}\label{eqn:iteratedexpgrowth} \distW( f^n(z) , \zeta_0) \geq 3E^n(t/3).  \end{equation}
          Indeed, this is trivial for $n=0$, and if the claim holds for $n$, then 
            \[ \distW(f^n(z) , \zeta_{F_n}) \geq
               \distW(f^n(z) , \zeta_0)  - \exp^n(T) \geq  2E^n(t/3). \]
        Hence, by~\eqref{eqn:expgrowth} and~\ref{eqn:exp1}, 
           \[ \distW(f^{n+1}(z), \zeta_0) = \distW(f^{n+1}(z), f(\zeta_{F_n})) \geq E(2E^n(t/3)) \geq 3E^{n+1}(t/3). \]
        The claim follows by induction. In particular,
              \[ \distW(f^n(z),\zeta_0) \geq 3E^n(\tilde{R}/3) \geq  \exp^n(T).
                    \]
         Since this applies to all points in the union~\eqref{eqn:expboundedunion}, the second 
         claim of the proposition follows. 

       Now let us prove the first. Set $t\defeq \tilde{R}/3$. For $n\geq 0$, choose $z_n\in \unbdd{F}_n$ with
         $\distW(z_n , \zeta_0) = E^n(t)$. 
         Such a point exists because $\dist(\zeta_{F_n} , \zeta_0) < E^n(t)$, and $\zeta_{F_n}$ can be 
         connected to infinity within  $\unbdd{F}_n$. 
  
       For $n\geq 0$, let 
        $\Gamma_n$ be the union of cross-cuts from Lemma~\ref{lem:crosscuts}, applied to
         $z_n$, and let  $A_n^k$, for $n,k\geq 0$, 
         be defined precisely as in the proof of Theorem~\ref{thm:unboundedgeneral}.       

     \begin{claim}
       $\distW(A_n^k , \zeta_0) \leq 3E^n(t)$ for all $n,k\geq 0$. 
     \end{claim}
     \begin{subproof}
       If $A_n^k\cap \Gamma_n \neq \emptyset$, then the claim follows by choice of 
        $z_n$. In particular, this is always the case for $k=0$. 

      So suppose that $k\geq 0$ is such that the claim holds for all  $n\geq 0$, and that 
        $n$ is such that $A_n^{k+1}\cap \Gamma_n = \emptyset$. Then, by
        definition,  $f(A_n^{k+1}) = A_{n+1}^k$. By the inductive hypothesis, the latter
        contains a point $w$ with $\distW(w, \zeta_0) \leq 3E^{n+1}(t)$. 
        Set $\tilde{w} \defeq (f|_{F_n})^{-1}(w) \in A_{n+1}^k$. Then, 
        by~\eqref{eqn:expgrowth} and~\eqref{eqn:exp1}, 
                   \[ \distW( \tilde{w} , \zeta_{F_n}) \leq E^{-1}(3 E^{n+1}(t)) \leq
                             2E^n(t). \] 
       So
          \[ \distW(\tilde{w} , \zeta_0)  \leq \distW(\tilde{w} , \zeta_{F_n}) + \distW(\zeta_{F_n} , \zeta_0) \leq
                  2E^n(t) + \exp^n(T) \leq 
                   3E^n(t) \]
      by~\eqref{eqn:exp2}, as claimed.
     \end{subproof} 

     In particular, the set $\bigcap_{k\geq 0} A_0^k$ has a connected component $X$ containing 
       a point with $\dist(A_n^k , \zeta_0) \leq 3t = \tilde{R}$, and hence of modulus
       at most $R$. Since $X$ is connected and unbounded, it also contains a point of modulus exactly $R$. 
       The proof is complete. 
   \end{proof} 

\section{Hyperbolic expansion and fundamental tails}\label{sec:tails}
 For the remainder of the article, we shall specialise to the case where our 
     transcendental entire function $f$ has bounded postsingular set $\PP(f)$.
        
     This section collects some fundamental preliminary material concerning
     these functions. We begin by noting a global expansion property away from the postsingular set,
    and then proceed to introduce the notation of \emph{fundamental tails}, which will later be used to study 
    the structure of the set of escaping points.

Recall that if $\Omega$ is a hyperbolic domain, 
  we denote by  $\rho_{\Omega}$   the density of the hyperbolic metric on $\Omega$.
  \begin{prop}[Hyperbolic expansion]\label{prop:expansion}
   For every transcendental entire function $f$, $\# \P(f) \geq 2$. 

   Now suppose that $\P(f)$ is bounded, let $\P$ be a compact forward-invariant set with
    $\P(f)\subset \P$, and let 
    $\Omega$ be the unbounded connected component of $\C\setminus\P$. 
   If $z\in V\defeq f^{-1}(\Omega)\subset \Omega$, 
     then $f$ is strictly expanding at $z$ in the hyperbolic metric of $\Omega$. Moreover, this expansion
     factor tends to infinity as $z\to\infty$ in $V$. That is, 
        \begin{equation}\label{eqn:Omegaexpands}
          \lVert\Deriv f(z)\rVert_{\Omega} \defeq \lvert f'(z)\rvert \cdot \frac{\rho_{\Omega}(f(z))}{\rho_{\Omega}(z)} > 1, \end{equation}
    and 
      \begin{equation}\label{eqn:strongexpansion}
        \lVert \Deriv f(z)\rVert_{\Omega} \to\infty \qquad\text{as}\qquad \lvert z\rvert \to \infty.
      \end{equation}

    In particular, for every $\eps>0$ there is $\Lambda>1$ such that
      $\lVert \Deriv f(z)\rVert_{\Omega}\geq \Lambda$ whenever $z\in V$ with 
      $\dist(z, \P)\geq \eps$. 
  \end{prop}

   This result was proved, for a hyperbolic entire function  $f$ and a certain choice of $\P$, 
    in \cite[Lemma 5.1]{Re09}. The same proof goes through
    whenever $\P$ intersects the unbounded connected component of $\C\setminus S(f)$. This
    can always be ensured by adding a periodic orbit to $\P$ that intersects this component,      
    which
    is sufficient for all our purposes. 
    However, for completeness and future reference, we shall
     prove the result for general $\P$, with a slightly
     simpler proof than that given in \cite{Re09}. 
   To do so, we use the following simple fact.

 \begin{lem}[Preimages in annuli]\label{lem:preimages}
   Let $f$ be an entire transcendental function which is bounded on an unbounded connected set. Let $z_1,z_2\in\C$.
     Then, for all $C>1$, and all sufficiently large $R$, 
      $f^{-1}(\{z_1,z_2\})$ contains a point of modulus between $R/C$ and $C\cdot R$. 
 \end{lem}
 \begin{proof}
   Let $C>1$ and set $A\defeq \{z\in\C\colon 1/C < \lvert z \rvert <C\}$.
     Suppose by contradiction that $R_n\to\infty$ is a sequence such that the functions
      \[ g_n\colon A\to\C; \qquad z\mapsto f(R_n\cdot z) \]
     omit both {$z_1$ and  $z_2$}. By Montel's theorem, this sequence of functions is normal, and hence 
      converges locally uniformly, possibly after restricting to a subsequence. By assumption,
      $\limsup \min_{|z|=1} g_n(z) <\infty$, and hence the limit function is holomorphic. But this implies
      that $f$ remains bounded on the circle of radius $R_n$ as  $n\to\infty$. Hence $f$ is bounded by 
      the maximum principle and hence constant by Liouville's theorem, a contradiction.
  \end{proof}

 \begin{proof}[Proof of Proposition~\ref{prop:expansion}]
    The
    fact that $\# \P(f) \geq 2$ is well-known: Otherwise, $f$ would be a self-covering of a 
      punctured plane,     and hence conformally conjugate to $z\mapsto z^d$ for some $d$.  
       However, $f$ is transcendental. 

    So $\Omega$ is indeed a hyperbolic domain. 
              Since $f\colon V\to\Omega$ is a covering map, and hence
      a local isometry with respect to the hyperbolic metrics of $V$ and $\Omega$,     
     $\rho_V(z)=\rho_\Omega(f(z))\lvert f'(z)\rvert$ for all $z\in V$.

   The open mapping theorem implies that every connected component
    of $V=f^{-1}(\Omega)$ is unbounded 
     (compare e.g.\ \cite[Lemma~4.3]{alhabib-rempe15}). 
     Hence, by forward-invariance of $\P$, we have $V\subset \Omega$.
    
       Since $\# \mathcal{P} \geq 2$, and by  Picard's theorem, $\C\setminus V = \C\setminus f^{-1}(\P)$ contains
       points of arbitrarily large modulus. Hence 
      $V\subsetneq \Omega$.  By Pick's theorem, it follows that
    \[ \lVert \Deriv f(z)\rVert_{\Omega} = \lvert f'(z)\rvert \cdot \frac{\rho_\Omega(f(z))}{\rho_{\Omega}(z)}
        = \frac{\rho_V(z)}{\rho_{\Omega}(z)} > 1 \]
   for all $z\in V$. This establishes the first claim.
     
   Furthermore, since $f\in\BB$, there is an unbounded connected set on which $f$ is bounded. (For example,\ the
    boundary of one of the  tracts of $f$). By Lemma~\ref{lem:preimages}, there is a sequence 
    $(c_n)_{n\geq 0}$ in $\Omega\setminus V$ such that $c_n\to\infty$ and 
    $\lvert c_{n+1}/c_n\rvert \leq 2$. (We can even ensure
    $\lvert c_{n+1}/c_n\rvert \to 1$ by letting the constant $C$ in Lemma~\ref{lem:preimages} tend to $1$, but
   do not require this here.)   
    Hence $1/\rho_V(z) = O(\lvert z\rvert)$ as $z\to\infty$ in $V$.
    (See \cite[Lemma 2.1]{Re09} and compare also  \cite[Proposition~3.4]{wandering} and \cite{mindaquotients}.) 
     On the other hand, 
    $\rho_{\Omega}(z) = O(1/(\lvert z\rvert \log \lvert z \rvert))$ as $z\to\infty$. The claim follows. 
 \end{proof} 

\begin{cor}[Sets remaining in $\Omega$]\label{cor:remaininginOmega} Let $f\in\B$. 
  Let $\Omega, \P$ be as in Proposition~\ref{prop:expansion}, and suppose that $U\subset \Omega$ is open with $f^n(U)\subset \Omega$ for all $n$. Then 
    $\dist(f^n(z),\P)\to 0$ uniformly on compact subsets of $U$. 

  In particular, if $z_0\in\Omega$ with $f^n(z_0)\in\Omega$ for all $n$, and $\limsup \dist(f^n(z_0),\P)>0$, then $z_0\in J(f)$. 
\end{cor} 
 \begin{proof}  Note that $(f^n|_U)$ form a normal family by Montel's theorem,
     and hence $U\subset F(f)$. Therefore the first claim follows from the second: 
     if $\dist(f^n(z),\P)\to 0$ for all $z\in U$, then this convergence is automatically
     uniform on compact subsets of $U$. 
    
     So suppose that $z_0\in \Omega$, and that
      there exist $\eps>0$ and an increasing sequence $(n_k)_{k=0}^{\infty}$ such that 
        $\dist(f^{n_k}(z_0),\P)\geq \eps$. We must show that $z_0\in J(f)$.
        
    It follows from Lemma~\ref{lem:preimages} that 
        $\dist_{\Omega}(z, f^{-1}(\P))$ remains bounded as $z\to\infty$
        in $\Omega$. Hence 
        \[ \delta \defeq \sup_{k} \dist_{\Omega}(f^{n_k}(z_0),f^{-1}(\P)) < \infty. \]
   
      So we can connect $f^{n_k}(z_0)$ to a point of $f^{-1}(\P)$ by a curve $\gamma_k$ 
         of hyperbolic length at most $\delta$ in  
         $\Omega$. Pulling back $\gamma_k$ under $f^{n_k}$, we obtain a curve
         $\alpha_k$ connecting $z_0$ to a point $w_k \in f^{-(n_k+1)}(\P)$. 
         By Proposition~\ref{prop:expansion}, each of the curves
         $\alpha_k^j = f^{n_j}(\alpha_k)$, for $j<k$, has hyperbolic 
         length at most $\delta$. Hence these curves stay a uniform distance away from
         $\P$, and again by Proposition~\ref{prop:expansion}, there is $\lambda>1$ such 
            that 
           \[ \lVert \Deriv f (z) \rVert \geq \lambda  \]
           for all $k$, all $j < k$, and all $z\in \alpha_k^j$. 

       So in fact $\alpha_k$ has hyperbolic length at most $\delta \cdot \lambda^k$, 
         and $w_k\to z_0$ as $k\to\infty$. 
         converging to $z_0$. 
         But $f^{n_{k+1}}(w_k)\in \P$, and $\dist(f^{n_{k_1}}(z_0),\P)\geq \eps$. 
         Therefore the family of iterates of $f$ is not equicontinuous at $z_0$,
           and $z_0\in J(f)$, as desired.
\end{proof} 
\begin{remark}
  The result can also be proved by appealing to the classification of Fatou components. 
\end{remark}

 \subsection*{Fundamental tails}  
   Let $f$ be an entire transcendental function with bounded postsingular set,
      and denote the unbounded connected component of $\C\setminus\P(f)$ by $\Omega$.  
       (In everything that follows, we could more generally let 
         $\Omega$ be as in  Proposition~\ref{prop:expansion}; 
         i.e.\ the unbounded connected component of $\C\setminus \P$ where $\P$ is a 
         forward-invariant compact   set containing $\P(f)$. However, we shall not require 
         this extra generality.)   
      
     Let $D$ and $\gamma$ be as in Section~\ref{sec:unboundedsets}.
      We may additionally
       assume that $D$ is chosen sufficiently large to ensure that 
         $\P(f)\subset D$ and that
       \begin{equation}\label{eqn:Omeganormalised}
         \lVert \Deriv f(z)\rVert_{\Omega} \geq 2  
      \end{equation} 
    whenever $f(z)\notin D$ (recall \eqref{eqn:strongexpansion}). 
         
         Recall that \emph{fundamental  domains} are the 
         connected components of the
          preimage of $\WW = \C\setminus (\overline{D}\cup \delta)$ under $f$.
        All concepts that follow depend a priori 
     on this choice of fundamental domains, i.e.\ on the choice of the initial configuration consisting of
      $D$ and $\delta$. However, it turns out that this choice is not essential. {(Compare Observation~\ref{obs:filamentsindependence}.)}
 
  The postsingular set $\P(f)$ is contained in $D$, and  the image of any fundamental domain 
    is contained in  $\C\setminus \ov{D}$. Hence the closure of a fundamental domain 
    does not intersect the postsingular set. 
     It follows that for any fundamental domain $F$,  any $n\geq 0$,  and any connected component $\tail$ of $f^{-n}(F)$, $\tail$ is a Jordan domain
    on the Riemann sphere whose boundary contains $\infty$ and whose closure is mapped homeomorphically onto  $\ov{F}$ by $f^{n}$. Moreover, $f^k(z)\to\infty$ as $z\to\infty$ in $\tail$, for all $k\geq n+1$. 
    
 \begin{defn}[Fundamental tails]
   Let $n\geq 1$. 
     A connected component $\tail$ of $f^{-n}(\WW)$ is called a  \emph{fundamental tail of level $n$}. 
     In particular, the fundamental tails of level $1$ are precisely the fundamental domains of $f$. 
 \end{defn}

\begin{prop}[Facts about fundamental tails]\label{prop:facts}
  Fundamental tails of the same level are disjoint. Moreover, for every $N$ and any compact set $K$,
    there are at most finitely many fundamental tails of level at most $N$ that meet $K$.

   If $\tail$ is a fundamental tail of level $n>1$, then $f(\tail)$ is a fundamental tail of level $n-1$. 
\end{prop}
\begin{proof}
    Recall that the  fundamental tails of level  $N$ are precisely the connected components of $f^{-N}(\WW)$. 
     Hence they are pairwise disjoint,
     and the second claim follows from Lemma~\ref{lem:preimageK}.
     The final claim holds by definition. 
\end{proof}

  \begin{lem}[Fundamental domain associated to fundamental tail]\label{lem:fdtotails}
    Let $\tail$ be a fundamental tail of level $n$. Then there is a unique fundamental domain
    $F$ such that $F\cap \tail$ is unbounded. In fact, if $n>1$, then 
    $F$ contains all sufficiently large points of $\overline{\tail}$. 
  \end{lem}
\begin{proof}
 We proceed by induction. The claim is trivial for $n=1$.
  Now suppose that  $n>1$. By induction, there is a unique fundamental domain
    $F_1$ whose intersection with the fundamental tail $f(\tail)$, of level $n-1$, is unbounded. 

    Let $A_1$ be the unbounded connected component of $\overline{F_1}\setminus D$, and let 
     $A_2$ be the unbounded connected component of $A_1\cap \overline{f(\tail)}\setminus \overline{D}$. 
     Then $f(\tail)\setminus A_2$ is bounded. Indeed, this is clear if $n=2$, and otherwise follows from
     the inductive hypothesis. 

    Recall that $f\colon \overline{\tail} \to \overline{f(\tail)}$ is a homeomorphism. Since $A_2\subset \WW$, there is 
      a unique fundamental domain $F$ containing $A\defeq f^{-1}(A_2)\cap \overline{\tail}$,
      and $\overline{\tail}\setminus A\subset \overline{\tail}\setminus F$ is bounded, as claimed. 
\end{proof}

 It follows from the above that we can associate natural symbolic sequences to fundamental 
   tails. 

 \begin{defn}[Addresses of fundamental tails]
   Let $\tail$ be a fundamental tail of level  $n$, and denote by  
     $F_k(\tail)$ the unique fundamental domain whose intersection with
     the fundamental tail  $f^k(\tail)$ is unbounded.
     We call the finite sequence
      $\s= F_0(\tail) F_1(\tail) \dots F_{n-1}(\tail)$, of length  $n$,  the \emph{(finite) external address} of~$\tail$.
 \end{defn}
 
Conversely, we can construct a fundamental  tail  having an arbitrary prescribed (finite) address by taking repeated pull-backs
   along the correct branches. Recall that a sequence $\s^1$, say of length $n$,
     is a \emph{prefix} of another sequence $\s^2$ of length $m\geq n$ if
     the first $n$ entries of  $\s^2$ coincide with those of $\s^1$. 

\begin{deflem}[Tails at a given address]\label{deflem:pullbacks}
  Let $\s= F_0 F_1 \dots$ be a finite or infinite sequence of fundamental domains having length at least $n\geq 1$.
    Then there is a unique fundamental     tail $\tail$ of level $n$ having address $\tilde{\s} \defeq F_0 F_1 \dots F_{n-1}$. We denote this fundamental 
    tail by $\tail_n(\s)$. We also define the inverse branches
     \[ f^{-n}_{\s} \defeq (f^n|_{\tail_n(\s)})^{-1}\colon \WW \to \tail_n(\s). \]
\end{deflem}
\begin{proof} The proof is by induction on the level $n$ of the tails.  For $n=1$ and for every fundamental domain $F_0$, the fundamental  tail of level 1 and address $F_0$ is the fundamental domain $F_0$ itself, which is  unique and has address $F_0$ by definition.  
  Now let $\s$ be in the claim and let $\tail=\tail_{n-1}(\sigma(\s))$ be the fundamental tail of level $n-1$ 
  and address  $F_1\ldots F_{n-1}$. This tail exists and  is unique by the inductive hypothesis.
   Let $R$ be sufficiently large,
	   and let  ${\tail^{1}}$ be the unique
    unbounded    connected component of $\tail\setminus \overline{D_R(0)}$.  By Lemma~\ref{lem:fdtotails}, 
    ${\tail}^{1}$ is contained in the fundamental domain $F_1=F_0(\tail)$ for sufficiently large $R$. 
     Hence, if we additionally assume that $R$ is greater than the radius of $D$, we have
     $\tail^{1}\subset f(F_0) = \WW$.
    So there is a unique connected component $\tilde{\tail}^{1}$    of $f^{-1}(\tail^{1})$ contained in $F_0$, 
      and a unique connected component
     $\tilde{\tail}$   of $f^{-1}(\tail)$ containing $\tilde{\tail}^{1}$.
     Then $F_0(\tilde{\tail})=F_0$; i.e.,\ $F_0$ is the initial entry in the 
     address of the fundamental tail $\tilde{\tail}$, which hence has address $\s$.
\end{proof}

 The following are immediate consequences of the preceding results and definitions. 
\begin{obs}\label{Obs:Nested tails}
  Let $\tail$ be a fundamental tail of level $n$, and let $\s$ be the address of  $\tail$.  
    Then the address of $f(\tail)$ is $\sigma(\s)$, where $\sigma$ denotes the shift map.

  Suppose that $\tail^1$ and $\tail^2$ are fundamental tails of levels $n_1$ and $n_2$, with $n_1 \geq n_2$. 
    Let $\s^1$ and $\s^2$
    be the addresses of $\tail^1$ and $\tail^2$, respectively.  
     Then  $\tail^1 \cap \tail^2$ 
     is unbounded if and only if $\s^1$ is a prefix of $\s^2$. In this case, if addditionally $n_1>n_2$,
     all sufficiently large points of $\overline{\tail^1}$ lie in $\tail^2$. 
\end{obs}

\section{Filaments}\label{sec:filaments}
 Maintaining the same notation as in the previous section, we now define and study 
    the central objects of this article: 
     filaments. Recall that $\P(f)$ is bounded; the main goal of this section is to show that, under this assumption,
     each of the sets $J_{\s}^0$ defined 
      in Section \ref{sec:unboundedsets} can be consistently extended to a larger~-- and, in a certain sense, maximal~--
      set $J_{\s}$. The intersection of $J_{\s}$ with the escaping set $I(f)$ forms
      the \emph{filament} $G_{\s}$ at address $\s$.

     As we shall see, each filament is an unbounded connected set of escaping points, 
      the escaping set can be written as the union of filaments, and via their external addresses
      the collection of filaments is endowed with a natural combinatorial structure. 
      Furthermore, the definition of filaments does not depend on the initial choices made in the 
      construction of fundamental domains.  Together these facts indicate that 
      filaments can indeed be considered a natural generalisation of ``hairs'' or ``rays''. 

 For a fundamental domain $F$, recall that $\unbdd{F}$ denotes the unbounded connected component of
    $F\setminus\overline{D}$. We extend this definition to fundamental tails as follows.

\begin{defn}[Unbounded parts of tails]
   Let $\tau$ be a fundamental tail of level $n$. 
     We define $\unbdd{\tail}$ to be the unbounded connected component of
       $\tail\setminus f^{-(n-1)}(\overline{D})$. In other words, if $F=f^{n-1}(\tail)$, then 
       $\unbdd{\tail}$ is the component of $f^{-(n-1)}(\unbdd{F})$ contained in $\tail$.
\end{defn}
Observe that, if $\s$ is an external address and $n\geq 2$, then $\unbdd{\tail}_n(\s)$ 
   is precisely the unbounded connected component of $\tail_n(\s)\cap \tail_{n-1}(\s)$.

\begin{defn}[Filaments]\label{defn:filament}
   Let $\s$ be an (infinite) external 
    address. We say that a point
      $z\in \C$ \emph{has external address $\s$} if $z\in\unbdd{\tail}_{n+1}(\s)$ for all sufficiently large~$n$.

   We denote the set of all points $z\in \C$ having external address $\s$ by $J_{\s}$. The 
    \emph{filament} $G_{\s}$ is defined to be $G_{\s}\defeq J_{\s}\cap I(f)$.
\end{defn}  
\begin{remark}
   If $z\in J_{\s}$, then  $f^n(z)$ belongs to $\unbdd{F}_n$ for all sufficiently large $n$.
    In particular,  $z\in J(f)$ by Lemma~\ref{lem:J}. 
    Observe, however, that $J_{\s}$ is not closed in general.

Note that other notions of points having external address $\s$ appear in the literature;
    for example, in~\cite{eremenkoproperty}, $z$ is said to have address $\s=F_0 F_1 \dots$ if 
    $f^n(z)\in F_n$ for all $n\geq 0$. The advantage of the above definition, in the context of
    postsingularly bounded functions, is that we shall see that every escaping
    point has an external address (Corollary~\ref{cor:escapingpointsarefilaments}), and that this address 
    is in a certain sense independent of the initial choice of fundamental domains (Observation~\ref{obs:filamentsindependence}).
\end{remark}

\begin{lem}[Properties of addresses and filaments]\label{lem:addressproperties} 
   Suppose that $\s$ is an external address and $z\in  J_{\s}$. Then 
   \begin{enumerate}[(a)]
      \item $z\notin J_{\tilde{\s}}$ for $\s\neq \tilde{\s}$.\label{item:Jsdisjoint}
     \item The point $f(z)$ has address $\sigma(\s)$.\label{item:imageaddress}
     \item If $w\in f^{-1}(z)$, then $w$ has an address in $\sigma^{-1}(\s)$, and every such address 
       is realised by exactly one element of  $f^{-1}(z)$. 
     \item The restriction $f\colon J_{\s} \to J_{\sigma(\s)}$ is a continuous bijection. 
  \end{enumerate}
\end{lem}
 \begin{proof}
     The first claim is trivial since fundamental tails of a given level are disjoint.
      Now write $\s = F_0 F_1 F_2 \dots$; observe that 
      $f({\unbdd{\tail}}_{n+1}(\s)) = \unbdd{\tail}_n(\sigma(\s))$
       and that $f\colon \unbdd{\tail}_{n+1}(\s) \to \unbdd{\tail}_n(\sigma(\s))$ 
      is a conformal isomorphism.
 
      By definition, $z$ belongs to $\unbdd{\tail}_n(\s)$ 
      for all sufficiently large $n$; say for 
      $n\geq n_0$. 
      Let $n\geq \max(1,n_0-1)$. Then $f(z)\in f({\unbdd{\tail}}_{n+1}(\s)) = \unbdd{\tail}_n(\sigma(\s))$,
      as claimed in~\ref{item:imageaddress}.

    Now let $w\in f^{-1}(z)$. For $n \geq  n_0$,
    we have $z = f(w) \in \unbdd{\tail}_n(\s)$. In particular, for $n\geq n_0+1$, 
     there is a fundamental tail $\tailtilde_{n}$ such that $f(\tailtilde_{n}) = \tail_{n-1}(\s)$ and
     $w\in\unbdd{\tailtilde}_n$. Recall that $\unbdd{\tail}_n(\s)\subset \tail_{n-1}(\s)$, and that
     the intersection $\unbdd{\tailtilde}_{n+1} \cap \tailtilde_n$ is non-empty (since it contains~$w$). 
     Hence $\unbdd{\tailtilde}_{n+1}\subset \tailtilde_n$ for all $n\geq n_0+1$. 
     That is, $\tailtilde_{n+1}$ tends 
     to infinity within $\tailtilde_n$. So if $F$ is the fundamental domain whose intersection with $\tailtilde_{n_0+1}$ is 
     unbounded, then $\tailtilde_n$ tends to infinity in $F$ for $n\geq n_0+1$. It follows that 
     $\tailtilde_n = {\tail}_n(F\s)$. In particular, $w$ has address $F\s\in \sigma^{-1}(\s)$.

   Conversely, let $\tilde{\s} \in \sigma^{-1}(\s)$; that is, $\tilde{\s} = F \s$ for some fundamental domain $F$.
     Then the fundamental tail $\tailtilde_{n_0+1}$ of level  $n_0+1$ associated
     to the address $\tilde{\s} \defeq F\s$ is a component of $f^{-1}(\tail_{n_0})$. Hence there is 
     a unique element $w$ of $f^{-1}(z)$ in $\tailtilde_{n_0+1}$, and $w$ has address $\tilde{\s}$ as above. 

   The final claim follows from the previous two. 
  \end{proof}

 The following proposition establishes a connection between the sets
   $J_{\s}$ defined in  Definition~\ref{defn:filament}
   and the sets $J_{\s}^0$ studied in Section~\ref{sec:unboundedsets}.

 \begin{prop}[$J_{\s}$ and $J_{\s}^0$]\label{prop:externaladdresses}
    Let $\s$ be an (infinite) external address. If $z\in J_{\s}^0$, then
        $z\in \unbdd{\tail}_n(\s)$ for all $n\geq 1$. In particular, $J_{\s}^0\subset J_{\s}$. 
 \end{prop}
\begin{proof}
   This is essentially the content of the second paragraph of \cite[Proof of Theorem~1.1]{eremenkoproperty}. Since that proof is somewhat concise, and we are using a different terminology, we provide the details. 

   For $n\geq 1$, let $\tailtilde_n$ denote the fundamental tail of level $n$ containing $z$.
    By assumption,  $f^{n-1}(z)\in \unbdd{F}_{n-1}$, and hence $z\in \unbdd{\tailtilde}_n$. 

   We claim that $\unbdd{\tailtilde}_{n+1} \subset \tailtilde_n$; that is,
     $\tailtilde_{n+1}$ tends to infinity within $\tailtilde_n$. 
     Indeed, $\unbdd{\tailtilde}_{n+1}$ is a connected component of  $f^{-n}(\unbdd{F}_n)$ and
     $\tailtilde_n$ is a connected component of  $f^{-n}(\WW)$. Since $\unbdd{F}_n\subset \WW$ and 
     $\unbdd{\tailtilde}_{n+1}\cap \tailtilde_n \neq\emptyset$, the claim follows. 

   Inductively, $\tailtilde_n$ tends to infinity within $F_0=\tailtilde_1$ for each $n\geq 0$. 
    Applying this fact to $f^k(z)$, we see likewise that $f^k(\tailtilde_n)$ tends to infinity within $F_k$  
    for $n\geq k$. Thus, for any $n\geq 1$, we conclude that $\tailtilde_n =\tail_n(\s)$, and hence 
    $z\in \unbdd{\tailtilde}_n = \unbdd{\tail}_n(\s)$, as desired.
\end{proof} 

 \begin{cor}[Escaping points are organised in filaments]\label{cor:escapingpointsarefilaments}
    \mbox{}
   \begin{enumerate}[(a)]
      \item A point $z$ has an external address if and only if there is $n$ and an external address $\s$ such that 
               $f^n(z)\in J_{\sigma^n(\s)}^0$.\label{item:addresscharacterisation}
      \item\label{item:stayinglarge}
       There is a number $R>0$ with the following property. If $z\in J(f)$ is such that $\lvert f^n(z)\rvert \geq R$ for all
          sufficiently large $n$, then 
         $z$ has an external address $\s$. 
      \item Every escaping point $z\in I(f)$ has an external address $\s$, and hence belongs to a filament $G_{\s}$.\label{item:escapingpointshaveaddresses} 
   \end{enumerate}
 \end{cor}
\begin{proof}
  The ``only if'' direction of~\ref{item:addresscharacterisation} is immediate from Definition~\ref{defn:filament}. 
    On the other hand, if $z$ has the stated property, then $f^{n}(z)$ has address $\s$ by   
     Proposition~\ref{prop:externaladdresses}, and hence 
      $z$ also has an external address by Lemma~\ref{lem:addressproperties}.

  To prove~\ref{item:stayinglarge}, let $\FF$ be the set of fundamental domains that intersect $D$.
      {Recall that  $F\setminus \unbdd{F}$ is bounded for all
     $F\in\FF$, and that $\FF$ is finite by Lemma~\ref{lem:preimageK}. Now} fix
     \[ R > \sup\left\{ \lvert z\rvert \colon z\in \overline{D}\cup \bigcup_{F\in\FF} F\setminus\unbdd{F}\right\}. \]

    Suppose that $\lvert f^n(z)\rvert \geq R$ for all $n\geq n_0$. Then, for all $n\geq n_0$, 
     $f^n(z)$ belongs to some fundamental domain  $F$, and by choice of $R$, it must belong to $\unbdd{F}$. 
     Now $z$ has an address by~\ref{item:addresscharacterisation}. 

  Finally,~\ref{item:escapingpointshaveaddresses} is an immediate consequence of~\ref{item:stayinglarge}.  
\end{proof}  

\begin{rmk}
  Under the hypotheses of Proposition~\ref{prop:externaladdresses}, 
    $f^n(z)$ belongs to $\unbdd{F}_n$ 
      for all $n$. However,  
      it may well be that  $z$  belongs to a \emph{bounded} component of $\tail_k(\s)\setminus \overline{D}$ for 
      all sufficiently large $k$ 
       and similarly for all points on the forward orbit of $z$.  

      Indeed, typically when $I(f)$ does not consist of hairs, exactly this is the case 
       for many escaping points, so 
       Proposition~\ref{prop:externaladdresses} is not at all trivial, and uses the boundedness of the postsingular set
      in an essential manner. No analogue thereof is currently known for functions with unbounded postsingular set,
      and in particular in this setting there is no canonical way of associating external addresses to arbitrary
      escaping points as in Corollary~\ref{cor:escapingpointsarefilaments}. 
      This is a major challenge in showing the unboundedness of components of the
      escaping set. (Compare Corollary~\ref{cor:lms} below.) 
\end{rmk}

\subsection*{Connectedness of filaments and uniform escape to infinity}
  Thanks to Proposition~\ref{prop:externaladdresses}, we can use the 
    results of Section~\ref{sec:unboundedsets} to study filaments. We use 
    the following definition from \cite{arclike}. 

   \begin{defn}[Escaping composants] \label{defn:unifesccomponent}
    Let  $z\in I(f)$. Then the
      \emph{escaping composant} $\mu(z)=\mu(z,f)$ is the union of all 
       connected sets $A$ containing $z$ on which the iterates of $f$ tend to infinity uniformly. 
   \end{defn}

  \begin{remark}
While $\mu(z)$ is a union of sets on which the iterates tend to infinity uniformly, typically $f^n$ will not tend to infinity
     uniformly on  $\mu(z)$.
  \end{remark}

   \begin{lem}[Filaments and uniform escape]\label{lem:uniformcomponents}
     Let $\s$ be an external address, and suppose that
        $z\in G_{\s}$. Let $A\ni z$ be a connected set
        {and suppose that there is $N\geq 0$ such that $f^n(A)\cap D=\emptyset$ for $n\geq N$. 
        (In particular, this is the case when the iterates of $f$ tend to infinity uniformly on $A$.)}

   Then there is $n_0\geq N$ such that 
       $A\subset \unbdd{\tail}_n(\s)$ for $n\geq n_0$. 
      In particular, $\mu(z)\subset G_{\s}$. 
   \end{lem} 
   \begin{proof}
     {Let $n_0\geq N+1$ be such that $z\in \unbdd{\tail}_n(\s)$ for $n\geq n_0$. 
      Let $n\geq n_0$. 
         Then 
               \[ A \subset f^{-n}(\WW) \cap f^{-(n-1)}(\WW). \] 
          Now $\unbdd{\tail}_n(\s)$ is a connected component of the set on the right-hand side, $z\in A\cap \unbdd{\tail}_n(\s)$, 
           and $A$ is connected. Hence $A\subset \unbdd{\tail}_n(\s)$ for  $n\geq n_0$, and 
           all points in $A$ have address $\s$.}
   \end{proof} 
  
  We remark that a filament $G_{\s}$ may contain uncountably many different escaping composants.  
    Moreover, it is possible for $\mu(z)$ to consist of a single point.
    (See \cite[Theorem~1.6]{arclike}.) 
   However, by the results of Section~\ref{sec:unboundedsets}, $G_{\s}$ contains a 
   distinguished uniformly escaping composant, namely the one consisting of those points
     for which the set $A$ can be taken to be unbounded. 

 \begin{deflem}[The core of a filament]\label{deflem:uniformlyescapingsetsarenested}
    Let $\s$ be an infinite external address. Let $\mathcal{X}$ be the collection of all 
       closed, unbounded, connected sets {$X\subset G_{\s}$} on which the
       iterates of $f$ tend to infinity uniformly. Then no element of $\mathcal{X}$ separates the plane.
       Furthermore, $\mathcal{X}$ is linearly ordered by inclusion; i.e., if 
      $X_1,X_2\in \mathcal{X}$ then  $X_1\subset X_2$ or  $X_2\subset X_1$. 

     In particular, if $J_{\s}\neq \emptyset$, then the union 
           $\mu_{\s}\defeq \bigcup \mathcal{X} \subset G_{\s}\neq\emptyset$  
       satisfies $\mu_\s=\mu(z)$ for all $z\in \mu_{\s}$.
     We call $\mu_{\s}$ the \emph{core} of the filament 
       $G_{\s}$.
 \end{deflem}

 \begin{proof}
    If $X\in \mathcal{X}$, then there is $n_0\geq 1$ 
      such that $X\subset \unbdd{\tail}_n(\s)$ for $n\geq n_0$. 
      In particular, $f^n(X)\subset J_{\sigma^n(\s)}^0$ for
      $n\geq n_0$.

      Let $n\geq n_0$. Since $f^n\colon \tau_n(\s)\to \WW$ is a conformal isomorphism, and $X$ is unbounded, it follows
      that $f^n(X)$ is unbounded for all $n$.  By Lemma~\ref{lem:J}
      we know that $f^n(X)$, and hence $X$, does not separate the plane.

    Furthermore, if $X_1,X_2\in \mathcal{X}$, then we can choose $n_0$ sufficiently large so that the above
      holds for both sets. 
      It follows from 
      part~\ref{item:terminal} of Theorem~\ref{thm:unboundedsets}
      that one of $f^{n_0}(X_1)$ and $f^{n_0}(X_2)$ is contained in the other, and the same holds for 
        $X_1$ and  $X_2$. 

   By part~\ref{item:unboundedsets} of Theorem~\ref{thm:unboundedsets}, together with
     Proposition~\ref{prop:externaladdresses}, the set $\mu_{\s}$ 
     (as defined in the statement of the lemma) is non-empty. By the fact we just proved, $\mu_{\s}$ 
     is a nested union of
     connected sets, and hence itself connected. Finally suppose
     that $z\in \mu_{\s}$; so $z\in X_z$ for some $X_z\in \mathcal{X}$. 
      Since $\mathcal{X}$ is linearly ordered by inclusion, we have
          \[ \mu_{\s} = \bigcup\{ X\in \mathcal{X}\colon X_z\subset X\} \subset
                 \mu(z). \] 
      Conversely, if the iterates of $f$ tend to infinity uniformly on the connected
        set $A\ni z$, then
      they do so also on the closed, unbounded and connected set 
       $\overline{A}\cup X_z$, which is contained in
       $G_{\s}$ by Lemma~\ref{lem:uniformcomponents}. Hence $A \subset \mu_{\s}$,
        and we have proved $\mu(z) = \mu_{\s}$, as desired.
  \end{proof}  
 
 \begin{prop}[Filaments are connected]\label{prop:core}
  Let $\s$ be an external address with $J_{\s}\neq\emptyset$. 
    Then $\mu_{\s}$ is dense in $J_{\s}$ and $G_{\s}$. In particular, both of these sets are connected and unbounded. 
 \end{prop}
\begin{proof}
  Let $z\in J_{\s}$. {By Corollary~\ref{cor:escapingpointsarefilaments},} there is $n_0$ such that 
      $f^n(z)\in J_{\sigma^n(\s)}^0$ for $n\geq n_0$. Let $X_n$ be the unbounded connected subset of
      $\mu_{\sigma^n(\s)}$ whose existence is guaranteed by Theorem~\ref{thm:unboundedsets}. 
      Recall that $X_n$ can be connected to  $f^n(z)$ by a curve $\gamma_n\subset\C\setminus \overline{D}$ 
      of cylindrical length at most $4\pi$ 
      such that $\gamma_n$ is homotopic to a curve in $\WW$. 
      In particular, the pullback  $\tilde{\gamma}_n$ of $\gamma_n$ along the orbit of $z$ 
      connects $z$ to the set $\tilde{X}_n \defeq f^{-n}_{\s}(X_n)\subset \mu(\s)$. 

    The density $\rho_{\Omega}$ of the 
      hyperbolic metric on $\Omega$ tends to zero like $1/\lvert z\rvert \lvert \log z\rvert$. 
      (Recall that $\Omega$ is the unbounded connected component of $\C\setminus\P(f)$.)
      Therefore 
      the hyperbolic length $\ell_{\Omega}(\gamma_n)$ is also uniformly bounded, independently of $n$. 
      It follows by Proposition~\ref{prop:expansion} that $\ell_{\Omega}(\tilde{\gamma}_n)\to 0$
      as $n\to \infty$. Hence $\dist_{\Omega}(z , \tilde{X}_n) \to 0$, and $z\in \overline{\mu_{\s}}$, 
      as claimed.
      
   Recall that the (relative) closure of a connected set is again connected. So 
      $J_{\s}$ and $G_{\s}$ are connected and unbounded, as the dense subset 
      $\mu_{\s}$ has these properties.
\end{proof}

   We note that the above result (and its proof) 
    is essentially a reformulation of the main argument in the proof of the main theorem 
     of \cite{eremenkoproperty}, which {we recover as} follows. 
\begin{cor}[Unbounded sets of escaping points]\label{cor:lms}
  If $f$ is a postsingularly bounded entire function, then every connected component of the
    escaping set $I(f)$ is unbounded.
\end{cor}
\begin{proof}
  Let $z\in I(f)$. By  Corollary~\ref{cor:escapingpointsarefilaments}, $z\in G_{\s}$ for some external address $\s$, 
    and by Proposition~\ref{prop:core}, $G_{\s}$ is an unbounded connected subset of $I(f)$. 
\end{proof}

\subsection*{Independence of filaments from the construction}

 Note that the definition of addresses, and hence of filaments, depends a priori on the choice of fundamental domains, and hence
    on the domain $\WW$ (i.e., on the choice of the disc $D$ and the curve $\delta$). 
   However, if filaments are to be considered canonical objects associated to $f$, then this dependence should not be essential. We briefly discuss why this is indeed the case. 

 Suppose that $\hat{W}_0$ is a different choice of base domain, and that 
    $\hat{\tail}$ is a fundamental tail of level at least $2$, using this alternative initial configuration. 

  Then it is easy to see that there is a unique fundamental domain $F$ for the original domain $\WW$ such that all sufficiently large points of 
     $\hat{\tail}$ lie in $F$. Indeed, if $z\in \hat{\tail}$ is sufficiently large, then $f(z), f^2(z)\notin \overline{D}$, and in particular $f(z)\notin \delta$. 
      So $f(z)\in \WW$, and the claim is established. 

    In particular, if $\hat{\s}$ is an external address with respect to $\hat{W}_0$, then we can associate to $\hat{\s}$ an address $\s= F_0 F_1 \dots$ with respect
      to $\WW$. {Here for each $n\geq 0$, $F_n$ is the fundamental domain associated, in the above manner, 
      to the fundamental tails 
      $\hat{\tail}_k(\sigma^n(\hat{\s}))$, for $k\geq 2$. (Observe that $F_n$ is independent of $k$.)}
      It is easy to see that, in turn, $\hat{\s}$ can be obtained from $\s$ using the same procedure in the opposite direction, so that the correspondence
     $\s\mapsto \hat{\s}$ is a bijection between external addresses defined using the original configuration $\WW$, and those defined using the modified one. {The following observation is an easy consequence of this construction.}

  \begin{obs}[Filaments are independent of the initial configuration]\label{obs:filamentsindependence}
    Let $\s$ and $\hat{\s}$ correspond to each other as in the above construction; let $G_{\s}$ be the filament obtained by using preimages of 
     $\WW$ in the construction, and let $\hat{G}_{\hat{\s}}$ be the corresponding filament according to the choice $\hat{W}_0$. 
    Then $G_{\s} = \hat{G}_{\hat{\s}}$. 

   In other words, the collection $\mathcal{G} = \{G_{\s}\}$ of subsets of $J(f)$ is independent of the initial choice of $\WW$. 
  \end{obs} 
 
 We remark that the same is not true for the set $J_{\s}$. Indeed, suppose that
    $z$ is a non-escaping point that has an external 
   address for one choice of $\WW$. Then we can choose a different initial configuration $\tilde{W}_0$,
    where the initial disc $\tilde{D}$ is chosen sufficiently large to ensure that $z$ enters $\tilde{D}$
    infinitely many times. Clearly $z$ does not have an external address with respect to this configuration.

  In similar fashion, we see that filaments are preserved under iteration. Indeed, 
    let $n\geq 1$. Suppose that
      $\hat{D}$ is a disc centred around zero such that $f^{n-1}(D)\subset \hat{D}$, and that $\hat{\delta}\subset\delta$ 
      connects $\hat{D}$ to $\infty$ outside of $\hat{D}$. Then $\hat{W}_0\defeq \C\setminus(\closure(\hat{D})\cup \hat{\delta})$ 
      is a valid initial configuration
      for $f^n$. Every fundamental domain of $f^n$ (with respect to $\hat{W}_0$) is contained in a unique fundamental tail
      of level $n$ for $f$. Conversely, every fundamental tail of $f$ contains exactly one fundamental domain of $f^n$. 
     Hence there is a natural bijection between the fundamental domains of $f^n$ and finite external addresses of length $n$
     for $f$, and thus between external addresses of $f^n$ and those of $f$. In particular, we obtain the following.   

  \begin{obs}[Filaments and iterates]\label{obs:filamentiterates}
    Every filament of $f$ is a filament of $f^n$, for $n\geq 1$, and vice versa. 
  \end{obs}
 
\subsection*{Cyclic order of addresses and filaments}  
   There is a natural \emph{cyclic order} on the set of fundamental tails of any given level,
    and in particular on the set of fundamental domains of $f$: if $A, B, C$ are fundamental tails, then
     $A \prec B \prec C$ means that $B$ tends to infinity between $A$ and $C$ in positive orientation.
    (See Section~\ref{sec:cyclicorder} for background on the cyclic order near infinity.) 
    We can also define a cyclic order on the 
    collection of filaments, by choosing for each a closed, connected, unbounded set on which the 
     iterates tend to infinity, as in Proposition~\ref{prop:core}, and considering the cyclic order of these sets. 

  Recall that the function $f$ acts in a natural way on filaments, and maps fundamental tails of level
    $n+1$ to fundamental tails of level $n$. By the remark above, this action locally preserves
    cyclic order, in the following sense. Let $A\prec B \prec C$ be either filaments or 
    fundamental tails of some fixed level $n>1$. If the addresses of  $A$, $B$ and $C$ all have the same initial
    entry, then $f(A)\prec f(B)\prec f(C)$. 

  We can also define a ``lexicographical'' cyclic order on (finite or infinite) external addresses.  To do so, we use 
     the curve $\delta$ to convert the cylic order on fundamental domains to a 
    linear order in the usual  sense, setting  $F < \tilde{F}$ if and only if 
    $\delta \prec F \prec \tilde{F}$. This linear order gives rise to a lexicographic order $<$ on 
    external addresses in the usual sense. The cyclic order on addresses is then
    defined by $\s^1 \prec \s^2 \prec \s^3$ if and only if $s^1<s^2<s^3$, $s^2<s^3<s^1$, or
     $s^3<s^1<s^2$. 

   It follows from what was written above that this cyclic order on addresses agrees with 
    the cyclic order of the associated fundamental tails or filaments. 

\subsection*{Disjoint-type addresses} 
 To conclude this section, we discuss a particularly well-behaved type of filament. 
   Note that, in general, the points in the closure of a filament $G_{\s}$ need not belong to $J_{\s}$. 
   Indeed, this is the
   case for those filaments of greatest interest to us, namely those accumulating on a periodic point whose orbit does not lie outside of $\overline{D}$. 
   We show now that this cannot occur 
    when the address $\s$ (or a forward iterate thereof) contains only fundamental domains that do not intersect $\overline{D}$.

\begin{lem}[(Eventually) disjoint-type addresses]\label{lem:control}
  Let $\s = F_0 F_1 \dots$ be an external address, and suppose that there is $n_0\geq 0$ such that      
    $\overline{F_n}$ does not intersect $\partial D$ for $n \geq n_0$. Then 
    $\overline{G_{\s}}  = J_{\s} = f_{\s}^{-n}( J_{\sigma^{n}(\s)}^0)\subset \tail_n(\s)$ for $n\geq n_0$.
\end{lem}
 \begin{proof}
   By assumption, we have $\overline{F_{n}}\subset W_0$, and in particular
      $\unbdd{F_n}=F_n$,   for $n \geq n_0$. 
   Thus $\overline{\tail_{n+1}(\s)}\subset \tail_n(\s)$ 
    for $n \geq  n_0$ (where we use the convention that $\tail_0(\s)= W_0$ for convenience). In particular,
     $\unbdd{\tail}_{n+1}(\s) = \tail_{n+1}(\s)$, and 
    \[ J_{\s} = \bigcup_{k> 0} \bigcap_{j\geq k} \unbdd{\tail}_j(\s) = 
             \bigcap_{j> n} \tail_j(\s) = 
              \{ z\in \tail_{n}(\s)\colon f^j(z)\in F_j\text{ for $j\geq n$}\} =
             f_{\s}^{-n}( J_{\sigma^n(\s)}^0) \]
     by definition. Furthermore,
     \[  \overline{G_{\s}} = \overline{J_{\s}} = \bigcap_{n > n_0} \overline{\tail_n(\s)} \subset
             \bigcap_{n > n_0} \tau_{n+1}(\s) = J_{\s} \] 
    (where the first equality holds by Proposition~\ref{prop:core}). This completes the proof.

 \end{proof} 
 
\begin{rmk}[Disjoint-type addresses]\label{rmk:disjointtype}
   If $n_0=0$, then we say that $\s$ is \emph{of disjoint type} (since $\tau_n(\s)$ and $\tau_{n'}(\s)$ have disjoint boundaries for $n\neq n'$). 
     For a \emph{disjoint-type} function (i.e., one that is hyperbolic with connected Fatou set), all addresses are 
     of disjoint-type, given a suitable choice of fundamental domains. 
In this case, every component of the Julia set $J(f)$ is one 
     of the sets $J_{\s}= J_{\s}^0$, 
     and the closure $J_{\s} \cup \{\infty\}$  in $\Ch$ is called a \emph{Julia continuum}. Compare \cite{arclike}.
  Suppose that, additionally, $G_{\s}$ is a hair
    (in the sense of Definition~\ref{defn:hair} below). Then it follows, using 
     \cite[Corollary~5.6]{arclike}, that 
    either $G_{\s} = J_{\s}$, and $J_{\s}$ is an arc to $\infty$ on which the iterates of $f$ tend to infinity uniformly, 
    or $J_{\s}\setminus G_{\s}$ contains a single non-escaping point $z_0$; here $G_{\s}$ is an arc connecting $z_0$ to $\infty$.    
    (For functions satisfying a head-start condition, this follows already from the results of~\cite{R3S}, without the need 
      for the results of~\cite{arclike}.)

In fact, it can be shown 
     that, whenever $\s$ is of disjoint type, the set $J_{\s}$ is homeomorphic to the component of the Julia set of 
   a  
   suitable disjoint-type function, with escaping points
     corresponding to escaping points. (Compare also Theorem~\ref{thm:juliacontinua} below.) Hence the above observation remains true for disjoint-type addresses, 
     even if $f$ itself is not of disjoint type.
    So we may think of filaments at disjoint-type addresses as always ``landing''. 
    When $\s$ is bounded, we show below 
    that this is true in a precise sense (see Proposition~\ref{prop:accumulationsets}~\ref{item:disjointtype}).

However, when $G_\s$ is not a hair, it may happen that $J_{\s}$ contains a dense or uncountable set of non-escaping points, 
   even when $f$ and hence all addresses are of disjoint type \cite[Theorem~2.3]{arclike}.
  In this article, we only consider landing properties for filaments at bounded external addresses, so these subtleties will not become relevant. 
\end{rmk}

\section{Hairs and filaments}\label{sec:hairs}
  
  We shall now discuss the relationship between hairs and filaments. 
     The term ``hair'' was coined by Devaney in the 1980s
     (see \cite[p.~168]{devaneybifurcation}), and is commonly used in an
     informal manner to refer to dynamically natural curves 
     in Julia sets of transcendental entire functions. 
    We use the 
    following convention. (See Remark~\ref{rmk:hair} below
    for comparison with some other definitions in the literature.) 

\begin{defn}[Hairs]\label{defn:hair}
  A filament $G_{\s}$ is a \emph{hair} if one of the following holds. 
     \begin{enumerate}[(a)]
        \item There is a homeomorphism
             $\gamma\colon [0,\infty]\to G_{\s}\cup \{\infty\}$ such that 
                 $\gamma(\infty)=\infty$ and $f^n|_{G_{\s}}\to\infty$ uniformly.\label{item:fastaddress}
        \item  There is a continuous bijection 
     $\gamma\colon (0,\infty] \to G_{\s}\cup\{\infty\}$ 
        such that
                 $\gamma(\infty)=\infty$ and $f^n|_{\gamma([t_0,\infty))}\to\infty$
                   uniformly, for any $t>0$.\label{item:slowaddress}
   \end{enumerate}
\end{defn}
\begin{remark}[Remark 1]
  In general, the bijection in~\ref{item:slowaddress} is not a homeomorphism.
    Indeed, even for exponential maps a hair may accumulate
    upon itself. Compare e.g.\ \cite{devaneyjarque,devaneyjarquemoreno,nonlanding,worsley}. 
\end{remark}
 \begin{remark}[Remark 2]
   The second part of condition~\ref{item:fastaddress} is essential. 
    Indeed, by \cite[Theorem~2.10]{arclike}, there exists a postsingularly bounded
     entire function having a filament 
     which is an arc connecting a finite endpoint to infinity, but such that
     the iterates on this arc do not tend to infinity uniformly. 
     
  On the other hand, the second part of condition~\ref{item:slowaddress} can be
   shown to be inessential. That is, suppose $\gamma\colon (0,\infty]\to G_{\s}\cup\{\infty\}$ 
   is a continuous bijection with $\gamma(\infty)=\infty$. Then, possibly after 
    reversing the orientation of $\gamma$ on $(0,\infty)$, for all $t>0$ we have
    $f^n|_{\gamma([t_0,\infty))}\to\infty$
                   uniformly. We shall not provide the (somewhat lengthy) proof,
                   as we do not require this fact in our paper. 
 \end{remark}

 \begin{obs}[Properties of hairs]\label{obs:boundedaddresshairs}
  \mbox{}\begin{enumerate}[(a)]
     \item $G_{\s}$ is a hair if and only if $G_{\sigma(\s)}$ is a hair. Up to
       suitable reparameterisation, the corresponding
       functions
        $\gamma$ and $\tilde{\gamma}$ satisfy $\tilde{\gamma}(t) = f(\gamma(t))$.
     \item 
   If $\s$ is a bounded external address, then $f^n|_{G_{\s}}$ does not tend to
     infinity uniformly. 
     In particular, if $\s$ is bounded, then $G_{\s}$ is a hair if and only if 
     condition~\ref{item:slowaddress} of Definition~\ref{defn:hair} holds.
      \end{enumerate}
 \end{obs} 
 \begin{proof}
  To prove the first claim, note that $f\colon G_{\s}\to G_{\sigma(\s)}$ is a continuous bijection. Furthermore, 
     for every $n\geq 2$, $f|_{\tau_n(\s)}$ extends to a homeomorphism $\tau_n(\s) \cup\{\infty\} \to \tau_{n-1}(\sigma(\s))\cup\{\infty\}$. 
      If $G_{\s}$ is a hair, define 
       $\tilde{\gamma}$ as in the claim. Clearly it is only necessary to show that $\tilde{\gamma}(t)\to\infty$ as $t\to\infty$, which follows 
       from the above fact since $\gamma|_{[1/2,1)}\subset \tau_n$ for sufficiently large $n$. For the converse direction, we define 
       $\gamma(t) \defeq (f|_{\tau_n(\s)})^{-1}(\tilde{\gamma}(t))$, where $n$ is sufficiently large depending on $t$, and proceed analogously. 
       
Now suppose that $\s$ is bounded, and let $R>0$ be as in Corollary~\ref{cor:boundedexistence}.
    Then, for all $n\geq 0$, the filament $f^n(G_{\s}) = G_{\sigma^n(\s)}$ contains a point of modulus at most $R$. Hence the iterates of $f$ do not tend to infinity
    uniformly on $G_{\s}$. 
 \end{proof} 

The following is an alternative formulation of Definition~\ref{defn:hair}, and allows
 us to connect the notion with our definition of {criniferous} functions. 

 \begin{prop}[Characterisation of hairs]\label{prop:haircharacterisation}
   A filament $G_{\s}$ of $f$ is a hair if and only if, for every $z\in G_{\s}$, there is
    an arc connecting $z$ to $\infty$ on which $f^n\to\infty$ uniformly. 
 \end{prop}
  \begin{proof}
   If $G_{\s}$ is a hair, then the stated condition holds by definition. 
   
   So now suppose that every point $z\in G_{\s}$ can be connected to infinity by an arc $\gamma_z$ on which the iterates tend to infinity uniformly. So 
   $\gamma_z\subset \mu(z) \subset \mu_{\s}\subset G_{\s}$ {by Lemma~\ref{lem:uniformcomponents}. 
   By Lemma~\ref{deflem:uniformlyescapingsetsarenested}, the arcs $\gamma_z$ are linearly ordered by inclusion,
    and the arc $\gamma_z$ is unique.}

Let $(x_n)_{n=0}^{\infty}$ be a countable dense subset of 
    $G_{\s}$. Set $n_0=0$ and define $n_{k+1}$ inductively as the minimal value of $n$ for which $x_n$ does not lie on the 
     arc $\gamma_{x_{n_k}}$. {(If no such $n$ exists, then 
     $G_{\s}=\gamma_{x_{n_k}}$ satisfies Definition~\ref{defn:hair}~\ref{item:fastaddress} and we are done.)}
We set $y_k \defeq x_{n_k}$. Then the union of the arcs $\gamma_{y_k}$ is a single continuous injective
     curve $\gamma\colon (0,\infty)\to G_{\s}$, which can be parameterised such that $\gamma([1/k,\infty)) = \gamma_{y_k}$ for all $k$. 

   First suppose that $\gamma(t)$ has a limit $z_0\in G_{\s}$ as $t\to 0$. Then {$\gamma_{z_0} = 
     \gamma\cup \{z_0\} =\overline{\gamma} = G_{\s}$, since $\gamma$ is dense in $G_{\s}$. 
     Hence $G_{\s}$ satisfies Definition~\ref{defn:hair}~\ref{item:fastaddress}.}

 Otherwise, 
     $\gamma\bigl((0,\infty)\bigr)\not\subset \gamma_z$ for all $z\in G_{\s}$. Hence we see that $\gamma_z\subset \gamma_{y_k}$ for sufficiently large $k$; it follows
     that $\gamma$ is surjective, and Definition~\ref{defn:hair}~\ref{item:slowaddress} holds. 
 \end{proof} 

  \begin{cor}[Filaments and criniferous functions]\label{cor:filamentscriniferous}
    A transcendental entire function $f$ with bounded postsingular set is {criniferous} if and only if every filament is a hair.
  \end{cor}
\begin{proof}
  Recall that every escaping point of $f$ belongs to some filament. Hence the claim is immediate from {Proposition~\ref{prop:haircharacterisation}}.
\end{proof}

Observe that we now have two apparently different definitions of periodic hairs, namely 
  Definition~\ref{defn:periodichair}, and the case of a periodic filament (i.e., a filament $G_{\s}$ at a periodic address~$\s$) that is a hair. 
  We conclude the section by showing that the two coincide.

\begin{prop}[Periodic hairs]\label{prop:periodichairfilament}
  Every periodic hair in the sense of Definition~\ref{defn:periodichair} is a periodic filament, 
    and this filament is a hair in the sense of 
    Definition~\ref{defn:hair}. 
    Conversely, any periodic filament that is a hair has a parameterisation as a periodic hair in the sense of Definition~\ref{defn:periodichair}. 
\end{prop}
\begin{proof}
  Suppose that $\gamma$ is a periodic hair. Then, for every $t\in\R$, the iterates $f^n$ tend to infinity 
     uniformly on
     $\gamma\bigl([t,\infty)\bigr)$.
    Hence, by 
    Lemma~\ref{lem:uniformcomponents}, $\gamma$
    is contained in a filament $G_{\s}$, whose address $\s$ is necessarily periodic. Let $\mathcal{X}$ be as in
    Lemma~\ref{deflem:uniformlyescapingsetsarenested}. If $X\in\mathcal{X}$, then 
    $f^n(X)\subset \gamma\bigl([0,\infty)\bigr)$ for sufficiently large $n$, and hence
    $X\subset \gamma\bigl([-n,\infty)\bigr)$. Thus we conclude that $\gamma = \mu_{\s}$.

    Let $R$ be as in part~\ref{item:boundedescape} of Theorem~\ref{thm:unboundedsets}, applied to the set of fundamental domains occurring in $\s = F_0 F_1 F_2 \dots$. 
     It follows that there is $t$ such that $\gamma((-\infty,t))$ contains only points of modulus less than $R$. 
     Indeed, by choice of $R$ there is $n_1$ such that $f^n(z)\notin \gamma([0,1])$ for $n\geq n_1$, whenever $z\in G_{\s}$ with $\lvert z \rvert \geq R$; then
     we  may take $t = 1-n_1$.

       Let $z\in G_{\s}$, and let $n_0$ be such that $f^n(z)\in F_n$ and $\lvert f^n(z)\rvert > R$ for $n\geq n_0$. We may assume that $n_0$ is a 
      multiple of the period of $\s$. {Recall that
       $f^{n_0}(z)\in G_{\s}$ is in the closure of $\gamma$ by Proposition~\ref{prop:core}. 
       Since $\gamma((-\infty,t))$ contains no points
       of modulus greater than $R$, it follows that $f^{n_0}(z)\in \gamma([t,\infty))$.}  
      Thus $z\in \gamma([t-n_0 , \infty))$. We have proved that $G_{\s} = \gamma$, and clearly $G_{\s}$ is a hair in the sense of Definition~\ref{defn:hair}. 

   Now suppose that $\s$ is periodic and that $G_{\s}$ is a hair. Let $\tilde{\gamma}\colon (0,\infty) \to G_{\s}$ be a parameterisation of $G_{\s}$ as 
       in part~\ref{item:slowaddress} of Definition~\ref{defn:hair}. Consider the point $z_0\defeq \tilde{\gamma}(1)$, and its image 
      $z_1 = f(z_0)$, then $z_1 = \tilde{\gamma}(t)$ for some $t$. It follows easily that $t>1$, and that the piece 
      $\tilde{\gamma}\bigl((1,t)\bigr)$ is disjoint from its forward and backward images in $G_{\s}$. We may reparameterise $\tilde{\gamma}$ 
      to a curve $\gamma\colon\R\to G_{\s}$ such that $\gamma([0,1])$ corresponds to $\tilde{\gamma}([1,t])$, and such that
     $\gamma(t+1) = f(\gamma(t))$ for all $t$. Then $\gamma$ is a periodic hair in the sense of Definition~\ref{defn:periodichair}.
\end{proof} 

 \begin{rmk}[On the concept of hairs]\label{rmk:hair}
   The notion of a ``dynamic ray'' of an entire function is given in \cite[Definition~2.2]{R3S}. By the same reasoning as in 
     Observation~\ref{obs:boundedaddresshairs}, in our setting of postsingularly bounded functions this definition
     can be phrased as follows: a ``dynamic ray'' of a
     postsingularly bounded entire function is 
     a maximal curve in the escaping set satisfying~\ref{item:slowaddress} of
     Definition~\ref{defn:hair}. Hence every dynamic ray is contained in the
     core of a
     filament, and a filament $G_{\s}$ is a dynamic ray
     if and only if it satisfies Definition~\ref{defn:hair}~\ref{item:slowaddress}. 
     In particular, for a bounded address $\s$, $G_{\s}$ is a dynamic ray if
     and only if it is a hair. 
     Moreover, for criniferous functions~-- and in particular the class of functions for which
    hairs are constructed in \cite{R3S}~-- every 
    dynamic ray in the sense of~\cite[Definition~2.2]{R3S} is
    either a hair in the sense of Definition~\ref{defn:hair}, or becomes such
    upon the addition of a finite escaping endpoint. 
 
  In general, \cite{R3S} leaves open
   the possibility that a filament contains a hair as a proper subset. 
    For example, it follows from~\cite[Theorem~2.5]{arclike} 
    that there is an entire function $f$ with bounded postsingular set 
    (and indeed of disjoint type), 
    having a bounded-address filament
    $G_{\s}$ with the following properties. The set $\overline{G_{\s}}\setminus G_{\s}$ consists of a single point $z_0$ with bounded orbit, and 
     $\hat{G}_{\s}$ is homeomorphic to a $\sin(1/x)$ continuum, with the starting point of the accumulating curve corresponding to $\infty$, and one of 
     the endpoints of the limiting interval situated at $z_0$. Then the accumulating curve 
   itself is a ``dynamic ray'', but this ray does not include all points of $G_{\s}$. 

  Here, we restrict only to consider cases where
    the entire filament $G_{\s}$ is a hair, and hence leave open the question
    of whether proper subsets of filaments should be considered ``hairs'' or not. 
    
   Also note that \cite[p.~740]{devaneyhairs} defines a notion of \emph{Devaney hairs}; for postsingularly bounded functions 
     such a Devaney hair is a curve $\gamma$ as in~\ref{item:slowaddress} of
     Definition~\ref{defn:hair}, with the addition of a finite, not necesarily escaping,
     endpoint. In particular, if $G_{\s}$ is a hair, then $G_{\s}$ contains many Devaney hairs in the sense of~\cite{devaneyhairs}, linearly ordered by inclusion.
     Conversely, any Devaney hair is either contained in a filament, or consists of a filament together with a finite landing point. 
 \end{rmk}

\section{Accumulation sets and landing properties of bounded-address filaments}\label{sec:accumulationlanding}

We now study the accumulation behaviour of filaments at bounded external addresses. 
(The restriction
      to bounded addresses is due to the phenomena discussed at the end of
      Remark~\ref{rmk:disjointtype}.)
   In the case where the filament in question is a hair, it is clear that the ``accumulation set'' of this filament should
    be the 
    set of limit points of the curve $\gamma(t)$ as $t\to 0$, where $\gamma$ is the 
    curve from  Definition~\ref{defn:hair}.  For periodic hairs, one can see easily 
     that this is equivalent to fixing a base point on this hair, considering its successive preimages
     along $\gamma$, and studying the 
      accumulation set of this sequence. This motivates the following definition.
      
 \begin{defn}[Accumulation sets of bounded addresses]\label{defn:accumulationsetabstract}
    Let $\s$ be a bounded external address, and let $\zeta\in \WW$.
 
  For $n\geq 1$, set $\zeta_n \defeq \zeta_n(\s) \defeq f_{\s}^{-n}(\zeta)\in \tau_n(\s)$.
   (Recall from Definition~\ref{deflem:pullbacks} that $f_{\s}^{-n}=(f^n|_{\tail_n(\s)})^{-1}$.)
     Then the \emph{accumulation  set} $\Lambda(\s,\zeta)$ of $\s$ with respect to $\zeta$
      is defined to be the accumulation set (in  $\Ch$) of the sequence
      $(\zeta_n)_{n=1}^{\infty}$.
 \end{defn}
Note that this is an abstract definition of the accumulation set associated to an address; it uses only the notion of
     fundamental tails, and does not require the definition of filaments or their properties. 
  That this is a natural concept may not be clear a priori, but should become apparent
   through the results proved in this section. 
   We  begin by verifying that $\Lambda(\s,\zeta)$ is 
    independent of $\zeta$. 
 
   This follows
   from the following well-known fact about the shrinking of univalent preimages.  (Compare also~\cite[Proposition 3]{Lyu83}.)
Recall that $\Omega$ is the unique unbounded component of $\C\setminus P(f)$.

\begin{lem}[Euclidean shrinking]\label{lem:Euclideanshrinking}
    Suppose that $V\Subset \Omega$ is a bounded  Jordan domain.
      Then,  for any 
        $\eps>0$ and any compact set $K\subset\C$, 
        there exists $N_\eps$ with the following property. For  $n\geq N_\eps$,  
        every connected component of  $f^{-n}(V)$ 
        that intersects $K$ has Euclidean diameter at  most $\eps$.  
 \end{lem}
 \begin{proof}
  Suppose, by contradiction, that there is
    a sequence $(V_{k})_{k=0}^{\infty}$ of $n_k$-th preimages of $V$, with $n_k\to\infty$,
    with $V_k\cap K\neq \emptyset$ and $\inf_k \diam(V_k)>0$. 
    Let
    $\tilde{V}\Subset \Omega$ be a slightly larger Jordan domain than $V$ with $\overline{V}\subset\tilde{V}$, 
     and let $\tilde{V}_{k}$ be the component of $f^{-n_k}(\tilde{V})$ containing
    $V_{n_k}$. Then $f^{n_k}\colon \tilde{V}_{k}\to \tilde{V}$ is a covering map and hence a conformal isomorphism, 
    whose inverse
    $\phi_k\defeq (f|_{\tilde{V}_k})^{-1} \colon \tilde{V}\to \tilde{V}_k$ maps $V$ to $V_k$.  
 
   By assumption, there is a sequence $(z_k)_{k=0}^{\infty}$ with $z_k\in K\cap V_k$. 
     By Koebe's distortion theorem, $\tilde{V}_k$ contains a round disc around $z_k$ whose diameter 
     is comparable
     to that of $V_k$. By assumption, the latter is bounded from below. Hence, if $U$ is a sufficiently small disc
     centred at a limit point of the sequence  $(z_k)$, then $U$ 
      is contained in infinitely many
    $\tilde{V}_{n_k}$. It follows that $f^n(U)\subset \Omega$ for all $n\geq 0$, and $f^{n}(U)\subset \tilde{V}$ 
    for infinitely many $n$. This contradicts Corollary~\ref{cor:remaininginOmega}.
\end{proof}

 \begin{cor}[Spherical shrinking]\label{cor:shrinking}
    Suppose that $V\Subset \Omega$ is a bounded  Jordan domain.
      Then,  for any 
        $\eps>0$ there exists $N$ such that for  $n\geq N$  every connected component of  $f^{-n}(V)$ 
         has spherical diameter at  most $\eps$.  
 \end{cor}
\begin{proof}
   Let $K\subset\C$ be a compact set whose complement has spherical diameter less than~$\eps$. 
    Since the spherical and Euclidean metrics are comparable on any compact subset of the plane, there is
    $\eps_1$ such that any set of Euclidean diameter at most $\eps_1$ that intersect $K$ 
     has spherical diameter at most $\eps$. 
     Let $N=N_{\eps_1}$ be as in Lemma~\ref{lem:Euclideanshrinking}, let $n\geq N$, and let $X$ be a connected  
     component of $f^{-n}(V)$. Then either $X\cap K=\emptyset$, and hence $X$ has spherical diameter
     less than $\eps$, or $X\cap K\neq\emptyset$ and Lemma~\ref{lem:Euclideanshrinking} applies. 
     In the latter case, $\diam  X\leq \eps_1$, and hence $X$ also has spherical diameter at most $\eps$ by 
     choice of $\eps_1$.      
\end{proof}

\begin{cordef}[Accumulation sets and landing of filaments]\label{item:accumulationwelldefined}\label{defn:accumulationset2}
 Let $\s$ be a bounded external address. Then $\Lambda(\s,\zeta^1) = \Lambda(\s,\zeta^2)$ 
        for any two choices of  $\zeta^1,\zeta^2\in\WW$. 
        
        We write $\Lambda(G_{\s})\defeq \Lambda(\s)\defeq \Lambda(\s,\zeta)$, 
          and call $\Lambda(G_{\s})$ the \emph{accumulation set}
         of the filament $G_{\s}$.  The filament $G_\s$ is said to \emph{land} at a point $z_0\in\Ch$ if   $\Lambda(G_\s)=\{z_0\}$.
\end{cordef}
\begin{proof}
  Let $V$ be a bounded Jordan domain in $W_0$ containing both $\zeta^1$ and $\zeta^2$. For $j\in\{1,2\}$ and $n\geq 1$, write
    $\zeta_j^n \defeq f^{-n}_{\s}(\zeta^j) \in {\tail}_n\defeq {\tail}_n(\s)$. 
    Since $f^{n}\colon \tau_n \to W_0$ is univalent, for any $n$ the points $\zeta^1_n$ and $\zeta^2_n$ belong to the same connected component 
    of $f^{-n}(V)$. By Corollary~\ref{cor:shrinking}, the spherical diameter of this component tends to zero as $n\to\infty$. Hence the sequences
    $(\zeta^1_n)_{n=1}^{\infty}$ and $(\zeta^2_n)_{n=1}^{\infty}$   have the same accumulation set in $\Ch$.
\end{proof}

The following establishes a number of fundamental properties of accumulation sets.
 \begin{prop}[Properties of accumulation sets]\label{prop:accumulationsets}
   Let $\s$ be a bounded external address, and $G_{\s}$ the filament at address $\s$. 
   \begin{enumerate}[(a)]
      \item\label{item:accumulationconnected}
      The accumulation set $\Lambda(\s)$ is a nonempty connected subset of $\Ch$.
      \item\label{item:accumulationsetinclosure}%
             The closure $\hat{G_{\s}}$ of $G_{\s}$ in $\Ch$ is given precisely by 
                        \[\hat{G}_{\s} = G_{\s} \cup \Lambda(\s) \cup \{\infty\} . \]
     \item  If $J_{\s}\setminus G_{\s} \neq\emptyset$, then this set has a unique element $z_0$,
           and $G_{\s}$ lands at $z_0$, which has bounded orbit. 
     \label{item:landingfromaddress}
     \item {If $\s$ is of disjoint type, then $G_{\s}$ lands at a point $z_0\in J_{\s}$ having 
         bounded orbit.}\label{item:disjointtype}
    \item
     Let $U$ be a  neighbourhood of a point $z_0\in \Lambda(\s)$. Then $f^n$ does not tend to infinity uniformly on $ U\cap G_{\s}$.
   \label{item:accumulationnonuniform}
   \item
  Let $K\subset  \hat{{G}_{\s}}$ be compact with $K\cap\Lambda(\s)=\emptyset$.  Then $f^n\to\infty$ 
       uniformly on $K\cap\C$,  and in particular $K\cap \C\subset  J_{\s}$.
   \label{item:uniformawayfromaccumulation}
   \end{enumerate}
 \end{prop}

 Before proving these facts, we observe a few consequences. Firstly, we see that we can 
    characterise the accumulation set purely in terms of $G_{\s}$ as a subset of the escaping set, justifying the  notation $\Lambda(G_{\s})=\Lambda({\s})$.
 
\begin{cor}[Topological landing criterion]\label{cor:alternativelanding}
     Let $G_{\s}$ be a filament at a bounded address $\s$. 
      Then the accumulation set of $G_{\s}$ consists precisely of those points
      in $\hat{G}_{\s}$ having a neighbourhood $U$ such that $f^n$ does not tend to infinity uniformly on $U\cap G_{\s}$. 
 
    In particular, $G_{\s}$ lands at a point   $z_0\in\C$ if and only if \(\hat{G}_{\s} = G_{\s}\cup\{z_0,\infty\}\) and 
    if furthermore, for every neighbourhood $U$ of $z_0$ in $\Ch$, the iterates of $f$ tend to infinity 
      uniformly on $G_{\s}\setminus U$. 
 \end{cor}
\begin{proof}[Proof of Corollary~\ref{cor:alternativelanding}, using Proposition~\ref{prop:accumulationsets}]
   The first claim is a direct consequence of Proposition~\ref{prop:accumulationsets}~\ref{item:accumulationnonuniform} and~\ref{item:uniformawayfromaccumulation}. 
     The second claim follows from the first (together with the fact that the accumulation set of $G_{\s}$ is nonempty). 
\end{proof}

In particular, it follows that for hairs our definition of the accumulation set agrees with the usual one. Recall by Observation~\ref{obs:boundedaddresshairs} that if a bounded-address filament 
   is a hair, it must satisfy part~\ref{item:slowaddress} of Definition~\ref{defn:hair}.

\begin{cor}[Accumulation sets of hairs]\label{cor:hairaccumulation}
  Suppose that a filament $G_{\s}$ at bounded address is a hair, and let $\gamma\colon (0,\infty)\to G_{\s}$ be the continuous bijection from
    Definition~\ref{defn:hair}. Then $\Lambda(G_{\s})$ is precisely the accumulation set of $\gamma(t)$ as $t\to 0$. 
\end{cor}
\begin{proof}[Proof of Corollary~\ref{cor:hairaccumulation}, using Proposition~\ref{prop:accumulationsets}]
   Let $\Lambda(\gamma)\subset\Ch$ denote the set of accumulation points of $\gamma(t)$ as $t\to 0$.
    Recall that the iterates of $f$ tend to infinity uniformly on $\gamma([t,\infty))$ for all $t>0$, by  Definition~\ref{defn:hair}.
     Thus we see from 
    part~\ref{item:accumulationnonuniform} of Proposition~\ref{prop:accumulationsets} (or from 
    Corollary~\ref{cor:alternativelanding}) 
  that $\Lambda(G_{\s}) \subset \Lambda(\gamma)$.

  {It remains to prove the opposite inclusion, $\Lambda(\gamma)\subset \Lambda(G_{\s})$. 
     If $\gamma$ lands, i.e., when $\#\Lambda(\gamma) = 1$, this follows from the first inclusion and the fact that
      $\Lambda(G_{\s})\neq\emptyset$. So we may assume that $\Lambda(\gamma)$ is a nondegenerate 
      continuum.}

   {Since $\Lambda(\gamma)\subset \hat{G}_{\s}$, we have 
           \[ \tilde{\Lambda} \defeq \Lambda(\gamma)\setminus (G_{\s} \cup \{\infty\}) \subset \Lambda(G_{\s}) \]
          by~\ref{item:accumulationsetinclosure} of Proposition~\ref{prop:accumulationsets}. As $\Lambda(G_{\s})$ is closed,
           it thus suffices to  prove that $\tilde{\Lambda}$ is dense in $\Lambda(\gamma)$.} 

{This is clear if $\Lambda(\gamma)\cap G_{\s}=\emptyset$. Otherwise,  let $t_0>0$ be such that
      $\gamma(t_0)\in \Lambda(\gamma)$. We claim that $\gamma\bigl((0,t_0]\bigr)\subset \Lambda(\gamma)$.}
      
  {This follows from a well-known argument that we sketch as follows; 
     compare~\cite[Lemma~5.1]{topescapingnew} for the
     case of exponential maps. For every  $t>0$, there are pieces of
      other hairs of $f$ accumulating uniformly from above and below on $\gamma\bigl([t,\infty)\bigr)$.  (For example,
      this follows from \cite[Proposition~8.1]{arclike} via \cite[Theorem~1.1]{Re09}.) Thus, in order to accumulate on
       $\gamma(t_0)$, the curve $\gamma$ must also accumulate on  $\gamma([t,t_0])$; letting  $t\to 0$, the
       claim is established.}

 {Let $t_0\in (0,\infty]$ be the supremum over all possible choices of $t_0$, and let $(t_n)_{n=0}^{\infty}$ 
     be a decreasing sequence with  $t_n\to 0$. Then 
     $A_n \defeq \gamma([t_n,t_0])$ is a compact and nowhere dense subset of $\Lambda(\gamma)$ for all
     $n$. By Baire's theorem,
      \[ \tilde{\Lambda} = \Lambda(\gamma)\setminus \bigcup A_n \]
      is dense in $\Lambda(\gamma)$, as claimed. This completes the proof.}
\end{proof}
\begin{remark}[Remark 1]
{The proof shows that a hair $G_{\s}$ either lands, or otherwise 
     a generic point in $\Lambda(G_{\s})$ belongs to $\C\setminus G_{\s}$. 
     It is plausible that this is true without the assumption that $G_{\s}$ is a hair, with a similar proof.
     This would simplify the characterisation in Corollary~\ref{cor:alternativelanding} as follows:
     A filament $G_{\s}$ at a bounded address $\s$ 
     lands at a finite point $z_0\in\C$ if and only if $\overline{G_{\s}}\setminus G_{\s} = \{z_0\}$.} 
\end{remark}
\begin{remark}[Remark 2]
  For periodic addresses, or for addresses satisfying a head-start condition as in \cite{R3S}, 
     the proof of Corollary~\ref{cor:hairaccumulation}
   is considerably simpler. Indeed, in this setting it is easy to see directly that the iterates of $f$ do not tend to infinity
   uniformly on any neighbourhood of any point of $\Lambda(\gamma)$.
\end{remark}

Finally,  we observe that \emph{most}  periodic rays land and \emph{most} periodic points  are landing points.    Compare also \cite{BK07,BF15,Be16}.
 \begin{cor}[Most periodic filaments land]
  Let $p\geq 1$. Then, for all but finitely many periodic addresses $\s$ of period $p$, the 
    filament $G_{\s}$ lands at a periodic point $z\in J_{\s}$ of  period $p$.  
   Conversely, for all but finitely many periodic points $z$ of period $p$, the point $z$ is repelling
     and there is a periodic address $\s$ of period $p$ such that $z\in  J_{\s}$; in particular, 
     $G_{\s}$ lands at $z$.
 \end{cor}
  \begin{proof} 
      Passing to an iterate, we can assume that $p=1$. (Recall that the sets $J_{\s}$ are pairwise disjoint by Lemma~\ref{lem:addressproperties}~\ref{item:Jsdisjoint}.)
      Only finitely many fundamental domains $F$ intersect
     $\overline{D}$, and hence all but finitely many fixed addresses of $f$ are of disjoint type. Hence the
     first claim follows from Proposition~\ref{prop:accumulationsets}~\ref{item:disjointtype}.    Similarly, all but finitely many fixed points of $f$ are contained in $\WW$, and hence have 
     a fixed external address. So the second claim is a consequence 
     of Proposition~\ref{prop:accumulationsets}~\ref{item:landingfromaddress}.
 \end{proof}
The remainder of the section is dedicated to establishing Proposition~\ref{prop:accumulationsets}.
    We shall do so by applying Corollary~\ref{cor:shrinking} to a suitable large Jordan domain $V$, depending on the
    collection of fundamental domains involved.
   The following technical lemma collects the properties that we require of $V$.

 \begin{lem}[Domains for bounded-address filaments]\label{lem:boundeddomain}
 Let    $\zeta$ belong to an unbounded connected component of 
    $\WW\cap f^{-1}(D)$
  and let $R>0$. 
  Let $\FF$ be any finite collection of fundamental domains of $f$. 
   Then there is a Jordan domain $V\Subset \Omega$    with the following properties.
    \begin{enumerate}[(i)]
     \item $\zeta\in V$. 
     \item For all $F\in\FF$, the unique preimage $\zeta_F$ of $\zeta$ in $F$ also belongs to $V$.
     \item For all $F\in\FF$, there is a connected component $A_F$ of $V\cap \overline{F}$ containing $\zeta_F$ as well as all points of $\overline{F}$ having 
       modulus at most $R$.\label{item:AF}
     \item\label{item:AFunbdd} 
      If $U$ is the connected component of $V\cap \WW$ containing $\zeta$, 
        then $U\cap A_F$ intersects the unbounded connected component $\unbdd{F}$ of $F\setminus \overline{D}$, 
        for all $F\in\FF$. 
    \end{enumerate}
 \end{lem}
\begin{proof}
See Figure~\ref{fig:V}.
  Write $\FF = \{F^1,\dots, F^n\}$. Let $\gamma$ be a Jordan curve in $\C\setminus\overline{D}$ that intersects the arc $\delta$ in exactly one point.
    We may assume that $\gamma$ is chosen so large that $\gamma$ surrounds $\zeta$, $f(\overline{D})$ and the set $\{f(z)\colon \lvert z\rvert \leq R\}$.

  For each $i$, set $\gamma^i \defeq f^{-1}(\gamma)\cap F^i$ and $\zeta^i \defeq \zeta_{F^i}$. Then $\gamma^i$ is a cross-cut of $F^i$ that separates $\zeta^{i}$,
    $F^i\cap\overline{D}$ and all points of modulus at most $R$ in $F^i$ 
       from $\infty$ in $F^i$. In particular, $\gamma^i$ belongs to the unbounded connected component of $F^i\setminus\overline{D}$. 

    Let $X^i$ be the closure of the bounded component of
     $F^i\setminus \gamma^i$, and let $Y$ be the union of these components. Observe that different $X^i$ may intersect when the corresponding fundamental
     domains are adjacent. In this case, the corresponding two $X^i$ intersect precisely in a preimage component of the piece of $\delta$ that connects $\overline{D}$ to
     $\gamma$. Let  $Y_1,\dots,Y_k$ be the  $k\leq n$ connected components of $X$, and set $\Gamma_j \defeq Y_j\cap f^{-1}(\gamma)$. That is,
     $\Gamma_j$ is the union of finitely many $\gamma^i$, together with their endpoints.

  Then each $Y_k$ is a closed Jordan domain in $\Omega\setminus\delta$, and $\Gamma_j$ is an arc of $\partial Y_k$.
    (Recall that the closure of a fundamental domain does not meet $\delta$.)
     Let $\widetilde{Y}$ be obtained from $Y$ by adding, for each $j\leq k$,
    an arc $\beta_j$ 
    in $W_0\setminus  X$ joining $\zeta$ to a point of $\Gamma_j$; this
      is possible by the assumption on $\zeta$. 
      We may assume that two different $\beta_j$ intersect only at $\zeta$.

  Then $\widetilde{Y}$ is a compact and full set, and we may let $V$ be a Jordan domain containing $\widetilde{Y}$. The point
    $\zeta$ and all $\zeta^i$ belong to $V$ by construction. Moreover, 
   each $Y_j$ is connected, and hence belongs to a single connected component of $V\cap (\bigcup_i\overline{F^i})$. Finally,
    the union of all $\beta_j$ and all $\Gamma_j$ is connected by construction, is contained in $\C\setminus \overline{D}$ and 
    intersects the unbounded connected component of $F^i\setminus\overline{D}$ for all $i$. This completes the construction. 
\end{proof}
 
\begin{figure}
\def\svgwidth{.65\textwidth}
\begingroup%
  \makeatletter%
  \providecommand\rotatebox[2]{#2}%
  \ifx\svgwidth\undefined%
    \setlength{\unitlength}{512.05929974bp}%
    \ifx\svgscale\undefined%
      \relax%
    \else%
      \setlength{\unitlength}{\unitlength * \real{\svgscale}}%
    \fi%
  \else%
    \setlength{\unitlength}{\svgwidth}%
  \fi%
  \global\let\svgwidth\undefined%
  \global\let\svgscale\undefined%
  \makeatother%
  \begin{picture}(1,0.82915159)%
    \put(0,0){\includegraphics[width=\unitlength]{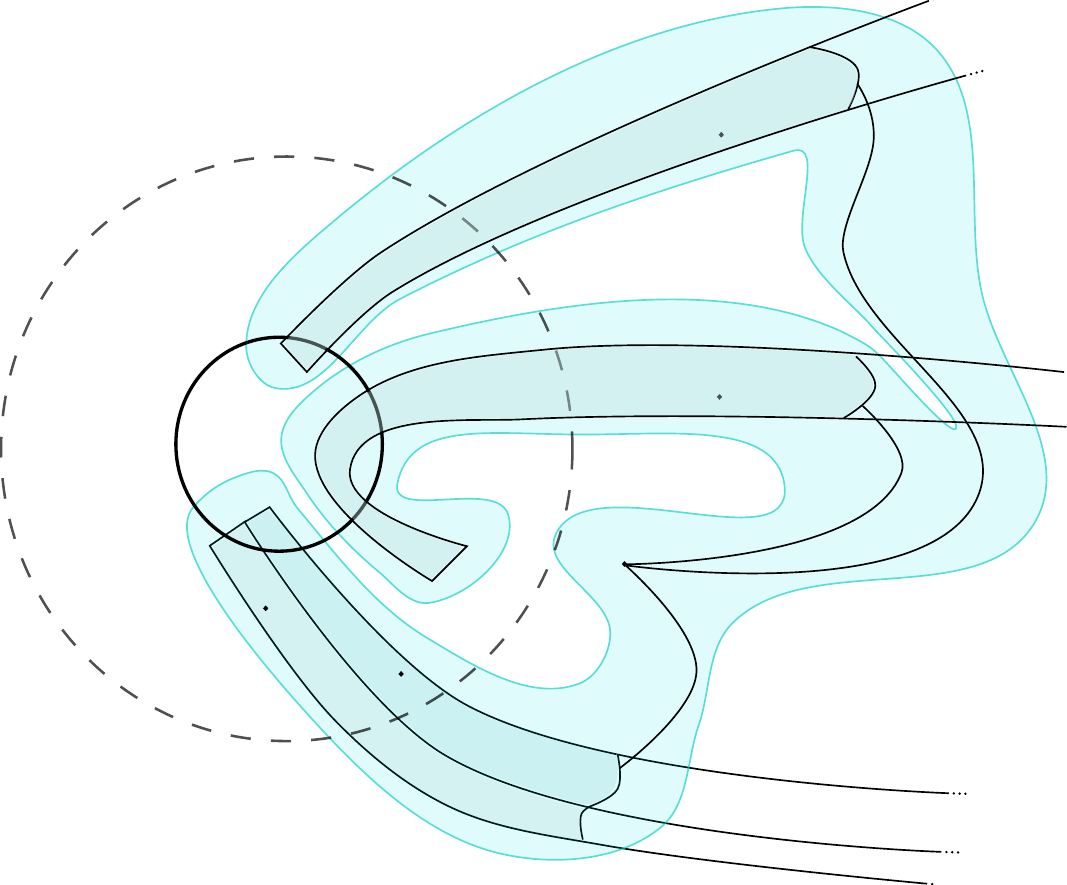}}%
    \put(0.19342145,0.41931805){\color[rgb]{0,0,0}\makebox(0,0)[lb]{\small{$D$}}}%
    \put(0.47550362,0.635){\color[rgb]{0,0,0}\makebox(0,0)[lb]{\small{$F^1$}}}%
    \put(0.47988486,0.4593828){\color[rgb]{0,0,0}\makebox(0,0)[lb]{\small{$F^2$}}}%
    \put(0.47871883,0.10){\color[rgb]{0,0,0}\makebox(0,0)[lb]{\small{$F^3$}}}%
    \put(0.57943319,0.31950585){\color[rgb]{0,0,0}\makebox(0,0)[lb]{\small{$\zeta$}}}%
    \put(0.38458639,0.19473898){\color[rgb]{0,0,0}\makebox(0,0)[lb]{\small{$\zeta^3$}}}%
    \put(0.7,0.45005447){\color[rgb]{0,0,0}\makebox(0,0)[lb]{\small{$\zeta^2$}}}%
    \put(0.7,0.70425243){\color[rgb]{0,0,0}\makebox(0,0)[lb]{\small{$\zeta^1$}}}%
    \put(0.255,0.24){\color[rgb]{0,0,0}\makebox(0,0)[lb]{\small{$\zeta^4$}}}%
    \put(0.81570605,0.75789049){\color[rgb]{0,0,0}\makebox(0,0)[lb]{\small{$\Gamma_1$}}}%
    \put(0.82037022,0.4523865){\color[rgb]{0,0,0}\makebox(0,0)[lb]{\small{$\Gamma_2$}}}%
    \put(0.58716108,0.07){\color[rgb]{0,0,0}\makebox(0,0)[lb]{\small{$\Gamma_3$}}}%
    \put(0.13034575,0.67533117){\color[rgb]{0,0,0}\makebox(0,0)[lb]{\small{$D_R$}}}%
    \put(0.66793204,0.17402376){\color[rgb]{0,0,0}\makebox(0,0)[lb]{\small{$\beta_3$}}}%
    \put(0.85757137,0.37685538){\color[rgb]{0,0,0}\makebox(0,0)[lb]{\small{$\beta_2$}}}%
    \put(0.8015041,0.59452831){\color[rgb]{0,0,0}\makebox(0,0)[lb]{\small{$\beta_1$}}}%
    \put(0.42387446,0.71985521){\color[rgb]{0,0,0}\makebox(0,0)[lb]{\small{$V$}}}%
    \put(0.45,0.055){\color[rgb]{0,0,0}\makebox(0,0)[lb]{\small{$F^4$}}}%
  \end{picture}%
\endgroup%

  \caption{\label{fig:V}
    A domain $V$ as described in Lemma~\ref{lem:boundeddomain}. The domain $V$ is shaded. 
       For clarity, we include only the fundamental domains $F^i$ that comprise the collection $\tilde{F}$; each of these is contained in a tract of $f$ and is 
         adjacent to other fundamental domains, which are not shown.}
         \end{figure}

 The domain $V$ from Lemma~\ref{lem:boundeddomain} allows us to study the accumulation sets of filaments.
   The following lemma is crucial not only in our study of accumulation sets, but also 
    for the proofs of our main theorems. The key idea is that, in order to investigate the landing of
    a given filament $G_{\s}$, we can study a certain chain of simply connected domains (obtained as iterated
    preimages of the domain $V$ from Lemma~\ref{lem:boundeddomain}) whose diameters shrink to zero.

\begin{lem}[Preimage domains]\label{lem:preimagedomains}
   Let $\FF$ be a finite collection of fundamental domains of $f$, {and assume that $\FF$ contains
     every fundamental domain $F$ with $\overline{F}\cap \overline{D}\neq\emptyset$.}
     Let $\zeta$, $R$ and $V$ be as in Lemma~\ref{lem:boundeddomain}. 
     If $R$ was chosen sufficiently large (depending only on $\FF$), then the following holds.
   
  Let $\s=F_0 F_1 \dots $ be any external address. For $n\geq 1$, set
     $\zeta_n(\s) \defeq f^{-n}_{\s}(\zeta)\in \tail_n\defeq \tail_n(\s)$.
     Also let  $V_n= V_{n}(\s)$  be the unique component of $f^{-n}(V)$ containing $\zeta_n(\s)$.
        Then the following hold for all $n\geq 1$. 
     \begin{enumerate}[(1)]
       \item The spherical diameter of $V_n(\s)$ tends to  $0$ as $n\to\infty$.          In particular,
         if $(\omega_n)_{n=0}^{\infty}$ is any sequence with $\omega_n\in V_n(\s)$ for all sufficiently large $n$, then
          $\Lambda(\s)$ coincides with the set $\Lambda_\omega$ of accumulation points of the sequence $(\omega_n)$. \label{item:shrinking}
       \item $\tail_{n+1}(\s) \subset \unbdd{\tail}_{n+1}(\s)\cup V_n(\s) \subset \tail_n(\s)\cup V_n(\s)$ and in fact
                  $\overline{\tail_{n+1}(\s)} \subset \tail_n(\s)\cup V_n(\s)$.\label{item:fundamentalunion}
       \item \ If $F_n\in \FF$, then
          $\zeta_{n+1}(\s)\in V_n(\s)$; in particular $V_n (\s)\cap V_{n+1}(\s)\neq \emptyset$.
\label{item:chain}
     \end{enumerate}
   If $F_k\in\FF$ for all $k\geq0$, then additionally:
   \begin{enumerate}[(1),resume]
       \item $V_n(\s)\cap \mu_{\s}\neq \emptyset$. \label{item:filamentintersection}        \item $f^m\to\infty$ uniformly on 
            $A_n\defeq G_{\s}\setminus \bigcup_{j\geq n} V_j(\s)$. (In particular, 
              $A_n$ is closed and $A_n = J_{\s}\setminus \bigcup_{j\geq n} V_j(\s)$.)
              In fact, if $z$ belongs to this set, then 
               $\lvert f^m(z)\rvert \geq \rho_{m-n}$ for $m\geq n$, where the sequence 
              $(\rho_{\ell})_{\ell=0}^{\infty}$ depends only on $\FF$.\label{item:uniformity}
  \end{enumerate}
\end{lem}
 \begin{proof} 
   By Corollary~\ref{cor:shrinking} the spherical diameter of $V_n(\s)$ tends to  $0$ as $n\to\infty$, uniformly in $\s$. In particular, 
    the set $\Lambda_{\omega}$ in~\ref{item:shrinking} is independent of the choice of the sequence
    $(\omega_n)$. Also recall that $\zeta\in V$ by assumption; if
     we choose $\omega_n = \zeta_n$, then $\Lambda(\s)=\Lambda(\s,\zeta) = \Lambda_\omega$. This proves~\ref{item:shrinking}.

    Part~\ref{item:fundamentalunion} is trivial if $F_n\notin \FF$.
     Indeed, by assumption on $\FF$, we then have 
       $\overline{F_n}\subset \WW$, 
      $\tau_{n+1}(\s)=\unbdd{\tau}_{n+1}(\s)$ and
      $\overline{\tau_{n+1}}(\s)\subset \tau_n(\s)$ by definition. 
       So suppose
        that 
        $F_n\in\FF$. Then  
       $F_n\subset \unbdd{F}_n \cup A_{F_n}$ and 
      $\overline{F_n}\subset \WW \cup A_{F_n}$, where $A_{F_n}\subset V$ is the connected set from 
       Lemma~\ref{lem:boundeddomain}~\ref{item:AF}. 
       
Observe that $\tau_{n+1}(\s)$ is the connected component of $f^{-n}(F_n)$ 
containing $\zeta_{n+1}(\s)$, that $\tau_n(\s)$ is the connected component of
$f^{-n}(\WW)$ containing $\zeta_n(\s)$, and that $V_n(\s)$ is the connected
component of $f^{-n}(V)$ containing $\zeta_n(\s)$. Let $A_1$ be the connected 
component of $f^{-n}(A_{F_n})$ contained in $V_n(\s)$ and let $A_2$ be the 
connected component of $f^{-n}(A_{F_n})$ contained in
$\overline{\tau_{n+1}(\s)}$. Since $\tau_{n+1}(\s)\subset \unbdd{\tau}_{n+1}(\s) \cup A_2$ and $\overline{\tau_{n+1}(\s)} \subset \tau_n(\s) \cup A_2$, we should
show that $A_1 = A_2$. 

Let $x\in A_F\cap \unbdd{F}_n$ be a point that can be connected to $\zeta$ in $\WW\cap V$. Such a point exists by Lemma~\ref{lem:boundeddomain}~\ref{item:AFunbdd}. Let
  $x_n$ be the unique point of $f^{-n}(x)$ in $A_1$. Then $\zeta_n(\s)$ and $x_n$
  belong to the same connected component of $f^{-n}(\WW)$; i.e., 
  $x_n\in \tau_n(\s)$. Now $x\in \unbdd{F}_n$, and $\unbdd{\tau}_{n+1}(\s)$ is the connected
    component of $f^{-n}(\unbdd{F}_n)$ contained in $\tau_n(\s)$. So 
    $x_n \in \unbdd{\tau}_{n+1}(\s)\subset \tau_{n+1}(\s)$, and hence
    $x_n \in A_2$. We have shown $A_1\cap A_2\neq \emptyset$, and therefore
     $A_1=A_2$. Also 
     $\zeta_{n+1}(\s)\in A_2 \subset V$. We have proved both~\ref{item:fundamentalunion}
      and~\ref{item:chain}.

  Now assume that   all fundamental domains $F_n$ occurring in $\s=F_0\ldots F_{n-1} F_n\ldots$ belong to $\FF$. Let $n\geq 0$.

 We next prove~\ref{item:filamentintersection}. For $n\geq 0$, let $X_n\subset \mu_{\sigma^n(\s)}\subset G_{\sigma^n(\s)}$ be a closed unbounded connected set as in 
   Theorem~\ref{thm:unboundedsets}~\ref{item:unboundedsets}. 
   By Theorem~\ref{thm:unboundedsets}~\ref{item:boundedexistence}, if $R$ is large enough 
   (depending only on $\FF$), then $X_n$ can be chosen to contain  a point of radius at most $R$. In particular, 
   $X_n$ intersects the set $A_{F_n}$ from Lemma~\ref{lem:boundeddomain}.
   So if $\tilde{X_n}\subset \mu_{\s}$ is the connected component of $f^{-n}(X_n)$ contained in $\tau_{n+1}(\s)$, 
    then $\tilde{X_n}\cap V_n(\s)\neq \emptyset$.

To prove \ref{item:uniformity}, suppose that $z\in A_n \defeq G_{\s}\setminus \bigcup_{j\geq n} V_j(\s)$. Since
    $z\in G_{\s}$,  there
    is $m_0$ such that $z\in \unbdd{\tail}_m(\s)$ for 
    $m > m_0$. If $m_0$ is minimal with this property,
    then by~\ref{item:fundamentalunion},
    \[ z\in \unbdd{\tail}_{m_0+1}(\s)\setminus \unbdd{\tail}_{m_0}(\s) \subset 
       V_{m_0}(\s). \]
    By assumption on $z$, we must have $m_0 < n$. 
    So, for $m \geq n$, we have 
    $f^m(z) \in \unbdd{F}_m$. Furthermore,
    by the proof of~\ref{item:fundamentalunion}, 
    $f^m(z)\notin A_{F_m}$, and thus $\lvert f^m(z)\rvert > R$. 
    
  So, if $R$ is chosen sufficiently large, $f^n(A_n)$ belongs to the set 
    from Theorem~\ref{thm:unboundedsets}~\ref{item:boundedescape}, on which the iterates
    tend to infinity uniformly and which depends only on $\FF$. 
    The same holds for $f^n(\overline{A_n})$;
     in particular, this set is contained in~$G_{\sigma^m(\s)}$.
\end{proof}

 \begin{proof}[Proof of Proposition~\ref{prop:accumulationsets}]
   Let $\FF$ be a finite collection of fundamental domains containing
     all fundamental domains occurring in $\s$, and also all 
     fundamental domains whose closure intersects $\overline{D}$. Let 
     $\zeta$, $V$ and $V_n$ be as in Lemma~\ref{lem:preimagedomains}.
     Here we assume that $R$ is chosen at least as large as the numbers from
      Theorem~\ref{thm:unboundedsets}~\ref{item:boundedexistence} and~\ref{item:boundedescape}.

      Throughout the proof we will frequently refer to properties \ref{item:shrinking}--\ref{item:uniformity} 
        of filaments $G_{\s}$ with $\s\in\FF^{\infty}$, as established in Lemma~\ref{lem:preimagedomains}.

By~\ref{item:shrinking},
      \[ \Lambda(\s) = \bigcap_{N\geq 1} \closure_{\Ch}\bigl\{\zeta_n(\s)\colon n\geq N\bigr\} = \bigcap_{N\geq 1} \closure_{\Ch}\left(\bigcup_{n\geq N} V_n\right). \]
     Each of the sets in the intersection on the right is compact and connected by~\ref{item:chain}; claim~\ref{item:accumulationconnected} follows. 

  We next prove~\ref{item:accumulationnonuniform}.
   First let $U$ be a neighbourhood of some point $z_0\in \Lambda(\s)$. By~\ref{item:shrinking}, 
   there are infinitely many  $n$ such that 
   $V_n\subset U$. 
   By definition, all points in $V_n$ map to the bounded set $V$ after $n$ iterations, and $V_n$ contains a point of
     $\mu_{\s}$ by~\ref{item:filamentintersection}. Hence $f^n$ does not
     tend to infinity uniformly on $U\cap \mu_{\s}\subset U\cap G_{\s}$, as claimed. Observe that this argument
     also shows that
     \begin{equation}\label{eqn:accumulationsetinclosure}
        \Lambda(\s)\subset \hat{G}_{\s}.\end{equation}

  On the other hand, let $K\subset \hat{G}_{\s}$ be compact with $K\cap \Lambda(\s)=\emptyset$. 
    Then there is $N\geq 1$ such that 
     \[ K \cap \overline{\bigcup_{n\geq N} V_n }= \emptyset. \]
     So $f^n$ tends to infinity uniformly on
       $K$ by~\ref{item:uniformity}, and in particular $K\subset J_{\s}$. This proves~\ref{item:uniformawayfromaccumulation}.

  Now we establish the remaining claims in Proposition~\ref{prop:accumulationsets}. 
    Let $z\in \overline{G_{\s}}\setminus\Lambda(\s)$. By~\ref{item:uniformawayfromaccumulation}, applied to $K=\{z\}$,
      we see that $z\in G_{\s}$. Thus
         \[ \hat{G}_{\s} = \overline{G_{\s}}\cup\{\infty\} \subset G_{\s}\cup \Lambda(\s) \cup\{\infty\}. \]
            Together with~\eqref{eqn:accumulationsetinclosure}, this proves~\ref{item:accumulationsetinclosure}.

  Next suppose that $z_0\in J_{\s} \setminus G_{\s}$; that is, $z_0$ has address $\s$ but is not escaping. There is $n_0$ such that $f^n(z_0)\in J_{\sigma^n(z)}^0$ for $n\geq n_0$. Hence
     \(\lvert f^n(z_0)\rvert < R \) by Theorem~\ref{thm:unboundedsets}~\ref{item:boundedescape}. So $f^n(z_0)\in V$ for $n\geq n_0$, 
     and thus $z_0\in V_n$. By~\ref{item:shrinking}, 
     this proves $\Lambda(\s)= \{z_0\}$, as claimed.  

  It remains to prove~\ref{item:disjointtype}. Recall from  Lemma~\ref{lem:control} that, if $\s$ is of disjoint type, then
    $J_{\sigma^n(\s)} = \overline{G_{\sigma^n(\s)}}\subset F_n$ for all $n\geq 0$. 
   {Let $X$ be the set from~\ref{thm:unboundedsets}~\ref{item:boundedescape}.
    Then there is $n_0$ such that $\lvert f^n(z)\rvert > R$ for all $n\geq n_0$ and all $z\in X$. 
    Set 
     \[ \tilde{R}\defeq \max\{ \lvert f^j(z)\rvert \colon \lvert z\rvert \leq R\text{ and } j \leq n_0\}. \]

  By Theorem~\ref{thm:unboundedsets}~\ref{item:boundedexistence}, for all $n\geq 0$ there is 
      $\zeta_n\in J_{\s}$ such that $\lvert f^n(\zeta_n)\rvert \leq R$. We claim that 
     $\lvert f^j(\zeta_n)\rvert \leq \tilde{R}$ for all $n\geq n_0$ and all
    $j\leq n$. Indeed, let $j$ be minimal such that
     $\lvert f^j(\zeta_n)\rvert >R$ (if no such $j$ exists, there is nothing to prove). Then $f^j(\zeta_n)\in X$, and 
     hence we must have $j> n-n_0$, and the claim follows by the definition of $\tilde{R}$.}

    Let $z_0\in J_{\s}$ be a limit point of the sequence $(\zeta_n)$; then all points on the orbit of $z_0$
     have modulus at most $\tilde{R}$. The claim now follows from~\ref{item:landingfromaddress}. 
 \end{proof}

\begin{rmk}[Coding trees]
  Fix $\FF$, $\zeta$ and $V$ as in Lemmas~\ref{lem:boundeddomain} and~\ref{lem:preimagedomains}. For each $F\in\FF$, we can choose an arc
     $\gamma_F$ connecting $\zeta$ to the point $\zeta_F$, and we may assume that these arcs are disjoint except at $\zeta$. For each 
     $\zeta_{F_0}$ and each arc $\gamma_{F_1}$, there is a component of $f^{-1}(\gamma_{F_1})$ connecting $\zeta_{F_0}$ to
     some point $\zeta_{F_0 F_1}$ of $f^{-2}(\zeta)$. By Lemma~\ref{lem:preimagedomains}, this is precisely the point contained
     in the fundamental tail at address $F_0 F_1$. 

   Continuing inductively, we obtain an infinite tree with root $\zeta$, whose vertices of depth $n>0$ are the elements of $f^{-(n)}(\zeta)$ contained
     in fundamental domains of level $n-1$ whose addresses contain only entries from $\FF$, and whose edges are
     all components of $f^{-n}(\gamma_F)$ for some $F\in\FF$. Recall that the spherical length of these edges tends to zero as $n\to\infty$. 

   This tree can be considered to be an analogue of the \emph{geometric coding tree} used by Przytycki \cite{Pr94} in the case of
     rational functions. We see that, for each address $\s$ whose entries are drawn from $\FF$, the accumulation set of
     the  filament $G_{\s}$ coincides precisely with the accumulation set of a branch of this coding tree. However, we will not use this language
     in the following. 
\end{rmk}

\section{Separation properties of filaments}\label{sec:separation}

 We now prove that a filament that lands at a non-escaping point $z_0\in\C$ 
  does not separate the plane. This fact is not used in our
  paper (except to deduce the corresponding parts of 
  Theorems~\ref{thm:landingsets} and~\ref{thm:hyperbolicsetsintroduction}), but is
   important for applications. 

 \begin{thm}[Filaments do not separate]\label{thm:separation}
  Let $f$ be a postsingularly bounded function $f$, and let $\s$ be 
    a bounded external address of $f$. Assume that $G_{\s}$ lands at a point
    $z_0\in \C\setminus G_{\s}$. Then $\overline{G_{\s}}$ does not separate the plane. 
 \end{thm}
 
  It is plausible that this can be directly deduced from our results and techniques
    in Section~\ref{sec:accumulationlanding}; indeed Pfrang does this for
    postsingularly finite $f$ \cite{pfrangthesis}. Instead, we deduce 
    Theorem~\ref{thm:separation} by relating
    landing filaments to Julia continua of disjoint-type entire functions. Recall from Remark~\ref{rmk:disjointtype} that
     a function $g$ is of \emph{disjoint type} if $g$ is hyperbolic with connected Fatou set, and that a \emph{Julia continuum} of $g$ is a set 
     of the form $\hat{C}=C\cup\{\infty\}$, where $C$ is 
     a connected component of $J(g)$.

\begin{thm}[Filaments and Julia continua]\label{thm:juliacontinua}
 Let $f$ be a postsingularly bounded function, and let $\s$ be a bounded (resp.\ periodic)
   external address of $f$ such that $G_{\s}$ lands. If $\lambda$ is
   sufficiently small to ensure that $g\colon z\mapsto \lambda f(z)$ is of disjoint type,
   then $\hat{G_{\s}}$ is homeomorphic to a 
   Julia continuum 
   $\hat{C}$ of $g$ at a bounded (resp.\ periodic) external address. 
     
  The homeomorphism can be chosen to
    fix $\infty$, and send $z_0$ to the unique point of bounded orbit
   in $\hat{C}$.
\end{thm}
  Much is known about the topology of Julia continua of disjoint-type entire functions;
    see \cite{arclike}. Theorem~\ref{thm:juliacontinua} allows us to transfer this
    information to landing filaments. In particular, we can easily deduce 
    Theorem~\ref{thm:separation}.

  \begin{proof}[Proof of Theorem~\ref{thm:separation}, using Theorem~\ref{thm:juliacontinua}]
Let $g$ be a disjoint-type function
       as in Theorem~\ref{thm:juliacontinua}. 
      Then the Fatou set $F(f)$ is connected and non-empty by definition.
       Hence $J(f)$ is a nowhere dense set that does not separate the plane, so no subset
       of $J(f)$ separates the plane. 

     So by Theorem~\ref{thm:juliacontinua}, the set
     $\hat{G}_{\s}$ is homeomorphic to a 
      non-separating plane continuum.        
      It is well-known that being  a one-dimensional 
      non-separating plane continuum is a topological property. 
       (Indeed, a one-dimensional plane continuum is non-separating
         if and only if it is tree-like; see  
          \cite[Theorem~6]{bing1951} and \cite[Theorem~1.5]{manka2012}. Compare also
         \cite{jungcktimm}.)  So $\hat{G}_{\s}$ is also non-separating. 
   \end{proof} 

\begin{proof}[Proof of Theorem~\ref{thm:juliacontinua}]
  By \cite[Example on p.~392]{BK07}, for $\lambda\in\C$ small enough the function $g$
     is indeed of disjoint type; we fix such $\lambda$ in the following. 
   By \cite[Theorem~1.1]{Re09}, there is a map $\theta$ 
    defined on 
 \[ J_{\geq R}(f)\defeq\{z\in \C\colon |f^n(z)|>R \text{ for all $n>1$}\}, \]
     which is a conjugacy between
       $f$ on $J_{\geq R}(f)$ and $g$ on $\theta(J_{\geq R}(f))$. 
    Furthermore, $\theta$ extends continuously to $\infty$ with $\theta(\infty)=\infty$~--
        in particular, $\theta$ maps escaping points of $f$ to escaping points of $g$~--
 and 
    $J_{\geq Q}(g)\subset \theta(J_{\geq R}(f))$ for some $Q$ \cite[Lemma~3.3]{Re09}. 

Assuming that $R>0$ is sufficiently large, the proof of \cite[Theorem~3.2]{Re09}
  furthermore implies the following.
 For any external address 
 address $\s$ of $f$, there is an external address $\tilde{\s}$ of $g$ such that
 \begin{equation}\label{eqn:corresponding address}
 \theta(J_\s^0(f)\cap J_{\geq R}(f))\subset J^0_{\tilde{\s}}(g).
 \end{equation}
 If $\s$ is bounded (resp.\ periodic), then so is $\tilde{\s}$. 
  We can extend $\theta|_{J_{\geq R}(f)\cap I(f)}$ to a bijection
    $\theta\colon I(f)\cup\{\infty\}\to I(g)\cup\{\infty\}$ by defining 
  \[ 
   \theta(z)\defeq g_{\tilde{\s}}^{-n}(\theta(f^n(x))).
  \] 
   The value $\theta(z)$ is independent of $n$ and, in particular,
     agrees with the original value when $z \in J_{\geq R}(f)\cup I(f)$.
       This follows from~\eqref{eqn:corresponding address} 
      and Proposition~\ref{prop:externaladdresses} and the conjugacy relation for $\theta$. 
      Note that we do not claim that this bijection is continuous on 
      $I(f)$. 
  
Now suppose that $\s$ is bounded and the filament $G_\s(f)$ lands at a point
 $x_0\in \C\setminus G_{\s}$. Note that 
  the filament $G_{\tilde{\s}}$ of $g$ also lands at some point
  $y_0\in\C$ of bounded orbit since
   the function $g$ is of disjoint type. 
    (Recall Proposition~\ref{prop:accumulationsets}~\ref{item:disjointtype}.)

  Consider the closures  $\hat{G}_\s(f)$, $\hat{G}_{\tilde{\s}}(g)$ of 
     $G_\s(f)$ and $G_{\tilde{\s}}(g)$ in $\Ch$. For $n\in\N$ define 
 $$
X_n\defeq f_{\s}^{-n}(J_{\sigma^n(\s)}^0(f)\cap J_{\geq R}(f))\cup\{\infty\} \subset \hat{G}_{\s}.
$$ 
 By definition, $X_n\supset X_{n-1}$,
   and $\theta$ is continuous when restricted to $X_n$. We claim that
   $\theta$ is continuous on $X\defeq \bigcup_n X_n = G_{\s}(f)\cup\{\infty\}$. 

 Let $x\in X$. By assumption, 
   $x\notin \Lambda(\s) = \{x_0\}$. 
    Hence by Corollary~\ref{cor:alternativelanding}, $z$ has a  
    neighborhood $U$ in $X$ on which the iterates escape to infinity uniformly.
    Then, for sufficiently large $n$, $f^n(U)\subset J_{\s}^0(f)$, 
    and hence $U\subset X_n$. Since $\theta$ is continuous on $X_n$ and $U$
    is a neighbourhood of $x$, 
      $\theta$ is continuous at $x$. 

  Moreover, $\theta^{-1}$ is continuous on the sets
      \[ Y_n \defeq g_{\s}^{-n}(J_{\sigma^n(\tilde{\s})}^0(g)\cap J_{\geq Q}(g))\cup\{\infty\}. \] 
      Hence, by the same argument as for $f$, the map $\theta^{-1}$ is 
       continuous on $Y = G_{\tilde{\s}}(g)\cup\{\infty\}$, and 
       $\theta\colon X\to Y$ is a homeomorphism. 
       In particular, it extends to a homeomorphism between their respective
       one-point compactifications $\hat{G}_{\s}(f) = X\cup \{x_0\}$ and
       $\hat{G}_{\tilde{\s}}(g) = Y\cup \{y_0\}$. 
\end{proof} 
  
\begin{remark}
 It is plausible that the homeomorphism in Theorem~\ref{thm:juliacontinua} is \emph{ambient}; 
    i.e., it extends to a homeomorphism of $\C$ onto itself. 
\end{remark}  

\section{Landing theorems for filaments}\label{sec:maintheorems}
 With 
  the definition of filaments in Definition~\ref{defn:filament} and
   their accumulation sets in Definition~\ref{defn:accumulationset2}
    we can now state the main result of our paper.  
  Recall that a filament $G_{\s}$ is \emph{periodic} if the address $\s$ is periodic under the shift map. Equivalently, 
        $G_{\s}$ is periodic under the action of $f$ as a subset of $\C$.

  \begin{thm}[Douady-Hubbard landing theorem for filaments]\label{thm:DHfilaments}
    Let $f$ be a transcendental entire function whose post-singular set $\P(f)$ is bounded. 

   Then every periodic filament of $f$ lands at a repelling or parabolic periodic point, and conversely every
    repelling or parabolic periodic point of $f$ is the landing point of at least one and at most finitely many periodic filaments. 
 \end{thm} 

  We also obtain a corresponding theorem about landing properties at more general hyperbolic sets. 

 \begin{thm}[Landing at hyperbolic sets]\label{thm:mainhyperbolic}
   Let $f$ be a transcendental entire function with bounded postsingular set. Moreover, suppose that $K$ is a hyperbolic set for $f$. 
    Then every point $z\in K$ is the landing point of a filament at some bounded external address $\s$; 
    if $z$ is periodic, then so is $\s$. 
 \end{thm}

In fact, we establish the following more precise version of Theorem~\ref{thm:mainhyperbolic}. Note that
   the space of external addresses is equipped with the product topology; this is the same as the
    order topology arising from the cyclic order on addresses. 

  \begin{thm}[Accessibility of points  in hyperbolic sets]\label{thm:hyperboliclanding_abstract}
   Let $K$ be a hyperbolic set of a postsingularly bounded entire function $f$. Then there exists a compact 
     and forward-invariant set $\mathcal{S}_K$ of bounded addresses of $f$, with the following properties. 
   \begin{enumerate}[(a)]
     \item\label{item:filamentslanding}%
        For every $\s\in \mathcal{S}_K$, the filament at address $\s$ lands at a point $z_0(\s)\in K$.
     \item\label{item:continuouslanding}%
       The function $\mathcal{S}_K\to K; \s\mapsto z_0(\s)$ is surjective and continuous. 
        In particular, every point of $K$ is the landing point of a filament. 
     \item\label{item:accessibleperiodic}%
         If $z_0$ is periodic, then all bounded-address 
      filaments landing at $z_0$ are periodic with the same period, and the number of such filaments is finite. 
     \item\label{item:uniformlanding} The filaments at addresses in $\mathcal{S}_K$ land uniformly at $K$, in the following sense.
          Let $\zeta\in\WW$ and let $\eps>0$. Then there is $n_0$ such that
         $\dist(\zeta_n(\s),K)\leq \eps$ for all  $n\geq n_0$ and all  $\s\in \mathcal{S}_K$.
 (Recall from Definition~\ref{defn:accumulationsetabstract} that $\zeta_n(\s) = f_{\s}^{-n}(\zeta)$.)
   \end{enumerate}
  \end{thm}

Note that we do not claim that $\mathcal{S}_K$ can be chosen to consist of \emph{all} addresses of filaments landing at points of $K$.
 In particular, we do not prove that the function $\s\mapsto z_0(\s)$ is continuous on the latter set.
 
 The remainder of the paper will be dedicated to the proofs of these theorems. Let us first show that they
     imply 
      the theorems stated
   in the introduction. 

 \begin{proof}[Proof of Theorem~\ref{thm:landingsets}, using Theorem~\ref{thm:DHfilaments}]
  Let $\zeta$ be a repelling or parabolic periodic point. By Theorem~\ref{thm:DHfilaments}, $\zeta$ is the landing point of a periodic
    filament $G_{\s}$. Let us set $A \defeq G_{\s}$, and let $p$ be the period of $\s$. 
   Recall that $G_{\s}\subset I(f)$ by definition, and is unbounded and connected by Proposition~\ref{prop:core}.
    Since $G_\s$ lands at $\zeta$, we see from Corollary~\ref{cor:alternativelanding} 
    that $\ov{A}=A\cup\{\zeta\}$ and that,
   for any neighborhood $U$ of $\zeta$, $f^n\to\infty$ uniformly on $A \setminus U$. 
   Furthermore, by Theorem~\ref{thm:separation}, $\overline{A}$ 
    does not separate the plane.
   This proves that $A$ satisfies properties~\ref{item:Aclosure} and~\ref{item:Auniform} 
    of Theorem~\ref{thm:landingsets}. Since $p$ is the period of $\s$, we have $f^p(A)=G_{\sigma^p(\s)}=A$
    and  $f^j(A)\cap A=G_{\sigma^j(\s)}\cap G_{\s} = \emptyset$ if $0<j< p.$
 \end{proof}

 \begin{proof}[Proof of Theorem~\ref{thm:mainrays}, using Theorem~\ref{thm:DHfilaments}]
  Recall that every periodic hair is a periodic filament (Proposition~\ref{prop:periodichairfilament}), and for criniferous functions, every periodic filament is a periodic hair
  (Corollary~\ref{cor:filamentscriniferous}).
    Moreover, {by Corollary~\ref{cor:hairaccumulation}}
    such a filament lands if and only if the corresponding hair lands in the sense of Definition~\ref{defn:periodichair}. 

  Thus Theorem~\ref{thm:mainrays} follows immediately from Theorem~\ref{thm:DHfilaments}.  
 \end{proof} 

\begin{proof}[Proof of Theorem~\ref{thm:hyperbolicsetsintroduction}, using Theorem~\ref{thm:hyperboliclanding_abstract}]
  Define $\mathcal{A}\defeq \{G_{\s}\colon \s\in \mathcal{S}_K\}$,
    where $\mathcal{S}_K$ is as in Theorem~\ref{thm:hyperboliclanding_abstract}. 
   Properties~\ref{item:hyperbolicclosure},~\ref{item:hyperbolicimage} and~\ref{item:hyperbolicperiod} in Theorem~\ref{thm:hyperbolicsetsintroduction}
   follow immediately, as in the proof of Theorem~\ref{thm:landingsets}. Furthermore, 
   if $f$ is {criniferous}, then every element of $\mathcal{A}$ is an arc connecting its landing point $\zeta(A)$ to
   $\infty$, by Corollary~\ref{cor:hairaccumulation}. 

   Thus it remains to establish~\ref{item:hyperbolicuniform}. Observe that, since $\mathcal{S}_K$ is 
    compact and forward-invariant, only finitely many different fundamental domains can occur within the
    addresses in $\mathcal{S}_K$. Let $\FF$ be the collection of all these fundamental domains,
    as well as of all fundamental domains whose closure intersects $\overline{D}$. Let $V$ be
    as in Lemma~\ref{lem:preimagedomains}, and let $n_0$ be as in
    part~\ref{item:uniformlanding} of 
    Theorem~\ref{thm:hyperboliclanding_abstract} for $\eps/2$. By Lemma~\ref{lem:preimagedomains},
    we can find $n_1\geq n_0$ such that $\diam V_n(\s) < \eps/2$ for $n\geq n_1$. Thus
     $\dist(z, K)<\eps$ for all $z\in V_n(\s)$.

   Let $z\in G_{\s}$ for some $\s\in\mathcal{S}_K$, and suppose that $\dist(z,K)\geq \eps$.
     Then, by the above $z \in G_{\s}\setminus \bigcup_{j\geq n_1}V_j(\s)$. By~\ref{item:uniformity} of 
      Lemma~\ref{lem:preimagedomains},
      the set of points with this property escape to infinity at a rate that depends only
      on $\FF$ and on $n_0$ (and hence only on $\eps$). This completes the proof.
\end{proof}

\begin{rmk}[Shortcut to the landing theorems]\label{rmk:shortcut}
Observe that Definition~\ref{defn:accumulationsetabstract}, concerning accumulation sets,
  requires only the definition of the fundamental tails $\tau_n(\s)$ associated to an address $\s$, rather than
   any properties of the filament $G_{\s}$ itself. Moreover, a key point in our 
    proofs of Theorems~\ref{thm:DHfilaments},~\ref{thm:mainhyperbolic} 
    and~\ref{thm:hyperboliclanding_abstract} is that we can ignore the fine structure of
    filaments and 
    use only the sets $V_n$ from Lemma~\ref{lem:preimagedomains}. We only require 
    properties~\ref{item:shrinking}, \ref{item:fundamentalunion} and~\ref{item:chain} of this lemma,
    all of which are likewise independent of the construction and analysis of filaments in
    Sections~\ref{sec:unboundedsets},~\ref{sec:filaments} and~\ref{sec:hairs}.

   Hence it would be possible to prove these theorems without requiring any 
     results from those sections, 
     and using 
     only the elementary parts of Section~\ref{sec:accumulationlanding}. Furthermore, while we used 
     properties of filaments above to deduce the statements of Theorems~\ref{thm:landingsets} and~\ref{thm:hyperbolicsetsintroduction},      many of these can be established more easily \emph{a posteriori} for the filaments under consideration, using the 
     additional information that these filaments land. Nonetheless, the
     general material concerning filaments is crucial to the interpretation of our results as a natural
     extension of the classical Douady-Hubbard landing theorem for functions without hairs.  
\end{rmk}

\section{Periodic filaments land}\label{Periodic filaments land}

  We now prove the first half of Theorem~\ref{thm:DHfilaments}, which we restate here as follows.

 \begin{thm}[Landing of periodic filaments]\label{thm:landingperiodic}
   Let $f\colon\C\to\C$ be a transcendental entire function with bounded postsingular set.

   Then every periodic filament of $f$ lands at a repelling or parabolic periodic point of $f$, 
     where the period of the landing point divides that of the filament. 
 \end{thm}
 \begin{proof}[Proof of Theorem~\ref{thm:landingperiodic}]
  Let $G_\s$ be a periodic filament. Recall from  Observation~\ref{obs:filamentiterates}
    that any filament of $f$ is also a filament of any iterate of $f$ (and vice versa). 
     Hence, it is no loss of generality to assume
    that $\s$ is fixed by $\sigma$; i.e., $\s= F F F F \dots$, where $F$ is a fundamental domain of $f$.

  Fix $\zeta$ and $V$ as in Lemmas~\ref{lem:boundeddomain} and \ref{lem:preimagedomains}. Let $\tail_n$ be the fundamental 
     tail of level $n$ associated to $\s$, and let $\zeta_n$ be the unique element of $f^{-n}(\zeta)$ in $\tail_n$. 

   Recall that $\zeta_1 = \zeta_F \in V$, and let $\gamma\colon [0,1]\to V$ be a rectifiable curve with $\gamma(0)=\zeta_1$ and $\gamma(1)=\zeta$. Then, by Lemma~\ref{lem:preimagedomains}, 
    for every $n\geq 1$ there is a component $\gamma_n$ of $f^{-n}(\gamma)$ connecting $\zeta_n$ and $\zeta_{n+1}$. 
   We can hence combine these to a continuous curve $\gamma\colon [-\infty,1]\to\C$,   with $f(\gamma(t))=\gamma(t+1)$ for $t\leq 0$, where
    $\gamma|_{[0,1]} = \gamma$, and $\gamma|_{[-n,1-n]} = \gamma_n$. 
    
 {By Lemma~\ref{lem:preimagedomains}~\ref{item:shrinking}, 
the spherical diameter of $\gamma_n$ tends to zero as $n\to\infty$. Therefore
    the set $\Lambda(\gamma)$ of accumulation points of $\gamma(t)$ as $t\to -\infty$ is precisely
    the accumulation set $\Lambda(G_\s)=\Lambda(\s,\zeta)$ of $G_\s$ in the sense of Definition~\ref{defn:accumulationsetabstract}.}

  \begin{claim}
   As $t\to -\infty$, $\gamma(t)$ converges to a fixed point of $f$. 
  \end{claim} 
 \begin{subproof}  Recall that the spherical distance between $\zeta_n$ and $\zeta_{n+1}$ tends to zero as $n\to \infty$, 
    and $f(\zeta_{n+1})=\zeta_n$. Hence, by continuity, 
     any finite point of  $\Lambda(\gamma) = \Lambda(\s,\zeta)$ is a fixed point of $f$. 
    Since the set of fixed points is discrete in $\C$, and $\Lambda(\gamma)$ is connected,
    the latter set is a singleton, whose sole element is either a fixed point or $\infty$. We must exclude the second possibility.

So suppose, by contradiction, that $\gamma(t)\to\infty$ as $t\to-\infty$. 
   Recall that $\Omega$ is the unbounded connected component of $\C\setminus \P(f)$.
    Let $\rho_n\defeq \ell_{\Omega}(\gamma_n)$ be the hyperbolic length of
    $\gamma_n$ in $\Omega$. 
 Since $\gamma_n\to\infty$ and the postsingular set is bounded, formula~\eqref{eqn:strongexpansion} of Proposition~\ref{prop:expansion} implies that
     \[ \rho_{n+1} \leq \rho_n/2 \]
     for all sufficiently large $n$. It follows that the hyperbolic length of  $\gamma$ is bounded. 
     As the hyperbolic metric on $\Omega$ is complete, this contradicts our assumption that $\gamma$ tends to
     $\infty$.
 \end{subproof} 

  The classical snail lemma of Douady and Sullivan (see \cite[Expos\'e VIII, Proposition 2, p. 60]{DH84} or \cite[Lemma 16.2]{Mi}) 
    shows that the limit point of $\gamma$ is either repelling or parabolic with multiplier $1$, 
     and the proof is complete. 
 \end{proof} 
\begin{remark}
  In fact, the above claim follows already from~\cite[Theorems~B.1 and~B.2]{Re08}. 
     Here Theorem~B.1 is a hyperbolic
     expansion argument going back to the proof by Douady and Hubbard~\cite[Expos\'e VIII]{DH84} of the first half of the 
     landing theorem for polynomials. On the other hand, the 
      proof of \cite[Theorem~B.2]{Re08}, which shows that $\gamma(t)$ cannot 
      converge to infinity as $t\to-\infty$,  
      used the notion of ``extendability'', which was developed for more general purposes in \cite{Re08}. 
      For the reader's convenience, we gave the complete and much simpler proof
      above, in the spirit of
      Deniz~\cite{De}.
\end{remark}

{It is possible that the landing point in Theorem~\ref{thm:landingperiodic} is also the landing point
   of other filaments. However, as we now observe, there can only be finitely many of these,
    and they all need to be periodic. The idea of the proof is very similar to the polynomial case
     \cite[Lemma~18.12]{Mi}, but we need to take into account the 
      non-compactness of the space of addresses. Compare also~\cite[Lemma~3.2]{expcombinatorics} for the  case of exponential maps.}
\begin{lem}\label{lem:One for all} Let $f$ be a transcendental entire function with bounded
   postsingular set, and let $z$ be the landing point of a filament with periodic address. Then the number of 
    bounded-address filaments landing at $z$ is finite, and their addresses are all periodic of the same period.
\end{lem}
  \begin{proof}   
   Let $\s^0$ be the address of the periodic filament landing at $z$; we may assume that its period
    $p$ is minimal with this property. 
    Replacing $f$ by $f^p$, we may then assume that $p=1$. 
    Let $\mathcal{S}_z$ be the set of bounded external addresses $\s$ for which  $G_{\s}$ lands at
     $z$. 

  \begin{claim}
    There is a finite collection  $\FF$ of fundamental domains with the following property.
     Every address $\s\in\mathcal{S}_z$ contains some element of $\FF$ infinitely many times. 
  \end{claim}
  \begin{subproof}
    Let $\FF$ consist of all fundamental domains $F$ such that either $\overline{D}\cap \overline{F}\neq\emptyset$,
     or $z\in F$. Clearly $\FF$ is finite. Suppose that $\s = F_0 F_1 F_2 \dots$ is such that
     $F_n\notin \overline{F}$ for $n\geq n_0$. Then $\sigma^{n_0}(\s)$ is of disjoint type, and
     $z\notin F_{n_0} \supset f^{n_0}(\overline{G_{\s}})$. 
     In particular, $\s$ does not land at $z$. 

    So any address in $\mathcal{S}_z$ contains infinitely many entries from  $\FF$; since 
      the latter set is finite, at least one of these is itself repeated infinitely many times. 
  \end{subproof}

   Since $f$ maps a neighborhood of $z$ to another neighborhood of $z$
     as an orientation-preserving homeomorphism, it
      preserves the cyclic order of the filaments landing at $z$. 
      As remarked at the end of Section~\ref{sec:cyclicorder}, 
      this implies that $f$ also preserves the cyclic order of these filaments at $\infty$. In other words, 
      the shift map $\sigma\colon \mathcal{S}_z \to \mathcal{S}_z$ is injective and 
      preserves the cyclic order of addresses 
      on $\mathcal{S}_z$.

     Recall that $\s^0$ is a fixed address, say $\s^0 = F^0 F^0 F^0 \dots$. 
      Hence $\sigma$ also preserves the (linear) order $<$, 
      where $\s < \tilde{\s}$ means that $\s^0\prec \s \prec \tilde{\s}$ in the cyclic ordering. 
      So if $\s = F_0 F_1 F_2\dots$ is an element of $\mathcal{S}_z$, then 
      $(\sigma^n(\s))_{n=0}^{\infty}$ is monotone. If $F\in\FF$ is the domain from the claim above,
      then we clearly must have $F_n = F$ for all sufficiently large $n$. By injectivity of $\sigma$, we conclude
      that $\s = F F F \dots$ is itself a fixed address. As $\FF$ is finite, the proof is complete.
 \end{proof}

\begin{rmk}\label{rmk:nounboundedatperiodic}
 Recall that we defined filaments only for maps with bounded postsingular sets, and 
   landing of filaments only for filaments at bounded addresses. This allows us to state the lemma
   in the above simple form, which is all that will be required for the purpose of our main results.

 However, observe that the proof of the lemma is purely combinatorial, and does not utilise either assumption in
   an essential manner. In particular, let $f\in\B$ be arbitrary (not necessarily with bounded postsingular set), 
   and let $z_0\in J(f)$ be  a fixed point of $f$. Suppose that $\mathcal{S}_z$ is a forward-invariant set of
   external addresses $\s$ of $f$, and that for every $\s\in\mathcal{S}$ there is an unbounded connected set
   $A_{\s}$ with the following properties:
   \begin{itemize}
     \item the cyclic order of the sets $A_{\s}$ at infinity agrees with the
        cyclic order of their external addresses; 
     \item $z_0\in\overline{A_{\s}}$ for all $\s\in \mathcal{S}_z$;
     \item $\overline{A_{\s}}\cap A_{\tilde{\s}}=\emptyset$ for $\s\neq \tilde{\s}$;
     \item $\overline{A_{\s}}$ does not separate the plane; 
     \item $f(A_{\s}) = A_{\sigma(\s)}$. 
   \end{itemize}

  If $\mathcal{S}$ contains a periodic element, then it follows as above that $\mathcal{S}_z$ is finite and 
    contains only periodic addresses. In particular, for any $f\in\B$, the landing point of a periodic hair
    cannot be the landing point of a non-periodic hair.  
\end{rmk}

\section{Landing at hyperbolic sets}\label{sec:Landing at hyperbolic sets}

\begin{proof}[Proof of Theorem~\ref{thm:hyperboliclanding_abstract}]
Let $K$ be a hyperbolic set of $f$. Replacing $f$ by a sufficiently high iterate, 
there is a neighbourhood $U$ of $K$ such that $\lvert f'(z)\rvert \geq 2$ for all $z\in U$. 
  (It is easy to see that proving the theorem for an iterate of $f$ also establishes it for $f$ itself;
     we leave the details to the reader.)
 We may additionally assume that the disc $D$ in the definition of fundamental domains is chosen so large that
    $K\subset D$. Finally, by Corollary~\ref{cor:shrinking}, it is enough to 
     prove~\ref{item:uniformlanding} 
    for some specific choice of
     $\zeta\in W_0$. Therefore we may fix $\zeta$ belonging to an unbounded
     connected component of 
     $W_0\cap f^{-1}(D)$, as required in the hypothesis of Lemma~\ref{lem:boundeddomain}.
    
   Set $\delta \defeq \dist(K, \partial U)$. For $z\in K$, define 
   $B_0(z)\defeq B(z,\delta)$, and let
   $B_n(z)$ denote the connected component of $f^{-n}(B_0(f^n(z)))$ containing $z$.  
   Then  $\overline{B_{n+1}(z)}\subset B_n(z)$ for all $n$,  and $f\colon B_{n+1}(z)\to B_{n}(f(z))$ is a conformal isomorphism.  For $n\geq 0$, define 
     \[ U_n \defeq \bigcup_{z\in K} B_n(z). \]  
     Then $\overline{U_{n+1}}\subset U_n\subset U$ and $f(U_{n+1})=U_n$ for all $n$. 
    Clearly $K=\bigcap_{n\geq 0} \overline{U_n} = \bigcap_{n\geq 0} U_n$.
    By the blowing-up property of the Julia set (see e.g.\ \cite[Lemma 2.2]{Ba84}), 
     and compactness of $K$,     there is    some $N_1$ with the following property:
    if $n\geq N_1$ and $z\in K$, then $f^{-n}(\zeta)\cap B_0(z)\neq\emptyset$. 
    In particular, for all such $n$ and $z$
     there is 
     a finite external address of length $n$ such that $\zeta_n(\s)\in B_0(z)$, and hence 
      $\tail_n(\s)\cap B_0(z)\neq \emptyset$.
  Define 
\[ \eps \defeq \frac{\delta}{2} \leq \min_{z\in K} \dist(\partial B_0(z) , \overline{B_1(z)}). \]

 Now let $V$ be as in Lemma~\ref{lem:boundeddomain} for the finite collection $\FF$ of fundamental domains that intersect $D$.  
   Similarly as in Lemma~\ref{lem:preimagedomains}, if $\s$ is an infinite external address or a finite
   external address of length at least $n$, we define 
     $V_{n}(\s)$  to be the unique component of $f^{-n}(V)$ containing $\zeta_n(\s)\defeq f^{-n}_{\s}(\zeta)\in \tail_n(\s)$.

     By Lemma~\ref{lem:Euclideanshrinking}, there exists $N\geq N_1$ with the following property. 
     Suppose that $n\geq N$ and that $\s$ is a finite or infinite external address of length at least $n$ 
      with $V_n(\s)\cap \overline{U_0} \neq \emptyset$. Then
 \begin{equation}
 \diam V_{n}(\s)< \eps.
 \end{equation}

For $n\geq N$ and $z\in K$, we now define $\mathcal{S}_n(z)$ to be the set of finite external addresses $\s$
      of length $n$ 
      for which the tail $\tau_{n}(\s)$ intersects $B_{n-N}(z)$. Observe that  
     \[ \mathcal{S}_n\defeq \bigcup_{z\in K} \mathcal{S}_n(z)\] is
      finite for every $n$ by Lemma~\ref{lem:preimageK}. We also define $\mathcal{S}_K$ to be the 
      set of infinite addresses $\s$ such that every prefix of length $n\geq N$ of $\s$ is an element 
      of $\mathcal{S}_n$. In order to show that this set has the properties 
      asserted in Theorem~\ref{thm:hyperboliclanding_abstract}, we investigate the
      sets $\mathcal{S}_n(z)$ more closely.
 
\begin{claim}[Claim  1]
   The following hold for all $n\geq N$ and all $z\in K$. 
    \begin{enumerate}[(i)]
      \item $\mathcal{S}_N(z)\neq \emptyset$.\label{item:B1}
      \item The shift map $\sigma$ maps $\mathcal{S}_{n+1}(z)$ onto $\mathcal{S}_n(f(z))$.\label{item:Bshift}
      \item $\mathcal{S}_n(z)\neq\emptyset$.\label{item:Bnonempty}
 \item Suppose that $\s=F_0 F_1 \dots F_{n}\in \mathcal{S}_{n+1}(z)$.
                Let $\pi_n(\s)$ be the prefix of length $n$ of $\s$; i.e., 
               $\pi_n(\s) = F_0 F_1 \dots F_{n-1}$. 
                Then $\pi_n(\s)\in \mathcal{S}_n(z)$.\label{item:Bprefix}   
     \item In particular, $\pi_N(\sigma^j(\s))\in \mathcal{S}_N\bigl(f^j(z)\bigr)$ for $j=0,\dots,n-N$.\label{item:Biteratedprefix}
    \end{enumerate}
\end{claim} 
Nb.\ In~\ref{item:Bshift}, we would like to claim that
  $\sigma\colon \mathcal{S}_{n+1}(z)\to\mathcal{S}_n(f(z))$ is a bijection. However, it is conceivable
  that a tail $\tau$ of level $n$ intersects $B_{n-N}(f(z))$ in more than one connected component, and that
  $\tau\cup B_{n-N}(f(z))$ surrounds a singular value of $f$. In this situation, there may be two
  different components of $f^{-1}(\tau)$ that intersect $B_{n+1-N}(z)$, and $\sigma$ may therefore not be injective 
  on $\mathcal{S}_{n+1}(z)$. Nonetheless, Claim~3 below implies that this situation can arise
   only for small $n$. Hence, for sufficiently large $n$, the map $\sigma\colon \mathcal{S}_{n+1}(z)\to\mathcal{S}_n(f(z))$ will turn out to be a bijection after all (see Claim 4).
\begin{subproof}
  The first claim holds by choice of $N_1$. Part~\ref{item:Bshift} is immediate from the fact that
    $f\colon B_{n+1-N}(z)\to B_{n-N}(f(z))$ is a conformal isomorphism. Claim~\ref{item:Bnonempty} follows
    from these first two claims by induction. 

    Now let us prove~\ref{item:Bprefix}. 
    So suppose that $\s\in \mathcal{S}_{n+1}(z)$, and let $w\in \tail_{n+1}(\s)\cap B_{n+1-N}(z)$. 
     Recall from Lemma~\ref{lem:preimagedomains} that 
      $\tail_{n+1}(\s)\subset \unbdd{\tail}_{n+1}(\s) \cup V_{n}(\s)$.
     First consider the case where $w\in \unbdd{\tail}_{n+1}(\s)$. 
     Then $\s\in \mathcal{S}_{n}(z)$ since $w\in \unbdd{\tail}_{n+1}(\s)\cap B_{n+1-N}(z)\subset \tail_{n}(\s) \cap B_{n-N}(z)$. 

   Now suppose that $w\in V_{n}(\s)$. By definition, 
    \begin{equation}\label{eqn:nonemptyintersectionV}
          \zeta_n(\s) \in V_{n}(\s)\cap \tail_{n}(\s).\end{equation}
   Set $\tilde{\s} \defeq\sigma^{n-N}(\s)$, $\tilde{z}\defeq f^{n-N}(z)$ and 
    $\tilde{w}\defeq f^{n-N}(w)$.
    Then $\tilde{w}\in V_{N}(\tilde{\s})\cap B_1(\tilde{z})$. 
    By definition of $\eps$ and choice of $N$, 
    it follows that $V_N(\tilde{\s})\subset B_0(\tilde{z})$. 
    Now $B_{n-N}(z)$ is mapped conformally to  $B_0(\tilde{z})$ by $f^{n-N}$. Since
    $V_n(\s)$ is a connected component of $f^{N-n}(V_N(\tilde{\s}))$ and intersects $B_{n-N}(z)$, 
    we see that 
    $B_{n-N}(z)\supset V_{n}(\s)\ni \zeta_n(\s)$. By~\eqref{eqn:nonemptyintersectionV}, we see that 
    $\tail_{n}(\s)\cap B_{n-N}(z)\neq \emptyset$. This proves~\ref{item:Bprefix}).

    The final claim~\ref{item:Biteratedprefix} follows by induction from~\ref{item:Bshift} and~\ref{item:Bprefix}.
\end{subproof}

\begin{claim}[Claim 2]
    There is a finite
     collection $\tilde{\FF}\supset\FF$ of fundamental domains such that all entries
     of addresses in $\bigcup_{n\geq N} \mathcal{S}_n$ are in $\tilde{\FF}$. 
\end{claim}
\begin{subproof}
    Let $\tilde{\FF}$ be obtained by adding to $\FF$ all fundamental domains appearing in 
    the finitely many external addresses in $\mathcal{S}_N$. 
    By part~\ref{item:Biteratedprefix} of Claim 1, it follows that
    all entries of $\s\in \mathcal{S}_n$ belong to $\tilde{\FF}$, for all $n\geq N$.
\end{subproof}

  Applying Lemma~\ref{lem:preimagedomains} again, this time to
    the collection $\tilde{\FF}$, we obtain a simply connected domain $\tilde{V}$ that we can use to
    study the accumulation sets of the addresses in $\mathcal{S}_K$.

\begin{claim}[Claim~3]
  For every $k\geq 0$ there is an $n_0\geq N$ with the following property.
    If $z\in K$, $n\geq n_0$ and  $\s\in \mathcal{S}_n(z)$, then
    $\tilde{V}_n(\s)\subset B_k(z)$. 
\end{claim}
\begin{subproof}
 By assumption, $K\subset D$. In particular, there exists $M>0$ such that
     no fundamental tail of level $N$ intersects the neighbourhood $U_M$ of $K$. 
     Recall that $f(U_{j+1})=U_j$ for all $j$, and that
      the image of a fundamental tail of level $j+1$ is a tail of level $j$.  Hence it
      follows inductively that, if $n\geq N+M$,  no fundamental tail of level $n-M$ intersects
    $U_{n-N}$.

   Let $n\geq N+M$, let $z\in K$, and let $\s\in \mathcal{S}_n(z)$. Then 
         \[ \tau_n(\s) \subset \tau_{n-M}(\s) \cup \bigcup_{\tilde{n}=n-M}^{n-1} \tilde{V}_{\tilde{n}}(\s) \]  
     by 
     Lemma~\ref{lem:preimagedomains}~\ref{item:fundamentalunion}. 
     Observe that $\tau_n(\s)$ intersects $B_{n-N}(z)\subset U_{n-N}$ by definition of $\mathcal{S}_n(z)$,
      while
      $\tau_{n-M}(\s)$ is disjoint from $U_{n-N}$ by the above. Hence there exists some  $\tilde{n}\in \{n-M,\dots,n-1\}$ such that 
      $\tilde{V}_{\tilde{n}}(\s) \cap B_{n-N}(z)\neq \emptyset$.

    Let $k\geq 0$. If $n_1\geq N$ is sufficiently large, then by Lemma~\ref{lem:Euclideanshrinking}, 
       \[ \diam(\tilde{V}_n(\s)) < \eps_k \defeq \frac{\min_{z\in K} \dist(B_{k+1}(z),\partial B_k(z))}{M+1} \]
     whenever $n \geq n_1$, $\s\in \mathcal{S}_{n}(z)$ for some $z\in K$, and $\tilde{V}_n(\s)\cap U_0\neq \emptyset$.
    Set 
           \[ n_0\defeq \max(n_1 + M , k+N+1) \geq N+M \]
     and suppose that $z$, $n$ and $\s$ are as in the statement of Claim~3. 
 
    Let $\tilde{n}$ be as above; i.e., $\tilde{n}\geq n-M\geq n_1$ is such that
      $\tilde{V}_{\tilde{n}}(\s)\cap B_{n-N}(z)\neq \emptyset$. 
       Since $n-N\geq n_0 - N \geq k+1$, we see that $\tilde{V}_{\tilde{n}}(\s)$ intersects
      $B_{k+1}(z)$. 

    It follows inductively for $j=\tilde{n}, \tilde{n}+1, \dots , n$ that, for all
      $\zeta \in \tilde{V}_j(\s)$,
     \[ \dist(\zeta , B_{k+1}(z)) < (j+1-\tilde{n})\cdot \eps_k \leq \dist(B_{k+1}(z) , \partial B_k(z)). \]
     (In the inductive step, we use  that $\tilde{V}_{j}(\s)\cap \tilde{V}_{j-1}(\s)\neq\emptyset$ by Claim~2 and
       Lemma~\ref{lem:preimagedomains}~\ref{item:chain}, and that $j\geq n_1$.)
     In particular, $\tilde{V}_n(\s)\subset B_k(z)$, as required. 
\end{subproof}

\begin{claim}[Claim~4]
    Let $n_0$ be as in Claim~3, for $k=1$. Then $\sigma\colon \mathcal{S}_{n+1}(z) \to \mathcal{S}_n(f(z))$ is a bijection
     for all $n\geq n_0$ and $z\in K$.
  \end{claim}
   \begin{subproof}
   By Claim~1~\ref{item:Bshift}, it remains to show that $\sigma\colon \mathcal{S}_{n+1}(z)\to \mathcal{S}_n(f(z))$ is injective. 
       Let $z\in K$ and let $\s^1,\s^2\in \mathcal{S}_{n+1}(z)$ with $\sigma(\s^1) =\sigma(\s^2) \eqdef \tilde{\s}$. 
       Then, for $j=1,2$, we know that $\tilde{V}_{n+1}(\s^j)\subset B_1(z)$ by choice of $n_0$, and 
       $\tilde{V}_n(\tilde{\s}) = f(\tilde{V}_{n+1}(\s^j)) \subset B_0(f(z))$. Since $f\colon B_1(z)\to B_0(f(z))$ is univalent,
       it follows that $\tilde{V}_{n+1}(\s^1)=\tilde{V}_{n+1}(\s^2)$, and hence $\s^1=\s^2$, as required.
 \end{subproof}

  Now consider the directed graph $G$ whose vertices are the elements of 
           \[ \mathcal{V}(G) \defeq \bigcup_{n\geq N} \mathcal{S}_n, \]
    and which contains an edge from $\pi_n(\s)$ to $\s$ for every
     $\s \in \mathcal{S}_{n+1}$. Note that $G$ is a locally finite, infinite graph on countably many vertices.
     For $z\in K$, let $G_z$ be the induced subgraph of $G$ whose vertices are the elements of
     $\bigcup_{n\geq N} \mathcal{S}_n(z)$. 
     By~\ref{item:Bnonempty} and~\ref{item:Bprefix} of Claim~1, we can apply K\"onig's lemma, and 
     $G_z$ contains an infinite path for every $z\in K$. 

   Recall that $\mathcal{S}_K$ is the set of infinite external addresses $\s$ such that $\pi_n(\s)\in \mathcal{S}_n$ for all $n\geq N$.
     If $\s\in \mathcal{S}_K$, the sequence $(\pi_n(\s))_{n=N}^{\infty}$ forms 
     an infinite path in $G$. Conversely, every infinite path in $G$ determines an associated 
     address $\s\in\mathcal{S}_K$. For $z\in K$, denote by $\mathcal{S}_z$ the set of all $\s\in \mathcal{S}_K$ with 
     $\pi_n(z)\in \mathcal{S}_n(z)$ for all $n\geq N$. By the above,
     $\mathcal{S}_z\neq\emptyset$ for all $z\in K$. 

  The set $\mathcal{S}_K$ is shift-invariant by part~\ref{item:Bshift} of Claim~1.
   Furthermore, $\mathcal{S}_K$ is contained in the compact set of addresses all of whose entries are taken from $\tilde{\FF}$;
    we need to show that $\mathcal{S}_K$ is itself compact. 
    Suppose that $(\s^k)_{k=0}^{\infty}$ is a sequence of addresses in $\mathcal{S}_K$ converging to some address
    $\s$. Then the prefixes $\pi_{n+1}(\s)$ and $\pi_{n+1}(\s^k)$ agree for all sufficiently large $k$,
    and in particular $\pi_n(\s)$ and $\pi_{n+1}(\s)$ are two vertices of $G$ connected by an edge.
    It follows that $\s$ is indeed represented by an infinite path in $G$, and hence $\s\in \mathcal{S}_K$ as required. 

 To prove claim~\ref{item:filamentslanding} of Theorem~\ref{thm:hyperboliclanding_abstract}, we must show that 
   $G_{\s}$ lands at a point $z_0(\s)\in K$ for all $\s\in\mathcal{S}_K$. 
  By Claim~3, there is some $n_0$ such that, for all $n\geq n_0$ and
    all $\s\in \mathcal{S}_K$, there is $z\in K$ such that  $\tilde{V}_{n}(\s)\subset B_0(z)$. 
     In particular, $\diam \tilde{V}_{n_0}(\s) \leq 2\delta$, and by expansion of  $f$ on $U$, we conclude that 
      \begin{equation}\label{eqn:diameterestimate}
          \diam \tilde{V}_{n}(\s) \leq 2^{n_0-n+1} \cdot \delta. \end{equation}

In particular, $\zeta_n(\s)$ is a Cauchy sequence, and hence convergent. So $G_{\s}$ lands at 
       a point $z_0(\s)$ with
         \begin{equation}\label{eqn:landingdistance}
              \dist(z_0(\s) , \tilde{V}_n(\s)) \leq 2^{n_0-n+1}\cdot \delta, 
         \end{equation}
      for all $n\geq n_0$. It is clear from
      Claim~3 that the landing point $z_0(\s)$ belongs to $K$. 
  Furthermore, if $z\in K$ and $\s\in \mathcal{S}_z$, then $z_0(\s)=z$ by Claim~3.

 As
     $\mathcal{S}_z\neq\emptyset$ for all $z\in K$, this shows that the function
     $\s\mapsto z_0(\s)$ is surjective. 
     To prove continuity, suppose that $\s,\tilde{\s}\in \mathcal{S}_K$ agree in the first $n\geq n_0$ entries.
     Then $\tilde{V}_n(\s) = \tilde{V}_n(\tilde{\s})$, and hence
      \[ \dist( z_0(\s) , z_0(\tilde{\s})) \leq 3 \cdot 2^{n_0-n+1}\cdot \delta \]
      by~\eqref{eqn:diameterestimate} and~\eqref{eqn:landingdistance}. This completes the proof of~\ref{item:continuouslanding}.

    Part~\ref{item:uniformlanding} follows directly from Claim~3. It remains
       to establish~\ref{item:accessibleperiodic}.

    So let $z\in K$ be a periodic point of period $p$. By Lemma~\ref{lem:One for all}
    and~\ref{item:filamentslanding}, 
    it is enough to show that $\mathcal{S}_z$ contains a periodic address. 
    Let $n_0$ be as in Claim~4; by increasing $n_0$ if necessary, we can assume that $p$ divides $n_0$. 
    Let $\psi\colon \mathcal{S}_{n_0}(z) \to \mathcal{S}_{2n_0}(z)$ be the
     inverse of $\sigma^{n_0}|_{\mathcal{S}_{2n_0}(z)}$. 
     Since $\mathcal{S}_{n_0}(z)$ is finite, the function $\phi \defeq \pi_{n_0}\circ \psi$ has a periodic element;
      say $\phi^{k_0}(\s) =\s$. For $k\geq 0$, let $\s^k$ be the unique preimage of 
      $\s$ under $\sigma^{k n_0}$ in $\mathcal{S}_{(k+1)n_0}(z)$. We claim that 
        \[ \pi_{n_0}(\s^{k+1}) = \phi(\pi_{n_0}(\s^k)) \]
      for all $k\geq 0$. This is true for $k=0$ by definition. If $k>0$, we have 
        \[ \pi_{2n_0}(\s^{k+1}) = \psi(\sigma^{n_0}(\pi_{2n_0}(\s^{k+1}))) = 
            \psi( \pi_{n_0}(\sigma^{n_0}(\s^{k+1}))) = \psi(\pi_{n_0}(\s^k)). \]
       Hence 
      \[ \pi_{n_0}(\s^{k+1}) =\pi_{n_0} ( \psi(\pi_{n_0}(\s^k))) = \phi(\pi_{n_0}(\s^k)). \]

    So      $\pi_{n_0}(\s^k) = \phi^k(\s)$
     for all $k\geq 0$. Hence $\s^k$ can be written as a concatenation 
        \[ \s^k = \phi^k(\s) \s^{k_1}= \dots = \phi^k(\s) \phi^{k-1}(\s)  \dots \phi^1(\s) \s. \]
     Since $\s$ is periodic under $\phi$, of period $k_0$, we conclude that 
      $\s^k = \pi_{(k+1)n_0}(\s^{k+k_0})$.
     Hence there is an infinite path in $G_z$ passing through the vertices
      $\s^{j\cdot k_0}$, $j\geq 0$. The associated address is  the periodic sequence
               $(\phi^{k_0-1}(\s) \dots \phi^1(\s) \s)^{\infty}$,
       and the proof is complete.
\end{proof}

We note the following corollary, which proves the accessibility of certain singular values. For definitions,
   we refer to~\cite{RvS}.
 
 \begin{cor}[Accessibility of non-recurrent singular values] Let $f$ be a postsingularly bounded transcendental entire function, and let $v\in J(f)$ be a  non-recurrent   singular value for $f$ 
     whose forward orbit  does not pass through any critical points. Suppose that
     the $\omega$-limit set of $v$ does not contain parabolic 
    points, and does not intersect
     the $\omega$-limit set of a recurrent critical point or of a singular value contained in a wandering domain.
    Then there is a bounded-address filament of $f$ that lands at $v$.
 \end{cor}
 \begin{proof}
 By \cite[Theorem 1.2]{RvS}, if the   postsingular set is bounded then any forward invariant compact subset of the Julia set is hyperbolic provided it does not contain parabolic points, 
   critical points, or it intersects the $\omega$-limit set of a critical point or of a singular value contained in wandering domains. 
   Hence $P(a):=\ov{\bigcup_n f^n(a)}$ is hyperbolic and  every point in $P(a)$ is the landing point of a filament.
 \end{proof}
 
\section{Landing at parabolic points}\label{sec:parabolic}

 We now complete the proof of our analogue of the Douady-Hubbard landing theorem, Theorem~\ref{thm:DHfilaments}, by showing that parabolic periodic points are also accessible by filaments.

\begin{thm}[Parabolic points are accessible by filaments]\label{thm:parabolic}
  Let $f\in\BB$ with bounded postsingular set, and let $z_0$ be a parabolic periodic point. Then there is a periodic filament of $f$ that lands at $z_0$. 
\end{thm}

 Let $f$ be as in the statement of the theorem. By passing to an iterate, we may assume that $f'(z_0)=1$.
    So $z_0$ is a multiple fixed point of $f$, say of multiplicity $m+1$ for $f$. Then there are $m$ 
   unit vectors $v_1\ldots v_n$, called 
   \emph{repelling directions} at $z_0$. Any backward orbit of $f$ converging to $z_0$ must asymptotically 
    converge to $z_0$ along one of these directions; see \cite[Lemma 10.1]{Mi}. 
    Similarly, there are $n$ attracting directions $w_n$ such that any forward orbit $(f^n(z))_{n=0}^{\infty}$ converging to $z_0$ 
    must converge to $z_0$ along one of these attracting directions $w_n$.

Let $U$  be a small simply connected neighborhood of $z_0$ on which $f$ is univalent,  
   and let $\psi\colon f(U)\to U$ be the branch of  $f^{-1}$ that fixes $z_0$.
  A \emph{petal} for an attracting (resp. repelling) direction $w$  (resp. ${v}$)  is an open  set $P\subset U$ containing $z_0$ on its boundary, such that 
\begin{enumerate}
\item $f(P)\subset P$ (resp. $\psi(P)\subset P$ );
\item an orbit  $z\ra f(z)\ra\ldots$ (resp. $z\ra \psi(z)\ra\ldots$) is eventually absorbed by $P$ if and only if it converges to $z_0$ from the direction $w$ (resp. $v$).
\end{enumerate}
  Petals for a given repelling or attracting direction are far from unique. 
    For each repelling direction $v$, we can choose a  repelling petal $P_v$ for $v$ which is  simply connected,  and such that  
    $\overline{\psi(P_v)} \subset P_v\cup\{z_0\}$ and 
    $\psi^n|_{P_v}\to z_0$ uniformly on $P_v$.
     Similarly, for each attracting direction $w$ we choose a simply connected attracting petal $P_w$ such that
    $f^n\to z_0$ uniformly on $P_w$. We furthermore require 
    that the union of these $n$ attracting and $n$ repelling petals forms a punctured neighborhood of $z_0$ 
   (see Definition 10.6 and Theorem 10.7 in \cite{Mi} and the subsequent discussion).

 \begin{defn}[Landing of filaments along a repelling direction]
   Let $\zeta\in \WW$, let $G_{\s}$ be a periodic filament of $f$, and let $v$ be a repelling direction at $z_0$. 
   We say that $G_{\s}$ \emph{lands at $z_0$ along $v$} 
    if the backwards orbit $(\zeta_n(\s))_{n=1}^{\infty}$  converges to $z_0$ along the direction $v$. 
 \end{defn}
We remark that it is not difficult to see that this is equivalent to requiring that
    $V_{n}(\s)\subset P_v$ for all sufficiently large $n$, where $V$ is as in Lemma~\ref{lem:preimagedomains}. In particular, the definition is
     independent of the choice of the base point $\zeta$.

The following establishes Theorem~\ref{thm:parabolic}.
 
\begin{prop}[Accessibility along repelling directions]\label{Parabolic accessibility}\label{prop:parabolicprecise} 
    Let $v$ be a repelling direction of $f$ at $z_0$. Then there is at least one periodic filament landing at~$z_0$ along~$v$.
\end{prop}
\begin{proof}
Let $\zeta\in\WW$ and let $V$ be as in Lemma \ref{lem:preimagedomains}, with $\FF$ once again the finite collection of fundamental domains
    whose closure intersects $\overline{D}$. Since $\overline{V}\subset \C\setminus \P(f)$ and 
     $z_0\in\P(f)\subset D$, we may assume that the repelling petals $P_v$ and 
     attracting petals $P_w$ chosen above all have closures disjoint from $V\cup \WW$. 

Let us define $B_i\defeq \psi^i(P_v)$ for $i\geq 0$.
   Let $\AA$ be the union of the attracting petals $P_w$. Since the union of attracting and repelling petals is a punctured neighbourhood of $z_0$,
   all points of $\partial B_0$ that are sufficiently close to $z_0$ must lie in $\AA\cup\{z_0\}$. So
    $\partial B_0 \setminus (\AA\cup \{z_0\})$ is a compact set disjoint from $\overline{B_1}$, and 
     \[ \eps \defeq \dist\bigl( (\partial B_0)\setminus (\AA \cup \{z_0\}), \overline{B_1} \bigr) > 0.\]

 Since $B_1$ intersects $J(f)$, there is an $N_1$ such that $f^{-n}(\zeta)\cap  B_1\neq \emptyset$ for $n\geq N_1$. 
   In particular, 
    there exists some finite external address of length $n$ such that $\tau_n(\s)\cap B_1\neq \emptyset$. By Lemma~\ref{lem:Euclideanshrinking}, there is
    $N\geq N_1$ such that, for all $n\geq N$ and all infinite external addresses $\s$ with $V_n(\s)\cap \overline{B_0}\neq \emptyset$, 
    $\diam V_n(\s) < \eps$ whenever $n\geq N$. 
    Observe that $V_n(\s)\cap \AA=\emptyset$ by our choice of petals. In particular, if $n\geq N$ and
    $V_n(\s)\cap B_1\neq\emptyset$, then $V_n(\s)\cap \partial B_0 = \emptyset$, and hence
     $V_n(\s)\subset B_0$. 

  As in the proof of Theorem~\ref{thm:hyperboliclanding_abstract}, for $n\geq N$ we define $\mathcal{S}_n$ 
    to consist of those finite external addresses of length $n$ for which $\tau_n(\s)$ intersects 
    $B_{n-N}$. 
    The remainder of the proof then proceeds analogously.
\end{proof}
 
In the case that all periodic filaments are hairs (for example, if $f$ is criniferous), 
   our Proposition~\ref{Parabolic accessibility} is a corollary of the Main Theorem in \cite{BF15} (since the hypothesis  that periodic rays land  is implied by assuming bounded postsingular set), with a completely different proof.  
   We remark that it is plausible that the results of \cite{BF15} can also be extended to non-criniferous functions, using    filaments instead of hairs.

\begin{proof}[Proof of Theorem~\ref{thm:DHfilaments}]
  That every periodic filament lands at a repelling or parabolic point was proved in Theorem~\ref{thm:landingperiodic}. Let $z_0$ be a repelling or parabolic point. If $z_0$ is repelling,
   then the orbit of $z_0$ is a hyperbolic set, and it follows from Theorem~\ref{thm:hyperboliclanding_abstract} that $z_0$ is the landing point of a periodic filament.
   If $z_0$ is a parabolic point, then this fact follows from Theorem~\ref{thm:parabolic}. By Lemma~\ref{lem:One for all} there are only finitely many  filaments landing at $z_0$ and they  are all  periodic of the same period.
\end{proof}

\section{Filaments landing together at points in a hyperbolic set}\label{sec:classS}

Recall from Theorem~\ref{thm:DHfilaments} that, for a repelling periodic point $z_0$ of a postsingularly bounded function $f$, the number of filaments landing at $z_0$ is finite. 
  In the polynomial case, this holds also for every point $z_0$ in a hyperbolic set $K$ of $f$.   
   It is plausible that this remains true also in the transcendental entire case.
    For postsingularly bounded exponential maps, the claim is proved in  \cite[Proposition~4.5]{BL14}, where
    it is proved that the number of hairs in question is even \emph{uniformly} bounded
    (depending on $K$). However, the proof uses the fact that postsingularly bounded exponential maps
    are non-recurrent, and hence the postsingular set is itself a hyperbolic set.

Here we shall be content with proving that the number of filaments of a postsingularly bounded function
    $f$ landing
   at a given point of a hyperbolic set is (pointwise) finite, in the important special case where 
    $f$ belongs to the \emph{Speiser class}; i.e.,\ the set of singular values $S(f)$ is finite.

\begin{thm}[Finitely many filaments landing together]\label{thm:boundedaddresses}
  Let $f$ be a postsingularly bounded entire function with finitely many singular values.
    Suppose that 
     $z_0\in J(f)\setminus I(f)$ is neither a Cremer periodic point nor a preimage of such.
     Then the number of bounded-address filaments $G_{\s}$ landing at $z_0$ is finite.
\end{thm}
\begin{remark}[Remark~1]
   The assumption that
   $f$ is postsingularly bounded implies, via Theorem~\ref{thm:DHfilaments} and Lemma~\ref{lem:One for all}, that 
   one can restrict to the case where $z_0$ is not (pre-)periodic. In addition, this hypothesis and the
   restriction to bounded addresses $\s$ ensure that
   we can speak about the filaments $G_{\s}$ and their landing properties at all. (Recall Remark~\ref{rmk:nounboundedatperiodic}.) 
   However, the argument can be applied also in more general circumstances.
   For example, the same proof can be used to show the following: if $S(f)$ is finite 
   (but the postsingular set is not necessarily bounded), 
   and $z_0\in J(f)\setminus I(f)$ is not periodic and also
   is the landing point of at least one bounded-address hair, 
   then the number of hairs landing at $z_0$ is finite, and all of them have bounded
   addresses.
\end{remark}
\begin{remark}[Remark~2]
The assumption that $z_0\notin I(f)$ is made to avoid complications in the
  case where $z_0$ itself belongs to one of the filaments landing at $z_0$. An escaping point in
  a bounded-address filament cannot in fact be accessible by the same or another filament, due to the
  presence of other filaments accumulating on it from both sides; recall the proof of Corollary~\ref{cor:hairaccumulation}, and compare~\cite[Theorem~2.3]{arclike}.
  Assuming this fact, the assumption that $z_0\notin I(f)$ could be omitted.
\end{remark}

\begin{cor}[Finiteness of filament portraits at hyperbolic sets]
  Let $f$ be a postsingularly bounded entire function  with finitely many singular values. If $K$ is a hyperbolic set for $f$, then
    every point $z_0\in K$ is the landing point of at least one and at most finitely many bounded-address filaments.  
\end{cor}
\begin{proof}
   By definition, a hyperbolic set contains no Cremer periodic points, their preimages, or escaping points. Hence this is a combination of Theorems~\ref{thm:mainhyperbolic} and~\ref{thm:boundedaddresses}.
\end{proof}

 We now fix a postsingularly bounded entire function  $f$ with $\# S(f)<\infty$ for the remainder of the section.
The key property that we need to establish in the proof of Theorem~\ref{thm:boundedaddresses} 
  is that the addresses of filaments landing at $z_0$ are uniformly bounded, in the sense that they all take their entries from a common finite family of fundamental domains. This    is the content of the following lemma.
\begin{lem}\label{lem:compactaddresses} Let $\FF_1$ be a finite collection  of fundamental domains for $f$.  Then there exists another  finite collection $\FF_2\supset \FF_1$ of fundamental domains such that the following holds. Suppose that $ \s^1$ takes  only entries from $\FF_1$ and that $G_{\s^1}$ lands at a non-escaping point $z_0\in\C$. If $\s^2$ is bounded and $G_{\s^2}$ also lands at $z_0$, then 
    all entries of $\s^2$ belong to $\FF_2$. 
\end{lem}

Let us suppose for a moment that the function  $f$ is {criniferous}. Then the idea  of the proof of  
   Lemma~\ref{lem:compactaddresses} can be described as follows. If $G_{\s^1}$ and $G_{\s^2}$ 
   land together at a point $z_0$, the filaments $G_{\sigma(\s^1)}$ and $G_{\sigma(\s^2)}$ 
   also land together at $f(z_0)$, by continuity of $f$. There is a branch $\phi$ 
    of the inverse of $f$ on the hair $G_{\sigma(\s^1)}$ that maps it to $G_{\s^1}$. The 
     curve
      $G_{\s^1}\cup \{z_0\}\cup G_{\s^2}$
     is then obtained by analytic continuation of $\phi$ 
     along the image curve
     $G_{\sigma(\s^1)}\cup \{f(z_0)\} \cup G_{\sigma(\s^2)}$. 
     For this reason, the homotopy class of the latter curve in  $\C\setminus S(f)$, 
     together with the 
     first entry of $\s^1$, essentially determines 
      the first entry of $\s^2$.
      As different pairs of hairs
      landing at the same point are 
      disjoint, and $S(f)$ is finite, there are only
      finitely many possible such homotopy classes. The claim follows. 
  In order to make this argument precise in the general case, i.e.\ where the filaments are not necessarily hairs,
   we should clarify what we mean by ``homotopy classes''.
   Let us fix
   the postsingularly bounded function $f$ with finite singular set for the remainder of the section.
   
Let $\Gamma$ be the class of continuous curves $\gamma\colon \R\to \C\setminus S(f)$
   that tend to
   infinity within  $\WW$ in both directions. 
   We shall say that such curves 
    $\gamma_1$ and $\gamma_2$ are \emph{homotopic (in $\Gamma$)} if they 
     are homotopic (relative to their endpoints at infinity) 
     in $\C\setminus (S(f)\cup \tilde{\delta})$, for some infinite piece $\tilde{\delta}$ of the
    curve $\delta$ used in the definition of fundamental domains.

  Similarly, let $\tilde{\Gamma}$ denote the set of curves connecting a finite endpoint $z_0\in\C$ 
     (possibly belonging to $S(f)$) to infinity
    within  $\C\setminus S(f)$, again tending to infinity within  $\WW$. Then we 
    analogously define
    homotopy classes for curves in $\tilde{\Gamma}$ having the same endpoint. 

We can now introduce a convenient notion for homotopy classes
    of bounded-address filaments. Suppose that $\s$ is a bounded external address,
      and that the filament $G_{\s}$ lands at a point $z_0\in\C\setminus G_{\s}$. 
      Then there is an infinite piece $\tilde{\delta}$ of $\delta$ not intersecting $G_{\s}$. 
      It follows that there is a  Jordan curve $J$, passing through infinity, that 
      separates $\overline{G_{\s}}$ from $\tilde{\delta}$ and all of the finitely many points
      of $S(f)\setminus \{z_0\}$. Let $\gamma$ be an arc connecting $z_0$ 
      to infinity in the connected
      component $V$ of $\C\setminus J$ containing $z_0$. The 
      \emph{homotopy class of $G_{\s}$} is the homotopy class of
      $\gamma$ in $\tilde{\Gamma}$, as defined above.
     
  Note that this homotopy class depends only on $\s$. Indeed, suppose that 
   $\tilde{V}$ is a second domain as above, and $\tilde{\gamma}\subset \tilde{V}$ connects
   $z_0$ to infinity. Since $\overline{G_{\s}} \subset V \cap \tilde{V}\eqdef U$, 
   this open set $U$ contains a curve $\alpha$ connecting $z_0$ to infinity. 
    (See e.g.\ \cite[Lemma~A.1]{Re08}.) Since $V$ is simply connected,
     $\alpha$ is homotopic to $\gamma$ in $V$, and hence in 
    $\tilde{\Gamma}$. For the same reason, $\alpha$ is homotopic to $\tilde{\gamma}$.
    
\begin{obs}[Disjoint curves representing homotopy classes]
  Let $\mathcal{S}_1,\dots, \mathcal{S}_n$ be finitely many different bounded 
   external addresses, such that each $G_{\mathcal{S}_j}$ lands at a non-escaping point $z_j\in\C$
   for all $j$. (We do not assume that all $z_j$ are distinct.) 
   Then there exists a collection $(\gamma_j)_{j=1}^n$ of arcs to infinity, 
   with $\gamma_j$ in the homotopy class of $G_{\s}$, such that these
   arcs are pairwise disjoint apart from common endpoints.
\end{obs}     
\begin{proof}
  Similarly 
    as above, we can find a finite collection of Jordan curves $(J_{\ell})_{\ell=1}^{m}$,
    disjoint from $\tilde{\delta}\cup \bigcup_{j=1}^n \overline{G_{\mathcal{S}_j}}$, such that 
    any two distinct landing points $z_{j_1}$ and $z_{j_2}$ are separated by
    some $J_{\ell}$. (Here, as above, $\tilde{\delta}$ is an infinite piece of the curve
    $\delta$ that does not intersect any of the filaments under consideration.)
    
    Let $V_j$ be the connected component of $\C\setminus \bigcup_{\ell=1}^m J_{\ell}$
    containing $z_j$. We can choose the curve $\Gamma$ in the definition of the
    homotopy class of 
    $G_{\mathcal{S}_j}$ in such a way that $\Gamma$ additionally separates $z_j$ from
    $\partial V_j$. This shows that the $\gamma_j$ may be chosen disjoint, 
    except possibly for those having a common endpoint. But any curves with a 
    common endpoint belong to the same $V_j$, and therefore can also be moved by
    homotopy within the simply connected domain $V_j$ to be disjoint, except at that
    endpoint. This completes the proof.
\end{proof}

 If two bounded-address
  filaments $G_{\s^1}$ and $G_{\s^2}$ land at a common non-escaping
  point $z_0$,
  we shall refer to these two filaments 
  as a \emph{filament pair}. If $z_0\notin S(f)$, then we can form a curve in
   $\Gamma$ by combining two arcs $\gamma_1$ and $\gamma_2$, in the 
   homotopy class of $G_{\s^1}$ and $G_{\s^2}$, respectively. The 
   corresponding homotopy class is called the \emph{homotopy class of the filament
   pair}. 

\begin{lem}[Finitely many homotopy classes]\label{lem:finitehomotopy}
  There are only finitely many different homotopy classes of filament pairs 
    not landing at singular values. 

   Similarly, for any $z_0\in\C$, there are only finitely many homotopy classes of filaments landing at $z_0$. 
\end{lem}
\begin{proof}
  The curves representing the homotopy class of two different filament pairs 
    are disjoint, except for the endpoints at infinity, and possibly a single additional point
    (if the filament pairs land at the same point). Also recall that neither curve self-intersects.
    It follows that, if both curves wind around the same
    collection of singular values in positive orientation, and both either surround or do not surround an infinite 
     piece of $\delta$, they represent the same homotopy class. 
    As there are only finitely many singular values, the set of homotopy classes is finite.
 
   The second claim follows in the same manner. 
\end{proof}

The following is immediate from the homotopy lifiting property.
\begin{obs}[Connecting fundamental domains]\label{obs:homotopylifting}
  Let $\gamma\in \Gamma$. 
   Suppose that $F$ is a fundamental domain, and let $\tilde{\gamma}\colon (-\infty,\infty)\to\C\setminus f^{-1}(S(f))$
     be the unique lift of $\gamma$ under $f$ such that $\tilde{\gamma}(-t)\in F$ for  all sufficiently large $t$.
     Then there is a fundamental domain  $\tilde{F}$ such that $\tilde{\gamma}(t)\in\tilde{F}$ for large $t$, 
     and $\tilde{F}$ depends only on $F$ and the homotopy class of $\gamma$ 
      in $\Gamma$. 

   Similarly, let $\gamma\in\tilde{\Gamma}$ connect a finite point $z_0\in\C$ to $\infty$.
    If $F$ is a fundamental domain, and 
    $\tilde{\gamma}$ is the lift of $\gamma$ under $f$ that tends to infinity within $F$, 
    then the finite endpoint $w_0$ of $\tilde{\gamma}$ depends only on $F$ and the 
    homotopy class of  $\gamma$ in $\tilde{\Gamma}$. 
\end{obs}

\begin{proof}[Proof of Lemma~\ref{lem:compactaddresses}]
Let $\tilde{\FF}_2$ consist of all domains $\tilde{F}$ as in Observation~\ref{obs:homotopylifting},
    where $F$ ranges over the finitely many elements of $\FF_1$, and the homotopy class of  $\gamma$ ranges over
    the finitely many homotopy classes of filament pairs of $f$. 

   Now suppose that $G_{\s^1}$ and $G_{\s^2}$ form a filament pair, with $F^1_0\in \FF_1$.
     Let $z_0$ be the common landing point of the two filaments. If 
      $f(z_0)\notin S(f)$, then it follows from Observation \ref{obs:homotopylifting} (applied to 
      the curve $\gamma_{\sigma(\s^1)}\cup \{f(z_0)\} \cup \gamma_{\sigma(\s^2)}$) that
       $F^2_0\in \tilde{\FF}_2$. 

    On the other hand, suppose that $s=f(z_0)\in S(f)$. Then, by 
      Observation~\ref{obs:homotopylifting}, $z_0$ depends only on the homotopy class of $\gamma_{\sigma(\s^1)}$, 
       and the entry $F^1_0$. Hence, for each singular value $s$, there are only finitely many possible preimages
       $z_0$ that can arise as landing points of filaments whose first entry is in $\FF_1$.

     Consider such $z_0$, and the curve $\gamma = \gamma_{\sigma(\s^2)}\in\tilde{\Gamma}$ connecting $f(z_0)$ to $\infty$.
      Then $\gamma$ has $d$ different lifts starting at $z_0$, where $d$ is the local degree of $f$ at $z_0$,
      tending to infinity within fundamental domains $\tilde{F}_1,\dots,\tilde{F}_d$. This collection 
      of fundamental domains
      depends only on the homotopy class of $\gamma$ by Observation~\ref{obs:homotopylifting}. 
      In particular, there is a collection $\FF(z_0)$  of at most $m\cdot d$ fundamental domains,
       where $m$ is the (finite) number of homotopy classes of filaments connecting $s$ to $\infty$,
       such that $F^2_0\in \FF(z_0)$ whenever $\s^2$ is as above. 

     Recall that there are only finitely many singular values $s$, and for each of these only finitely many
       preimages $z_0$ as above. Thus we can add the finitely many sets $\FF(z_0)$ to $\tilde{\FF}_2$ to
       obtain a set $\FF_2$ with the desired property.
\end{proof}

\begin{proof}[Proof of Theorem~\ref{thm:boundedaddresses}]
  If $z_0$ is (pre-)periodic, then by assumption $f^n(z_0)$ is a repelling or 
  parabolic periodic point for some $n\geq 0$. As remarked above, in this case
   the conclusion of the theorem holds by 
     Theorem~\ref{thm:DHfilaments} and Lemma~\ref{lem:One for all}. Hence we can
      assume that $z_0$ is not a pre-periodic point.

  Since $f$ is postsingularly bounded, every orbit of $f$ passes through only finitely many critical points. Indeed, points with unbounded orbits cannot go through critical points at all,  and the intersection of any bounded orbit with the
  (discrete) set of critical points is finite. Hence, passing to a forward iterate, we may additionally assume that the forward orbit of $z_0$ does not 
    contain a critical point.  
   Let $\FF_1$ be the set of fundamental domains occurring in $\s$, and let $\FF_2$ be the 
   set whose existence is guaranteed by Lemma~\ref{lem:compactaddresses};
    say  $\FF_2 = \{ F^0, F^1 , \dots , F^{m-1}\}$, where we assume that
     \[ F^0 \prec F^1 \prec \dots \prec F^{m-1} \prec F^1 \]
     with respect to the cyclical order at infinity. 
   
  Let $X$ be the set of points on the unit circle $S^1 = \R/\Z$ having an $(m+1)$-ary expansion that
    contains only the entries $0,\dots,m-1$. Via the $(m+1)$-ary expansion, this set
    is order-isomorphic to $\{0,\dots,m-1\}^\N$, which in turn is clearly 
    order-isomorphic to $\FF_2^{\N}$. Let $\phi\colon \FF_2^{\N}\to X$ be this order-isomorphism;
    then  $\phi$ conjugates the shift on  $\FF_2^{\N}$ to the $(m+1)$-tupling map on $X$.  

  Suppose that $T_0$ is a collection of $p\geq 1$ bounded external addresses that land at $z_0$; 
    we claim that $p\leq m+1$. Indeed, for $j\geq 0$, define $T_j\defeq \sigma^j(T_0)$. Then
    all filaments at addresses in $T_j$ land at $f^j(z_0)$. Since $z_0$ is not pre-periodic and 
    its orbit does not pass through any critical points, the $T_j$ are pairwise disjoint, and
    $\sigma\colon T_j\to T_{j+1}$ is an order-preserving bijection for all $j$. Furthermore,
    the $T_j$ are pairwise \emph{unlinked}. That is, if $j\neq \tilde{j}$, then all elements of
    $T_j$ lie between the same two adjacent elements of $T_{\tilde{j}}$ with respect to circular order. 

   This means that the set $\phi(T_0)$ is a \emph{wandering $p$-gon} for the $(m+1)$-tupling map
     on $S^1$. Kiwi \cite[Theorem 1.1]{Kiwi02}
        proved that polynomials of degree $d$ do not have wandering $(d+1)$-gons. 
        A combinatorial version of this result (see 
         \cite[Theorem~B]{blokhlevin02}) implies that $p\leq m+1$ as claimed.
\end{proof}
\begin{remark} 
  It seems likely that one can also directly prove the absence of wandering $d+2$-gons for maps 
    with at most $d$ singular values. (Compare \cite{alhabib-rempe15} for the proof of the case  $d=1$, i.e.\
    the \emph{no wandering triangles theorem} for exponential maps.) This would imply that the number
    of filaments in Theorem~\ref{thm:boundedaddresses} is always bounded by $d+1$ (assuming that $z_0$ is
     not pre-critical). 
\end{remark}

\section{Appendix: Cyclic order of unbounded closed connected sets}\label{sec:cyclicorder}
  In this section, suppose
    that $\mathcal{A}$ is any pairwise disjoint collection of unbounded, closed, connected subsets of $\C$ such that, for every $A\in\mathcal{A}$, all
  elements of $\mathcal{A}\setminus\{A\}$ belong to the same connected component 
  of $\C\setminus A$. Observe that the latter condition holds, in particular,
  if no $A\in\mathcal{A}$ separates the plane. 
  
   The purpose of this section is to note that
     there is a natural \emph{cyclic order} (at $\infty$) on $A$. 
      Recall that a cyclic order is a ternary relation 
      $A\prec B \prec C$ that is cyclic, asymmetric, transitive and total \cite[\S~5]{cechpointsets}.

    In our case, the relation $A\prec B \prec C$ means that $B$ lies between  
     $A$ and $C$ in positive orientation.
     To make this precise, let us begin by defining a circular order on any finite subset of $\mathcal{A}$. So suppose that $A_1,\dots , A_n$ ($n\geq 3$) are distinct elements of 
     $\mathcal{A}$. Let $W_j$ be the connected component of $\C\setminus A_j$ that contains $A_i$ for $i\neq j$, and set
     $\tilde{A}_j \defeq \Ch\setminus W_j$. Then $K\defeq \bigcup_{j=1}^{n} \tilde{A}_j$ is a compact, connected and full set in $\Ch$, and its complement is
      \[ W \defeq \Ch\setminus K = \bigcap_{j=1}^{\infty} W_j. \] 
      In other words, the simply connected domain $W$ is the unique connected component 
     $W$ of $\C\setminus \bigcup_{j=1}^n A_j$ whose boundary intersects $A_j$ for each $j$. 

  We now consider the space of \emph{prime ends} of $W$; see \cite[Section~2.4]{pommerenke}. Recall that these form 
    a topological circle, and therefore possess a natural cyclic order. Note that the connected components of 
    $K\setminus\{\infty\}$ are precisely the $\tilde{A}_j\setminus\{\infty\}$. It follows 
    (e.g.\ as a consequence of the plane separation theorem \cite[Theorem~3.1, Chapter~VI]{whyb}) that there are exactly
   $n$ different accesses $\zeta_1,\dots,\zeta_n$ to $\infty$ from $W$. They separate the circle of prime ends into $n$
   complementary intervals $I_1,\dots,I_n$, which may be labeled such that $I_j$ consists of those prime ends that can be represented by
   a sequence of cross-cuts both of whose endpoints belong to $A_j$. We define the circular order of the sets $A_j$ at $\infty$ (in positive orientation) to be 
   the circular order of these intervals, taken in \emph{negative} orientation.

 If we add a new element $A_{n+1}$ of $\mathcal{A}$ to our collection, then it is easy to check that this does not change the definition of the
   circular order of $A_1,\dots,A_n$. Hence we do indeed obtain a well-defined circular order on all of $\mathcal{A}$.
 Moreover, suppose that 
     $\tilde{\mathcal{A}}$ is a second collection  as above, where every element of $\tilde{\mathcal{A}}$ is
      contained in an element of $\mathcal{A}$ and every element of $\mathcal{A}$ contains
     exactly one element of $\tilde{\mathcal{A}}$. Then the cyclic order on $\tilde{\mathcal{A}}$ coincides with
     the corresponding order on  $\mathcal{A}$.

   We can use this observation to define cyclic order also for pairwise disjoint collections of \emph{open} unbounded
    domains, each of which contains exactly one homotopy class of curves to infinity. (Simply replace each
    domain by a representative in the mentioned homotopy class.) 

   Furthermore, suppose that
    $U$ and $\tilde{U}$ are unbounded domains in $\C$, that $\phi\colon U\to \tilde{U}$ is a conformal isomorphism.
    Also suppose that $\mathcal{A}$ and $\tilde{\mathcal{A}}$ are collections as above, whose elements are contained 
    in $U$ and $\tilde{U}$, respectively, that $\phi$ maps every element of $\mathcal{A}$ to an 
    element of $\tilde{\mathcal{A}}$, and that all elements of $\tilde{\mathcal{A}}$ arise in this manner.
    Then the action of $\phi$ on $\mathcal{A}$ preserves cyclic order. 

  Finally, let $\mathcal{A}$ be a pairwise disjoint collection of closed, connected sets in $\C^*=\C\setminus\{0\}$,
    and that the closure of each element of $\mathcal{A}$ contains both $0$ and $\infty$. Then we can define the cyclic order at $\infty$ 
    on $\mathcal{A}$, by replacing each element of $\mathcal{A}$ by an unbounded connected subset that is closed in $\C$, and applying the above definition.
   Analogously, we can define a cyclic order on $\mathcal{A}$ at $0$. It is easy to see (again using the plane separation theorem) that both orders coincide, and depend 
   only on $\mathcal{A}$ rather than any choices made in the construction.

 \begin{remark}
   There are some subtleties to the definition of circular order on connected sets, compared with the case of arcs to infinity which has 
     been previously considered in the complex dynamics literature. For example, note that the assumption that the
     sets in $\mathcal{A}$ are closed is crucial. Indeed, consider the case of a Knaster bucket-handle continuum $X$, whose 
     terminal point (that is, the initial point of the half-ray running through all of the endpoints of the complementary intervals
      of the ternary Cantor set) has been placed at $\infty$, and consider the collection of path-connected components of
      this set. Every such component is unbounded and connected, but since each component accumulates everywhere upon $X$,
      there is no sensible circular order among them.
 \end{remark}

\section{Appendix: Unbounded postsingular sets}\label{sec:unboundedpostsingularset}
  As mentioned in the introduction, the Douady-Hubbard landing theorem no longer holds
    for polynomials with escaping singular values. It is still true that every repelling (or parabolic) periodic point
    is accessible from the basin of infinity, and even by a dynamic ray, if we extend this notion appropriately to the case where
    the ray passes through critical points; compare \cite{ELv,LP}). However, it is possible for the set of landing
    rays to be uncountable, and for none of these rays
    to be periodic; compare \cite[Appendix~C]{goldbergmilnor} and \cite{LP}.

Let us now briefly discuss the case of transcendental entire functions $f$ with unbounded
    postsingular set $\P(f)$.
    When $f\notin \BB$, the structure of the escaping set may change dramatically within a given 
   parameter space (compare \cite[Appendix~B]{rempesixsmith}), and hence it is not clear whether questions
   concerning the landing of rays or filaments are even meaningful in this setting. 
  Let us hence restrict to the case of $f\in\BB$.

  First suppose that $f$ has an escaping singular value. In addition to the above-mentioned 
    behaviour that occurs already for polynomials,  it is also possible for a repelling periodic
    point to not be accessible from the escaping set at all (by hairs or filaments). Indeed, this is the case
    for the fixed point of the exponential map $z\mapsto e^z$ having imaginary part between $0$ and $\pi$, and shows
   that the question of landing behaviour at periodic points becomes considerably more subtle when $\P(f)$ is unbounded.

 However, consider now the full family of exponential maps, $f_a\colon z\mapsto e^z+a$. Suppose that
    the singular value $a$ has an unbounded orbit but
     does not belong to the escaping set. Then $f_a$ is
     criniferous. In \cite{Re06}, it is shown that that all periodic hairs of $f_a$ land. 
     Conversely, every periodic point, with the exception of 
     at most one
     periodic orbit, is the landing point of a periodic hair.  The exceptional orbit cannot be parabolic, but 
      it is an open question whether it can be repelling. It is shown in \cite{Re06} that a 
      plausible conjecture about parameter space of exponential maps
      (the ``no ghost limbs conjecture'') would imply that this is not the case.

 Hence it is plausible that the Douady-Hubbard landing theorem remains valid for 
       exponential maps as above, which raises the question whether the main theorem of our paper
       may also have an extension for functions $f\in\B$ with unbounded but non-escaping singular orbits. 
       A crucial step is to ensure the landing of periodic rays (or filaments). Indeed, if periodic rays land and the function has good geometry in the sense of \cite{R3S}, one can show that the number of rationally invisible repelling periodic orbits is bounded by the number of free singular values \cite{BF20}, just 
         as for the exponential family. Unfortunately,  the proofs in \cite{Re06} that periodic rays land use sophisticated results on the structure of 
        the (one-dimensional) parameter
     space of exponential maps, and it appears that 
     fundamentally new approaches would be required to resolve this question in full generality. 
     
\newcommand{\etalchar}[1]{$^{#1}$}
\providecommand{\bysame}{\leavevmode\hbox to3em{\hrulefill}\thinspace}
\providecommand{\href}[2]{#2}

\end{document}